\documentclass[11pt]{amsart}

\usepackage[utf8]{inputenc}
\usepackage[T1]{fontenc}
\usepackage{lmodern}         
\usepackage{amsmath, amsthm, amssymb, amsfonts,quiver}
\usepackage{mathtools}        
\usepackage{mathrsfs}         
\usepackage{hyperref}        
\usepackage{enumitem}         
\usepackage{tikz-cd}        
\usepackage{microtype}       
\usepackage{cleveref}    
\usepackage{stmaryrd}          

\usepackage[toc,page]{appendix}

\usetikzlibrary{cd, calc}

\tikzset{
  snake left/.style={
    rounded corners,
    to path={
      let \p1 = (\tikztostart.east),
          \p2 = (\tikztotarget.west),
          \p3 = ($(\p1)!0.5!(\p2)$),
          \n1 = {8pt} 
      in
      (\p1)
      -- (\x1 + \n1, \y1)
      -- (\x1 + \n1, \y3)
      -- (\x2 - \n1, \y3) \tikztonodes
      -- (\x2 - \n1, \y2)
      -- (\p2)
    }
  }
}

\usepackage[margin=1in]{geometry}
\usepackage[dvipsnames]{xcolor}

\definecolor{burgundy}{RGB}{128, 0, 32}
\definecolor{teal}{RGB}{0, 128, 128}

\theoremstyle{plain}
\newtheorem{theorem}{Theorem}[section]
\newtheorem{lem}[theorem]{Lemma}
\newtheorem{prop}[theorem]{Proposition}
\newtheorem{cor}[theorem]{Corollary}
\newtheorem{claim}[theorem]{Claim}

\theoremstyle{definition}
\newtheorem{definition}[theorem]{Definition}
\newtheorem{example}[theorem]{Example}

\theoremstyle{remark}
\newtheorem{remark}[theorem]{Remark}
\newtheorem{notation}[theorem]{Notation}

\makeatletter
\newcommand{\xtwoheadrightarrow}[2][]{\ext@arrow 0579\twoheadrightarrowfill@{#1}{#2}}
\def\twoheadrightarrowfill@{\arrowfill@\relbar\relbar\twoheadrightarrow}
\makeatother

\newcommand{\Q}{\mathbb{Q}}
\newcommand{\Z}{\mathbb{Z}}
\newcommand{\F}{\mathbb{F}}
\newcommand{\Qp}{\Q_p}

\newcommand{\Zp}{\Z_p}

\newcommand{\Fp}{\F_p}

\newcommand{\OF}{\calO_F}

\newcommand{\KUbar}{K_{\overline{U}}}

\newcommand{\SVi}{\mathcal{S}_V^i}

\newcommand{\Ski}{\mathcal{S}_k^i}

\newcommand{\id}{\operatorname{id}}
\newcommand{\Id}{\operatorname{Id}}

\newcommand{\conjtI}{\conj_{t_I^{-1}}}
\newcommand{\longtwoheadrightarrow}{\relbar\joinrel\twoheadrightarrow}
\newcommand{\longhookrightarrow}{\lhook\joinrel\longrightarrow}
\newcommand{\Matn}{\operatorname{Mat}_n}
\newcommand{\KUI}{K_{U,I}}

\newcommand{\HT}{\mathcal{H}_T}

\newcommand{\valpi}{\mathrm{val}_{\varpi}}

 \newcommand{\HH}{\mathcal{H}}
\newcommand{\Prx}{\operatorname{Pr}_X}
\newcommand{\Pry}{\operatorname{Pr}_Y}

\newcommand{\diag}{\operatorname{diag}}

\newcommand{\calO}{\mathcal{O}}

\newcommand{\Tr}{\operatorname{Tr}}

\newcommand{\BBB}{\mathbf{B}}

\newcommand{\Hom}{\operatorname{Hom}}

\newcommand{\End}{\operatorname{End}}
\newcommand{\GL}{\operatorname{GL}}
\newcommand{\Sym}{\operatorname{Sym}}
\newcommand{\Supp}{\operatorname{Supp}}

\newcommand{\res}{\mathrm{res}}

\newcommand{\loongmapsto}{\mathrel{\tikz[baseline]{
  \draw[->] (0.3,0) -- (1.8,0);
  \draw[-|] (0.3,0) -- (0,0);
}}}
\newcommand{\RInd}{\mathrm{RInd}}
\newcommand{\Inf}{\mathrm{Inf}}
\newcommand{\RInf}{\mathrm{RInf}}
\newcommand{\ind}{\mathrm{ind}}
\newcommand{\Ind}{\mathrm{Ind}}
\newcommand{\Res}{\mathrm{Res}}
\newcommand{\Rep}{\mathrm{Rep}}
\newcommand{\Ab}{\mathrm{Ab}}
\newcommand{\cores}{\mathrm{cores}}
\newcommand{\conj}{\mathrm{conj}}

\newcommand{\pre}[2]{\prescript{#1}{}{#2}}
\newcommand{\iso}{\cong}
\newcommand{\inv}{^{-1}}
\newcommand{\into}{\hookrightarrow}
\newcommand{\onto}{\twoheadrightarrow}
\newcommand{\pr}{\mathrm{pr}}
\newcommand{\mc}{\mathcal}
\newcommand{\ol}{\overline}

\newcommand{\ul}{\underline}
\newcommand{\xto}{\xrightarrow}
\newcommand{\supp}{\mathrm{supp}}

\def\bal#1\nal{\begin{align*}#1\end{align*}}

\newcommand{\scrJ}{\mathscr{I}}
\newcommand{\scrJi}{\mathcal{P}_i(\scrJ)}

\newcommand{\AV}{\mathbf{A}_V}
\newcommand{\AVi}{\mathbf{A}_{V,i}}
\newcommand{\BV}{\mathbf{B}_V}
\newcommand{\BVi}{\mathbf{B}_{V,i}}
\newcommand{\Bk}{\mathbf{B}_k}
\newcommand{\Bki}{\mathbf{B}_{k,i}}



\def\BAL#1\NAL{\[
\resizebox{\textwidth}{!}{$#1$}
\]}

\makeatletter
\renewcommand{\paragraph}{\@startsection{paragraph}{4}{\z@}%
  {3.25ex \@plus1ex \@minus.2ex}%
  {-1em}%
  {\normalfont\normalsize\bfseries}}
\makeatother

\numberwithin{equation}{section}

\hypersetup{
    colorlinks=true,
    linkcolor=burgundy,   
    citecolor=NavyBlue,       
    urlcolor=NavyBlue,        
    pdftitle={Higher mod p Satake transforms},
    pdfauthor={Hargun Bhatia, Yashi Jain, Casey Perdue, Ben Savoie},
}

\begin{document}
 
\title[Higher mod $p$ Satake transforms: explicit formulas and images for $\operatorname{GL}_2$]{Higher mod $p$ Satake transforms: \\explicit formulas and images for $\operatorname{GL}_2$}
\author{Hargun Bhatia, Yashi Jain, 
Casey Perdue, Ben Savoie}
\date{}

\vspace{-.75\baselineskip}
\begin{abstract}
Heyer's left adjoint to derived parabolic induction gives rise to 
higher mod $p$ Satake transforms
\[
\mathcal S_V^i: \mathcal H_G(V) \longrightarrow \mathcal H_M\bigl(H^i(K_U,V)\bigr),
\]
whose extreme degrees recover Herzig's invariant and coinvariant transforms. For $G=\mathbf G(F)$, where $\mathbf G$ is a connected reductive group defined over a finite extension $F/\Qp$, we describe these maps explicitly using Mackey decompositions and corestriction in group cohomology. For $G=\GL_n(F)$, we evaluate them on the standard Hecke generators, showing that their support lies in a single Weyl orbit and that their values on the individual support cosets are concrete corestriction operators.  

For $G=\GL_2(F)$, we determine the image in every cohomological degree for
each non-Steinberg Serre weight. The endpoint images are the classical
polynomial Satake images, while every interior image is of the form
\[
k[x,y^{\pm1}]/(x^r)
\]
for some positive integer $r$. Thus, every interior image is a finite
nilpotent thickening of the Laurent polynomial algebra $k[y^{\pm1}]$. We give explicit bounds for its thickness, prove that they are sharp, and
show that in the unramified case every interior image is either reduced or
a square-zero thickening.
\end{abstract}

\maketitle
\vspace{-1.5\baselineskip}
\tableofcontents

\section{Introduction}
The Satake isomorphism is a fundamental tool in the representation theory
of reductive groups over local fields, relating the spherical Hecke algebra
of a reductive group to a Hecke algebra for its maximal torus \cite{satake1963}. Let
$F/\Qp$ be a finite extension and $\mathbf G$ be a connected unramified
reductive group over $F$, and set $G=\mathbf G(F)$. Fix a hyperspecial
maximal compact subgroup $K\subset G$ and a Borel subgroup $B=TU$, with
maximal torus $T$ and unipotent radical $U$. The spherical Hecke algebra
\bal
\mathcal H(G,K)=C_c^\infty(K\backslash G/K,\mathbb C)
\nal
acts naturally on the space $\pi^K$ of $K$-invariants of a smooth
representation $\pi$ of $G$.

To study its eigenvalues, one uses the classical Satake transform
\bal
\mathcal S^{\mathrm{cl}}(\varphi)(t) =
\delta_B(t)^{-1/2}
\sum_{u\in (U\cap K)\backslash U}\varphi(ut),
\nal
where $\delta_B$ denotes the modulus character of $B$. The sum over $U$ is the constant term of $\varphi$, while the factor $\delta_B^{-1/2}$ gives the normalized Satake transform. The Satake isomorphism identifies
\bal
\mathcal H(G,K)
\overset{\sim}{\longrightarrow}
\mathcal H(T,T\cap K)^W
\nal
thereby reducing the study of the spherical Hecke algebra of the noncommutative group $G$ to the Weyl-invariant part of the Hecke algebra of the torus $T$.
Denoting by $\mathbf 1$ the trivial representation of $K$, one has a natural identification
\bal
C_c^\infty(K\backslash G/K,\mathbb C)
\cong
\operatorname{End}_G\left(\ind_K^G \mathbf 1\right),
\nal
where $\ind$ denotes the compact induction.

Now let $k$ be a coefficient field of characteristic $p$. In this setting,
the trivial coefficient representation of $K$ is naturally replaced by a
Serre weight $V$ of $K$.\footnote{By a Serre weight of $K$, we mean an
irreducible $k$-representation of $K$.} The classical normalizing factor is
no longer available, and the resulting transform is no longer
Weyl-invariant. The spherical Hecke algebra attached to $V$ is
\bal
\mathcal H_G(V)\coloneqq \End_G(\ind_K^G V).
\nal
Set $K_U\coloneqq K\cap U$. In \cite{herzig2011satake}, Herzig constructed an
injective algebra homomorphism
\bal
\mathcal H_G(V)
\longhookrightarrow
\mathcal H_T(V_{K_U}),
\nal
obtained by summing along $U$ and projecting from $V$ to its
$K_U$-coinvariants. Its image is supported on the dominant cone. The mod $p$ Satake isomorphism was subsequently extended by
Henniart--Vign\'eras to connected reductive groups over general local
fields and special parahoric subgroups
\cite{henniartvigneras2015}.

Let $\overline U$ denote the unipotent radical of the Borel subgroup
opposite to $B$, and set $\KUbar\coloneqq K\cap\overline U$.
Herzig also constructed an invariant form of the mod $p$ Satake transform,
with coefficient representation $V^{K_U}$ and support on the opposite cone.
The natural composite
\bal
V^{K_U}\longrightarrow V\longrightarrow V_{K_{\overline U}}
\nal
is an isomorphism of $K_T\coloneqq K\cap T$-representations. Thus, the
invariant and coinvariant transforms have canonically isomorphic coefficient
representations but are supported on opposite cones. 


For $G=\GL_n(F)$, Herzig's Satake transform also plays an important role
in mod $p$ representation theory: its compatibility with parabolic induction
relates Hecke eigensystems arising from proper parabolic induction to those
for the corresponding Levi subgroup
\cite[Lemma~2.14]{herzig2011classification}. In particular, this
compatibility is a basic ingredient in the distinction between parabolically
induced and supersingular eigensystems in the classification of irreducible
admissible mod $p$ representations of $\GL_n(F)$
\cite{herzig2011classification}. 

These two forms of the mod $p$ Satake transform fit into a broader family
arising from Heyer's construction of a left adjoint to derived parabolic
induction \cite{heyer2023left}. Let $P=UM$ be a parabolic subgroup and set
\bal
K_P\coloneqq K\cap P,
\qquad
K_U\coloneqq K\cap U,
\qquad
K_M\coloneqq K\cap M.
\nal
Let $V$ be a smooth finite-dimensional $k$-representation of $K$, not
necessarily irreducible. In this setting, this left adjoint gives
rise to algebra homomorphisms
\bal
\SVi:
\mathcal H_G(V)
\longrightarrow
\mathcal H_M\bigl(H^i(K_U,V)\bigr),
\qquad
0\leq i\leq\dim K_U,
\nal
which we call \emph{higher mod $p$ Satake transforms}. At the two extreme
degrees, these recover Herzig's invariant and coinvariant transforms, while
the transforms in the intermediate degrees $0<i<\dim K_U$ are far less
transparent. The central problem of this paper is to make these intermediate transforms explicit and determine their supports and images. We first derive double-coset formulas for a general reductive group in terms
of corestriction in group cohomology. We then specialize to $\GL_n(F)$ and
evaluate these formulas on the standard Hecke generators.

For $\GL_2(F)$, this problem admits a complete answer for every
non-Steinberg Serre weight: we determine the image in every cohomological
degree. The two endpoint images are supported on the
dominant and antidominant cones. In every interior degree, the dominant and
antidominant contributions annihilate one another and are individually
truncated by nilpotence. Consequently, the support collapses from two infinite cones to a strip of finite width along the diagonal; see Figure~\ref{fig:gl2-support}. Algebraically, the image is a finite nilpotent thickening of the Laurent polynomial algebra
generated by diagonal translation. This nilpotent thickening of the
diagonal is the central new phenomenon of the paper.

\subsection{Main results} 
We begin with a cohomological analogue of the usual double-coset formula
for the Satake transform. The hypotheses for the general results below are
specified in Subsection~\ref{subsec-notation-assumptions}, while the notation
appearing below is developed in Section~\ref{expl-formulas-sec}.

\begin{theorem}[Double-coset and cochain formulas]
\label{thm-intro-general-formula}
Let $G$, $P=UM$, and $K$ be as above, and let $V$ be a smooth
$k$-representation of $K$. For
$\Phi\in\mathcal H_G(V)$ and $\xi\in H^i(K_U,V)$, the Mackey
decomposition in cohomology gives
\bal
\SVi(\Phi)([1,\xi])
= \sum_{g\in K_P\backslash P/K_U}
\left[
\pr_M(g^{-1}),
\cores_{g^{-1}}^{(i)}
\left(
\pr_{g^{-1}}^{(i)}(\Phi([1,\xi]))
\right)
\right].
\nal
Here $\operatorname{pr}_{g^{-1}}^{(i)}$ is projection to the Mackey summand
indexed by $g^{-1}$, and $\cores_{g^{-1}}^{(i)}$ is the corresponding
corestriction map to $H^i(K_U,V)$.

Moreover, the transform admits an explicit cochain-level description in
the function model for the Hecke algebras. For every $m\in M$, there is a map
\bal
C_m^{(i)}:
H^i\left(K_U,\ind_{K_P}^P V\right)
\longrightarrow
H^i(K_U,V)
\nal
such that, if $\Psi$ denotes the function corresponding to $\Phi$, then
\bal
\SVi(\Psi)(m)(\xi) = C_m^{(i)}\bigl(\Psi(-)(\xi)\bigr).
\nal
\end{theorem}

The first formula expresses the transform in terms of Mackey projection and
corestriction, while the maps $C_m^{(i)}$ give an explicit realization of
the same procedure at the level of cocycles; see
Definition~\ref{def-delta-functor}.

For $\GL_n(F)$, the general formula becomes particularly simple on the
standard Hecke generators. Let $\OF$ be the ring of integers of $F$, let
$\varpi$ be a uniformizer, and set
\bal
G=\GL_n(F),
\qquad
K=\GL_n(\OF).
\nal
Let $T$ be the diagonal torus, let $U$ be the upper unitriangular subgroup,
and set
\bal
K_T\coloneqq K\cap T,
\qquad
K_U\coloneqq K\cap U.
\nal
For $1\leq j\leq n$, let $T_j$ be the standard Hecke operator supported on
$Kt_jK$, where
\bal
t_j \coloneqq \diag(\underbrace{\varpi,\ldots,\varpi}_{j},1,\ldots,1).
\nal
If $I\subseteq\{1,\ldots,n\}$ has cardinality $j$, let $t_I$ be the diagonal
matrix with $\varpi$ in the positions indexed by $I$ and $1$ elsewhere, and
set $K_{U,I}\coloneqq t_IK_Ut_I^{-1}\cap K_U$.  

\begin{theorem}[The standard generators for $\GL_n$]
\label{thm-intro-GLn}
For every $0\leq i\leq\dim K_U$ and every $1\leq j\leq n$,
\bal
\Supp\SVi(T_j) \subseteq \bigsqcup_{\substack{I\subseteq\{1,\ldots,n\}\\ \#I=j}} K_T t_I.
\nal
For each subset $I$ of cardinality $j$,
\bal
\SVi(T_j)(t_I) =\cores_{K_{U,I}}^{K_U}\circ T_{j,I},
\nal
where $T_{j,I}:H^i(K_U,V)\rightarrow H^i(K_{U,I},V)$ is induced by conjugation by $t_I^{-1}$ and the coefficient map $T_j(t_I)$.
\end{theorem}
Thus, the support of $\SVi(T_j)$ is contained in the $K_T$-cosets
represented by the Weyl orbit of $t_j$, and its value at each such coset is
a concrete corestriction operator. At the two extreme cohomological degrees,
these formulas recover Herzig's invariant and coinvariant transforms. For
$n>2$, the interior support components need not be nilpotent: they can
retain nonzero cohomology arising from proper Levi subgroups. Thus,
determining the image in higher rank requires additional control over both
the individual support components and their interactions under convolution.

For $\GL_2(F)$, the situation is more rigid. Write $[F:\Q_p]=ef$, where $e$ and $f$ are the ramification index and residue degree, respectively. For
every Serre weight $V$, the spherical Hecke algebra is
\bal
\mathcal H_G(V)=k[T_1,T_2^{\pm1}].
\nal
The central generator $T_2$ acts by diagonal translation and remains
invertible under every higher Satake transform. For a non-Steinberg weight, that is, a Serre weight which is not a determinant twist of the Steinberg weight, the two off-diagonal support components of $\SVi(T_1)$ annihilate one another, and in every interior degree they are nilpotent.

\begin{theorem}[The image for $\GL_2$]
\label{thm-intro-GL2-image}
Let $V$ be a non-Steinberg Serre weight. In degree $0$, the image is
supported on the antidominant cone, and in degree $ef$, it is supported on
the dominant cone. For every $1\leq i\leq ef-1$, there is a positive
integer $R_i(V)$ such that
\bal
\operatorname{im}\SVi \cong k[x,y^{\pm1}]/\bigl(x^{R_i(V)}\bigr),
\nal
where $\SVi(T_1)$ maps to $x$ and $\SVi(T_2)$ maps to $y$. At the
endpoints,
\bal
\operatorname{im}\mathcal S_V^0
\cong
\operatorname{im}\mathcal S_V^{ef}
\cong
k[x,y^{\pm1}].
\nal
The integers $R_i(V)$ satisfy
\bal
R_i(V)=R_{ef-i}(V),
\nal
and, for every $1\leq i\leq ef-1$,
\bal
R_i(V) \in \{e-1-r_i,\ e-r_i,\ e+1-r_i\},
\qquad r_i \coloneqq \left\lfloor
\frac{\min\{i,ef-i\}-1}{f}
\right\rfloor.
\nal
\end{theorem}

In particular, every interior image has reduced algebra $k[y^{\pm1}]$,
while $R_i(V)$ measures its nilpotent thickness in the transverse
direction. The symmetry $R_i(V)=R_{ef-i}(V)$ gives isomorphic images in
complementary degrees. Writing $N_j(V)$ for the nilpotence index introduced
in Subsection~\ref{subsec-GL2-notation}, we have
\bal
R_i(V)=\max\{N_i(V),N_{ef-i}(V)\}.
\nal
The exact support theorem, Theorem~\ref{thm-GL2-image}, shows that
$N_i(V)$ and $N_{ef-i}(V)$ determine the boundary distances on the
antidominant and dominant sides of the diagonal, respectively; these
distances need not be equal. This is illustrated in
Figure~\ref{fig:gl2-support}.

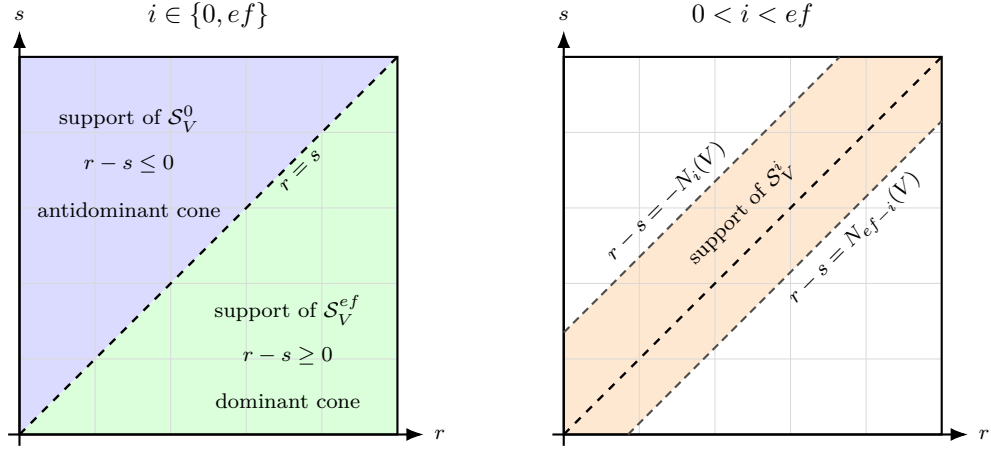
\begin{figure}[ht]
\centering
\begin{tikzpicture}[
    x=1cm,
    y=1cm,
    >=Latex,
    axis/.style={->, line width=0.8pt},
    gridline/.style={gray!30, line width=0.3pt},
    diag/.style={black, dashed, line width=0.9pt},
    bound/.style={black!70, dash pattern=on 3pt off 2pt, line width=0.8pt},
    every node/.style={font=\scriptsize}
]
\def\W{5}
\def\H{5}
\def\gap{2.2}
\def\a{1.35}
\def\b{0.85}
\begin{scope}[shift={(0,0)}]
    \fill[blue!14]
        (0,0) -- (0,\H) -- (\W,\H) -- (\W,\W) -- cycle;
    \fill[green!14]
        (0,0) -- (\W,0) -- (\W,\W) -- cycle;
    \foreach \k in {0,1,2,3,4,5} {
        \draw[gridline] (\k,0) -- (\k,\H);
        \draw[gridline] (0,\k) -- (\W,\k);
    }
    \draw[line width=0.8pt] (0,0) rectangle (\W,\H);
    \draw[axis] (-0.15,0) -- (\W+0.35,0) node[right] {$r$};
    \draw[axis] (0,-0.15) -- (0,\H+0.35) node[above] {$s$};
    \draw[diag] (0,0) -- (\W,\H);
    \node[rotate=45] at (3.72,3.50) {$r=s$};
    \node[font=\small] at (2.5,5.55) {$i\in\{0,ef\}$};
    \node at (1.45,4.18) {support of $\mathcal S_V^0$};
    \node at (1.45,3.58) {$r-s\leq 0$};
    \node at (1.45,2.98) {antidominant cone};
    \node at (3.55,1.65) {support of $\mathcal S_V^{ef}$};
    \node at (3.55,1.05) {$r-s\geq 0$};
    \node at (3.55,0.45) {dominant cone};

\end{scope}
\begin{scope}[shift={(\W+\gap,0)}]
    \fill[orange!18]
        (0,0)
        -- (\b,0)
        -- (\W,\W-\b)
        -- (\W,\H)
        -- (\W-\a,\H)
        -- (0,\a)
        -- cycle;
    \foreach \k in {0,1,2,3,4,5} {
        \draw[gridline] (\k,0) -- (\k,\H);
        \draw[gridline] (0,\k) -- (\W,\k);
    }
    \draw[line width=0.8pt] (0,0) rectangle (\W,\H);
    \draw[axis] (-0.15,0) -- (\W+0.35,0) node[right] {$r$};
    \draw[axis] (0,-0.15) -- (0,\H+0.35) node[above] {$s$};
    \draw[diag] (0,0) -- (\W,\H);
    \draw[bound] (0,\a) -- (\W-\a,\H);
    \draw[bound] (\b,0) -- (\W,\W-\b);
    \node[font=\small] at (2.5,5.55) {$0<i<ef$};
    \node[rotate=45]
        at (2.35,2.95)
        {support of $\mathcal S_V^i$};
    \node[rotate=45, anchor=south]
        at (1.55,2.90)
        {$r-s=-N_i(V)$};
    \node[rotate=45, anchor=north]
        at (3.65,2.80)
        {$r-s=N_{ef-i}(V)$};
\end{scope}
\end{tikzpicture}

\caption{Support picture for the images of the higher mod $p$
Satake transforms for $\GL_2(F)$ and a non-Steinberg Serre weight $V$,
in the lattice $T/K_T\cong\mathbf Z^2$, where $(r,s)$ corresponds to
$\diag(\varpi^r,\varpi^s)$.}
\label{fig:gl2-support}
\end{figure}

The reduced interior image also has a direct representation-theoretic
interpretation. For $\GL_2(F)$, Herzig's supersingularity condition is equivalent to requiring that the noncentral generator $T_1$ act trivially
\cite[Definition~4.7(iii) and Proposition~2.12]{herzig2011classification}. Since the isomorphism above is compatible with the standard generators, for every
interior degree we have
\bal
\ker(\SVi) =\bigl(T_1^{R_i(V)}\bigr),
\qquad
\sqrt{\ker(\SVi)} = (T_1),
\nal
and hence
\bal
\bigl(\operatorname{im}\SVi\bigr)_{\mathrm{red}}
\cong \mathcal H_G(V)/(T_1) \cong k[T_2^{\pm1}].
\nal
Thus, every interior transform is noninjective; more precisely, the radical
of its kernel is the supersingular ideal $(T_1)$, and its reduced image is
the corresponding Hecke quotient.

The nilpotent structure is particularly simple when $F/\Q_p$ is
unramified.

\begin{cor}[The unramified case]
\label{cor-intro-unramified}
Assume that $F/\Q_p$ is unramified and let $V$ be a Serre weight. Then, for
every $1\leq i\leq f-1$,
\bal
\operatorname{im}\SVi
\cong
k[x,y^{\pm1}]/\bigl(x^{R_i(V)}\bigr),
\qquad
R_i(V)\in\{1,2\}.
\nal
Thus, every interior image is either $k[y^{\pm1}]$ or
a square-zero thickening of it. For Steinberg Serre weights and sufficiently
generic non-Steinberg Serre weights, we have
\bal
R_i(V)=1
\qquad
\text{for every }1\leq i\leq f-1,
\nal
so all the interior images are reduced.
\end{cor}

The precise genericity condition is given in
Corollary~\ref{cor-unramified-image}. In particular, it is satisfied whenever
the symmetric-power exponents defining $V$ lie strictly between $0$ and
$p-1$.

The bounds in the image theorem are sharp. For a determinant twist of the
trivial weight, we obtain the exact formula
\bal
R_i(V)
=
1+
\left\lfloor
\frac{\max\{i,ef-i\}}{f}
\right\rfloor,
\qquad
1\leq i\leq ef-1.
\nal
Thus, all interior images in this family are reduced when $F/\Q_p$ is
unramified and nonreduced when $F/\Q_p$ is ramified, with the nilpotent
thickness varying in steps with the cohomological degree. Two further
families of Serre weights realize
\bal
R_1(V)=e+1
\qquad\text{and}\qquad
R_1(V)=e-1,
\nal
respectively. Consequently, when $e\geq2$ and $f\geq2$, the examples realize all three values permitted by the theorem in degree $1$. In particular, the nilpotent thickness depends on the ramification, the cohomological degree, and the Serre weight.

\subsection{Notation and assumptions}
\label{subsec-notation-assumptions}
We now set notation which will remain in force for the remainder of the paper, unless stated otherwise.  Let $k$ be a field of characteristic $p$.  Given a $p$-adic Lie group $G$, write $\Rep_k(G)$ for the category of smooth representations of $G$ on $k$-vector spaces.  As in \cite{schneider2015smooth}, we define $D(G)\coloneqq  D(\Rep_k(G))$ to be the derived category of unbounded cochain complexes.

Our basic setup will be the same as in \cite[Subsection 4.3]{heyer2023left}, which can be summarized as follows.  Let $\mathbf G$ be a connected reductive group defined over a finite field extension $F/\Q_p$, and set $G=\mathbf G(F)$.  Fix a parabolic subgroup $P\leq G$ with Levi quotient $M$ and unipotent radical $U$, so that $P=UM$.  Let also $K\leq G$ be an open compact mod center subgroup satisfying the Iwasawa decomposition $G=PK$ and $K_P=K_UK_M$, where as usual we are writing
$$K_P=K\cap P,\qquad K_U=K\cap U,\qquad K_M=K\cap M.$$
We define $M^+$ to be the positive monoid consisting of all $m \in M$ such that $m K_U m^{-1} \subseteq K_U$. We similarly write $P^+= K_U M^+$, which is also a monoid.

Given $V\in \Rep_k(K)$ and $W\in \Rep_k(K_M)$, we define their Hecke algebras to be
$$\mc H_G(V)\coloneqq\End_G(\ind_{K}^G V),\qquad \mc H_M(W)\coloneqq \End_M(\ind_{K_M}^M W).$$
Given a group $A$, a subgroup $B\leq A$, and an element $a\in A$, we write
$$\prescript{a}{}B=aBa\inv,\qquad B^a=a\inv B a.$$

In Section \ref{sec-GLn}, we will specialize to the case of $G=\GL_n(F)$, $K=\GL_n(\OF)$, $M=T$ its diagonal torus, and $P=B$ the standard Borel.  In Section \ref{sec-GL_2}, we will specialize even further and assume $n=2$.

In Sections \ref{sec-GLn} and \ref{sec-GL_2}, we restrict to irreducible
$V\in \Rep_k(K)$ (i.e., Serre weights); however, irreducibility is not
required in Section \ref{expl-formulas-sec}.

\subsection{Heyer's construction of the higher mod \texorpdfstring{$p$}{p} Satake transforms}
In this subsection, we briefly recall the part of Heyer's construction
\cite{heyer2023left} that gives rise to the Satake transforms studied in
this paper. Let
$$
i^G_M:\Rep_k(M)\longrightarrow\Rep_k(G)
$$
denote parabolic induction, given by the composition
\[
\begin{tikzcd}
\Rep_k(M)
\arrow[r,"{\Inf^P_M}"] &
\Rep_k(P)
\arrow[r,"{\Ind^G_P}"] &
\Rep_k(G)
\end{tikzcd}
\]
of inflation and smooth induction. In \cite{heyer2023left}, Heyer writes
the indices on inflation in the opposite order. Since both functors are
exact, $i^G_M$ induces a functor
$$
\mathrm{R}i^G_M:D(M)\longrightarrow D(G).
$$
Heyer's main construction gives an explicit description of a left adjoint
to this derived parabolic induction. Since
$$
\mathrm{R}i^G_M=\mathrm{R}\Ind^G_P\circ \mathrm{R}\Inf^P_M
$$
and $\mathrm{R}\Res^G_P$ is left adjoint to $\mathrm{R}\Ind^G_P$, the main new ingredient is
the construction of a left adjoint $L_U$ to $\mathrm{R}\Inf^P_M$. This is a special
case of the following more general result of Heyer:
\begin{theorem}
    [Heyer, Theorem A] 
    \label{thm-A}
    Let $S$ and $T$ be open submonoids of $p$-adic Lie groups, endowed with a surjective open monoid morphism $f: S \onto T$. The functor $\RInf^S_T : D(T) \to D(S)$ admits a left adjoint.
\end{theorem}

Heyer describes this left adjoint via the following result:

\begin{theorem}[Heyer, Theorem B]\label{thm-B}
    There exists a smooth character $\omega_{P^+}$ of $P^+$, sitting in cohomological degree $-\dim U$, for which there is an isomorphism of functors
    $$
    \ind_{M^+}^M \mathrm{R}H^0(K_U,\omega_{P^+}\otimes_k -) \Res^P_{P^+}
    \cong L_U.
    $$
\end{theorem}
Similar results hold true for the left adjoint
$$L(U,-) \coloneqq  L_U \circ \operatorname{RRes}_P^G$$ 
of $\mathrm{R}i^G_M$. 

By \cite[Corollary~3.1.8]{heyer2023left}, there is a left adjoint $L_{K_U}:D(K_P)\longrightarrow D(K_M)$ to the derived inflation functor
$\mathrm{R}\Inf^{K_P}_{K_M}:D(K_M)\longrightarrow D(K_P).$ Writing $L(K_U, -) = L_{K_U} \circ \Res^K_{K_P} : D(K) \to D(K_M)$, we have (as explained in \cite[Section 4.3]{heyer2023left}) a natural isomorphism of functors 
\begin{equation}\label{natiso}
L(U, -)\ind_K^G \cong \ind_{K_M}^M L (K_U, -).
\end{equation}
This naturally induces for any $V\in \Rep_k(K)$ a $k$-algebra homomorphism
$$\mathcal S_V: \End_{D(G)}(\ind_K^G(V)) \to \End_{D(M)}(L(U, \ind_K^G(V))),$$
and the natural isomorphism \eqref{natiso} induces from this a further map
$$\mathcal S_V: \End_{D(G)}(\ind_K^G(V)) \to \End_{D(M)}(\ind_{K_M}^M L(K_U, V)).$$
For every $0\leq i\leq \dim K_U$, taking cohomology in degree
$i-\dim K_U$ yields a $k$-algebra homomorphism
\begin{equation}\label{Sat-i}
\mathcal S_V^i:
\End_{\Rep_k(G)}(\ind_K^G V)
\longrightarrow
\End_{\Rep_k(M)}
\left(
\ind_{K_M}^M L^{i-\dim K_U}(K_U,V)
\right)
\end{equation}
which we refer to as the \textit{$i$-th higher mod $p$ Satake transform}. By the same argument as in \cite[Explanation~4.3.3]{heyer2023left}, we have an isomorphism
\bal
L^{i-\dim K_U}(K_U,V)\cong H^i(K_U,V) \otimes_k \delta_P.
\nal for a smooth character $\delta_P$ of $P$.
It then follows that \eqref{Sat-i} can be rewritten as
$$\mathcal S_V^i: \End_{\Rep_k(G)}(\ind_K^G V)\to \End_{\Rep_k(M)}(\ind_{K_M}^M H^i(K_U,V)),$$
or simply as a map $\SVi:\mc H_G(V)\to \mc H_M(H^i(K_U,V))$ of Hecke algebras. These are precisely the maps studied in this paper.

\subsection{Relation to derived Satake morphisms}

There is also a related construction of Satake morphisms for derived
Hecke algebras. Venkatesh introduced a derived Satake isomorphism in a
coefficient setting away from the residue characteristic
\cite{venkatesh2019derived}. For the trivial weight, Ronchetti constructs a Satake morphism
between graded derived spherical Hecke algebras
\cite{ronchetti2018satake}. In this construction, the grading is the
Ext-degree. By contrast, the maps $\SVi$ considered here are maps between
ordinary Hecke algebras, indexed instead by the cohomological degree
$i$ of $H^i(K_U,V)$.

These two directions are combined in the work of Kozio{\l}--P{\'e}pin
\cite{koziol2024derived}. Under additional hypotheses, including that
$F/\Q_p$ is unramified, $\mathbf G$ is split, and the weight is $p$-small
and sufficiently regular, they prove an orthogonal splitting of the
complex obtained from Heyer's left adjoint and use it to construct a
family of morphisms of graded spherical Hecke algebras. The Ext-degree-zero parts of these morphisms recover Heyer's family of ordinary Hecke-algebra maps; under the indexing $n=i-\dim K_U$, these are precisely the maps
$\SVi$ considered here. For the trivial weight and $n=0$, their
construction recovers Ronchetti's graded Satake morphism.

Our focus is instead on explicit formulas and images for these Ext-degree-zero maps. In particular, our general formulas do not
require the unramified or $p$-small and regular hypotheses above, and our
image theorem for $\GL_2(F)$ also treats ramified extensions $F/\Q_p$.
We nevertheless use the work of Kozio{\l}--P{\'e}pin in the proof of the
generic reducedness statement in Corollary~\ref{cor-unramified-image}.

\subsection{Acknowledgments}
We are deeply indebted to Florian Herzig for suggesting this problem and for
taking so much time to discuss it with us during the Arizona Winter School,
often staying up late to answer our many questions. We learned a great deal both
from those conversations and from his AWS notes on the subject. We are also very grateful to Heejong Lee for all of his help while the project was
taking shape. His encouragement, thoughtful feedback, and careful attention to
our ideas and arguments meant a great deal to us. We thank Karol Kozioł for his insightful feedback on this project, which led us to pursue the image and nilpotence questions that became central
to the paper.  We also thank Claudius
Heyer for valuable conversations that helped us understand several aspects of
his work. Finally, we thank the organizers of the 2025 Arizona Winter School for creating such a wonderful environment for learning and discussing mathematics, where this work began.

\subsection{Statement of AI use} GPT-5.6 Sol was used to assist with editing this manuscript. All mathematical results and arguments are
the work of the authors.

\section{Explicit formulas for the higher mod \texorpdfstring{$p$}{p} Satake transforms} \label{expl-formulas-sec}

In this section, our main goal will be to prove two explicit formulas for the higher mod $p$ Satake transforms $$\SVi:\mc H_G(V)\to \mc H_M(H^i(K_U,V))$$
we are studying.  In particular, this section includes the proof of Theorem \ref{thm-intro-general-formula}, stated in the introduction.

After beginning with a brief review of some relevant background, we proceed with proving in subsection \ref{double-coset-subsect} an analogue of \cite[Theorem 4.3.2]{heyer2023left}, namely Proposition \ref{prop-explicit-formula}.  This result specifically states that
\begin{align*}
\SVi(\Phi)([1,\xi]) = \sum_{g \in K_P \backslash P / K_U} [\pr_M (g^{-1}), \cores^{(i)}_{g^{-1}} ( \pr_{g^{-1}}^{(i)} (\Phi[1, \xi ] ) ) ]
\end{align*}
for $i \in \{0, \dots, \mathrm{dim}(U)\}$ (the notation is explained in the subsections below), but the idea is that this is the closest analogue we can obtain to the usual formulas defined by Herzig for the $K_U$-invariants and coinvariants in \cite{herzig2011classification}.  The proof of \cite[Theorem 4.3.2]{heyer2023left}, in part, demonstrates precisely that this formula reduces to that studied by Herzig in the cases he considered.

While this result is desirable in the sense that it generalizes known facts, it can be useful for applications to have a formula on the level of cochains.  In subsection \ref{cochain-subsec}, we introduce for each $m\in M$ a morphism 
$$C_m^{(i)}:H^i(K_U,\ind_{K_P}^P(-))\to H^i(K_U,-)$$
of $\delta$-functors $\Rep_k(K_P)\to \Ab$ defined explicitly on the level of cochains; see Definition \ref{def-delta-functor} for the explicit construction, which the authors based directly on the classical formulas for the Satake transform.  We then prove in Theorem \ref{thm-explicit-formula} that one can easily express $\SVi$ using this morphism of $\delta$-functors.  Concretely, we show for each $m\in M$ that there is a morphism of $\delta$-functors
$$C_m^{(i)}:H^i(K_U,\ind_{K_P}^P(-))\to H^i(K_U,-),$$
defined explicitly on the level of cochains, such that
$$\SVi(\Phi)([1,\xi])(m)=C_m^{(i)}(\Phi([1,\xi])),$$
where we again refer to the section below for details regarding notation.  As an example of how we apply this cochain formula later in the paper, see the proof of Theorem \ref{thm-GLn-SVi-support-formula} where this is repeatedly and crucially used.

\subsection{Mackey decomposition and corestriction}

In this subsection, we will review some facts involving Mackey decomposition and the corestriction maps in group cohomology, and note the relevant consequences in our case.  In particular, we will be interested in noting when we have morphisms of $\delta$-functors for use in the next two subsections.

Recall that $\prescript{g}{}{K_P}=gK_Pg\inv$ and $K_U^g=g\inv K_U g$.  By \cite[Theorem 1.1]{mackey}, we have the following Mackey decomposition for smooth representations $W$ of $K_P$, where for $g\in G$, $\prescript{g}{}{W}$ is the smooth representation of $\prescript{g}{}{K_P}$ defined by $\pre gW\coloneqq W$ with action $x\cdot w\coloneqq (g^{-1} xg)w$ for $x\in \pre gK_P$:

\begin{lem}
\label{mackey}
$$\Res_{K_U}^P\ind_{K_P}^P W\iso \bigoplus_{g\in K_U\backslash P/K_P}\ind_{K_U\cap \pre gK_P}^{K_U}\Res_{K_U\cap \pre gK_P}^{\pre gK_P} \pre gW.$$
\end{lem}
Concretely, this isomorphism sends a left-$K_P$-invariant function $f:P\to W$ to the tuple of functions 
$$(f_g:K_U\to \pre gW)_{g\in K_U\backslash P/K_P},\qquad f_g(x)\coloneqq f(g\inv x).$$  
The induced inclusion 
$$\ind_{K_U\cap \pre gK_P}^{K_U} \Res_{K_U\cap \pre gK_P}^{\pre gK_P} \pre gW\into \Res_{K_U}^P\ind_{K_P}^P W$$
for a fixed $g\in K_U\backslash P/K_P$ is given by $[a,w]\mapsto [ag,w]$.  We can use this to prove the following, which is noted when $i=0$ in \cite[pg. 52]{heyer2023left}:
\begin{prop}
\label{cohomological Mackey}
There are natural isomorphisms
$$H^i(K_U,\ind_{K_P}^P W)\overset{\sim}{\longrightarrow} \bigoplus_{g\in K_U\backslash P/K_P} H^i(K_U^g\cap K_P,W).$$
\end{prop}

\begin{proof}
We will prove this using a string of isomorphisms at the level of cochains, using the notation of \cite[Chapter 1]{NSW}.  (This can also be seen without considering cochains, however we do this for the sake of obtaining explicit formulas for later use.)  From the Mackey decomposition, we get natural isomorphisms at the level of homogeneous cochains

$$X^i(K_U, \ind_{K_P}^P W)^{K_U}\overset{\sim}\longrightarrow\bigoplus_{g\in K_U\backslash P/K_P} X^i(K_U,\ind_{K_U\cap \pre gK_P}^{K_U} \Res_{K_U\cap \pre gK_P}^{\pre gK_P} \pre gW)^{K_U}.$$
Since $K_U\cap \pre gK_P$ is open and hence of finite index in $K_U$, the compact induction and full induction coincide by \cite[Lemma 6.3.4]{Weibel}, so we can conclude from Shapiro's Lemma (see \cite[Proposition 1.6.4]{NSW}) that the above is naturally isomorphic to
$$\bigoplus_{g\in K_U\backslash P/K_P} X^i(K_U,\Res_{K_U\cap \pre gK_P}^{\pre gK_P} \pre gW)^{K_U\cap \pre gK_P}\overset{\sim}\longrightarrow \bigoplus_{g\in K_U\backslash P/K_P} X^i(K_U^g,W)^{K_U^g\cap K_P}.$$
One can then conclude by taking cohomology and noting that $X^i(K_U^g,W)$ is an acyclic resolution of $W$ as a representation of $K_U^g\cap K_P$ by \cite[Proposition 1.3.7]{NSW}.
\end{proof}

\begin{notation}
This canonical decomposition gives us natural projection and inclusion maps for any double coset representative $g\in P$, which we denote by
\begin{align*}
j_g^{(i)}: &H^i(K_U^g\cap K_P, W)\into H^i(K_U,\ind_{K_P}^P W),\\
\pr_g^{(i)}:&H^i(K_U,\ind_{K_P}^P W)\onto H^i(K_U^g\cap K_P, W).
\end{align*}
\end{notation}

\begin{remark}\label{mackey-on-cochains}
One can chase through the proof of Proposition \ref{cohomological Mackey}, using the explicit equation of the map coming from Shapiro's lemma given in \cite[Proposition 1.6.4]{NSW}, to arrive at explicit formulas at the level of cochains for 
\begin{align*}
\pr_g^{(i)}:&X^i(K_U,\ind_{K_P}^P W)^{K_U}\to X^i(K_U^g, W)^{K_U^g\cap K_P}.
\end{align*}

Concretely, given $x\in X^i(K_U,\ind_{K_P}^P W)^{K_U}$, its image under $\pr_g^{(i)}$ is given on $(\sigma_0,\ldots, \sigma_i)\in (K_U^g)^{i+1}$ by
$$\pr_g^{(i)}(x)(\sigma_0,\ldots, \sigma_i)=x(g\sigma_0g\inv,\ldots,g\sigma_ig\inv)(g\inv)\in W.$$
One can also deduce a general (quite complicated) cochain formula for $j_g^{(i)}$, however we will only need a description of it when $i=0$ in cohomology.  Explicitly, given $w\in H^0(K_U^g\cap K_P,W)=W^{K_U^g\cap K_P}$, we have that
$$j_g^{(0)}(w)=\sum_{u\in K_U/K_U\cap \prescript{g}{}K_P} [ug,w]$$
in $H^0(K_U,\ind_{K_P}^P W)=(\ind_{K_P}^P W)^{K_U}$; this fact is implicitly used by Heyer in \cite[pg. 52]{heyer2023left}.
\end{remark}

\begin{remark}
\label{deltafunctorremark}
We note for future use that the above maps $\pr_g^{(i)}$ and $j_g^{(i)}$ form morphisms of universal $\delta$-functors, and Lemma \ref{cohomological Mackey} can itself be rephrased as an isomorphism of $\delta$-functors. More precisely, fix any $g\in P$, and consider the following families of additive functors $\Rep_k(K_P)\to \Ab$:
$$F=\big(H^i(K_U,\ind_{K_P}^P (-))\big)_{i\geq 0},$$
$$F_g=\big(H^i(K_U^g\cap K_P,-)\big)_{i\geq 0}.$$
Then we claim the following:
\begin{enumerate}[label = (\roman*)]
\item $F$ and $F_g$ are universal $\delta$-functors $\Rep_k(K_P)\to \Ab$.
\item The natural maps $j_g^{(i)}$, $\pr_g^{(i)}$ form morphisms of $\delta$-functors
$$j_g:F_g\to F\qquad\text{and}\qquad \pr_g:F\to F_g,$$
respectively.
\end{enumerate}
Moreover, $\bigoplus_{g\in K_U\backslash P/K_P} F_g=\left(\bigoplus_{g\in K_U\backslash P/K_P} H^i(K_U^g\cap K_P,-)\right)_{i\geq 0}$ is itself a universal $\delta$-functor $\Rep_k(K_P)\to \Ab$, and our isomorphism above is in fact an isomorphism of universal $\delta$-functors.
\end{remark}

\begin{proof} (i) That $F_g$ is merely a $\delta$-functor is an easy observation: it is the family resulting from precomposing the terms in the $\delta$-functor $\big(H^i(K_U^g\cap K_P,-)\big)_{i\geq 0}:\Rep_k(K_U^g\cap K_P)\to \Ab$ with the exact functor $\Res^{K_P}_{K_U^g\cap K_P}:\Rep_k(K_P)\to \Rep_k(K_U^g\cap K_P)$.  To see that it is universal, it suffices, by \cite[Theorem 2.6.2]{NSW}, to show that each term for $i>0$ is effaceable; in particular, it suffices to find, for any $i> 0$ and any $W\in \Rep_k(K_P)$, a monomorphism $\iota:W\into I$ in $\Rep_k(K_P)$ such that $H^i(K_U^g\cap K_P,I)=0$.  To do this, we let $I\coloneqq \Ind_{\{1\}}^{K_P} W$, which is cohomologically trivial by \cite[Proposition 1.3.7]{NSW} and has a natural $K_P$-equivariant embedding $W\into I$, hence is sufficient for our purposes.  Note also that the independence of our choice of $I$ here from $g\in P$ implies that the $\delta$-functor $\bigoplus_{g\in K_U\backslash P/K_P} F_g$ is also universal, as its $i$-th term also vanishes on $I$ for $i>0$.

The argument to show $F$ is a universal $\delta$-functor is almost identical, so we omit it.


\noindent (ii) We will only argue why $\pr_g$ gives a morphism of $\delta$-functors $F\to F_g$; the argument for $j_g:F_g\to F$ is exactly the same idea.

To prove this, we recall that the proof of Proposition \ref{cohomological Mackey} shows that the construction of $\pr_g^{(i)}$ involves a chain of homomorphisms of cohomology groups, which we will show gives morphisms of $\delta$-functors one piece at a time.  Fix any $g\in P$.  Since the Mackey decomposition is functorial, we have a natural transformation of exact additive functors $\Rep_k(K_P)\to \Rep_k(K_U)$ coming from a single Mackey summand
$$\Res_{K_U}^P\ind_{K_P}^P(-)\to \ind_{K_U\cap \pre gK_P}^{K_U}\Res_{K_U\cap \pre gK_P}^{\pre gK_P} \pre g{(-)}.$$
Since $(H^i(K_U,-))_{i\geq 0}$ is a $\delta$-functor $\Rep_k(K_U)\to \Ab$, it follows from Lemma \ref{delta1} that the induced natural transformations
$$H^i\left(K_U,\ind_{K_P}^P(-)\right)\to H^i\left(K_U,\ind_{K_U\cap \pre gK_P}^{K_U}\Res_{K_U\cap \pre gK_P}^{\pre gK_P} \pre g{(-)}\right)$$
for $i\geq 0$ together form a morphism of $\delta$-functors $\Rep_k(K_P)\to \Ab$.

Next, we have a natural isomorphism 
$$\mathrm{sh}:H^i(K_U,\ind_{K_U\cap \pre gK_P}^{K_U}(-))\overset{\sim}\longrightarrow H^i(K_U\cap \pre gK_P,-)$$
of functors $\Rep_k(K_U\cap \pre gK_P)\to \Ab$ for all $i\geq 0$ coming from Shapiro's Lemma, which is in fact a morphism of $\delta$-functors as discussed in \cite[pg. 129]{NSW}.  It is then immediate from Lemma \ref{delta2} that precomposing with the exact functor 
$$\Res_{K_U\cap \pre gK_P}^{\pre gK_P}\pre g{(-)}:\Rep_k(K_P)\to \Rep_k(K_U\cap \pre gK_P)$$
gives us a morphism of $\delta$-functors $\Rep_k(K_P)\to \Ab$ which is given by
$$H^i\left(K_U,\ind_{K_U\cap \pre gK_P}^{K_U}\Res_{K_U\cap \pre gK_P}^{\pre gK_P} \pre g{(-)}\right)\to H^i\left(K_U\cap \pre gK_P,\Res_{K_U\cap \pre gK_P}^{\pre gK_P} \pre g(-)\right).$$

For the final piece, we will apply Lemma \ref{delta3}.  Consider the following four functors:
\begin{align*}
F_1&:\Rep_k(K_U\cap \pre gK_P)\to \Ab,\qquad\qquad V\mapsto V^{K_U\cap \pre gK_P};\\
F_2&:\Rep_k(K_U^g\cap K_P)\to \Ab,\qquad\qquad W\mapsto W^{K_U^g\cap K_P};\\
G_1\coloneqq \Res_{K_U\cap \pre gK_P}^{\pre gK_P} \pre g{(-)}&:\Rep_k(K_P)\to \Rep_k(K_U\cap \pre gK_P);\\
G_2\coloneqq \Res_{K_U^g\cap K_P}^{K_P}&:\Rep_k(K_P)\to \Rep_k(K_U^g\cap K_P).
\end{align*}
Note that $F_1,F_2$ are left-exact, $G_1,G_2$ are exact, and our earlier computations directly show that the $G_i$ take injective objects to $F_i$-acyclic objects for $i=1,2$.  One can easily check the following claim from the definition of $\pre gW$: given a representation $W$ of any subgroup $K\leq H$ and an element $g\in H$, there is a natural identification $(\pre gW)^{\pre gK}=W^K$ given by $w\mapsto w$ in the underlying sets.  In our case, this amounts to a natural isomorphism of functors $F_1\circ G_1\to F_2\circ G_2$, allowing us to apply our lemma.  Thus, we conclude that there is an induced morphism of $\delta$-functors $R(F_1)\circ G_1\to R(F_2)\circ G_2$, the construction of which can be unwound to show that the maps 
$$H^i\left(K_U\cap \pre gK_P,\Res_{K_U\cap \pre gK_P}^{\pre gK_P}\pre g{(-)}\right)\to H^i\left(K_U^g\cap K_P,\Res_{K_U^g\cap K_P}^{K_P}(-)\right)$$
we defined for all $i\geq 0$ in the proof of Proposition \ref{cohomological Mackey} form a morphism of $\delta$-functors.

The composition of all these morphisms of $\delta$-functors recovers exactly for $i\geq 0$ our maps
$$\pr_g^{(i)}:H^i(K_U,\ind_{K_P}^P(-))\to H^i(K_U^g\cap K_P,-)$$
of functors $\Rep_k(K_P)\to \Ab$, which thus form a morphism of $\delta$-functors.

The final remark that our isomorphisms in Proposition \ref{cohomological Mackey} in fact form an isomorphism of universal $\delta$-functors $\Rep_k(K_P)\to \Ab$
$$\left(H^i(K_U,\ind_{K_P}^P(-))\right)_{i\geq 0}\overset{\sim}\longrightarrow \left(\bigoplus_{g\in K_U\backslash P/K_P} H^i(K_U^g\cap K_P,\Res^{K_P}_{K_U^g\cap K_P}(-))\right)_{i\geq 0}$$
is now immediate: it is a $\delta$-functor morphism which is an isomorphism in every degree, hence the family of homomorphisms consisting of inverses on each term is itself a morphism of $\delta$-functors.
\end{proof}

We now briefly review the definition and basic properties of the corestriction map on group cohomology. For just the rest of this section, let $G$ be any profinite group, and let $K$ and $H$ be open (hence finite index) subgroups.  All facts will be taken straight from \cite[Section 1.5]{NSW}.

\begin{definition}
\label{def-cores}
Let $A$ be any discrete $G$-module.  The \textbf{corestriction} map $\cores: H^i(H,A)\to H^i(G,A)$ is defined on cochains as
\begin{equation}\label{cores-on-cochains}
\begin{aligned}
\cores : X^i(H,A)^H \to X^i(G,A)^G,\hspace{2.65cm} \\
(\cores(x))(\sigma_0,\ldots, \sigma_i) 
= \sum_{g \in H \backslash G} \overline{g}^{-1} 
x\left(\overline{g} \sigma_0 (\overline{g\sigma_0})^{-1},\ldots, 
\overline{g} \sigma_i (\overline{g\sigma_i})^{-1}\right).
\end{aligned}
\end{equation}

\end{definition}

\begin{prop} 
\begin{enumerate}
\item If $H=G$, then $\cores:H^i(G,A)\to H^i(G,A)$ is the identity.
\item If $K\leq H\leq G$, then the composition
$$H^i(K,A)\overset{\cores}\longrightarrow H^i(H,A)\overset{\cores}\longrightarrow H^i(G,A)$$
is exactly $\cores: H^i(K,A)\to H^i(G,A)$.
\item The maps $\cores:H^i(H,A)\to H^i(G,A)$  are functorial in $A$, and together form a morphism of $\delta$-functors $\Rep_k(G)\to \Ab$,
$$\cores:(H^i(H,\Res^G_H(-)))_{i\geq 0}\to (H^i(G,-))_{i\geq 0}.$$
\end{enumerate}
\end{prop}

\begin{proof}
The first two are easy to see, and the third is part of \cite[Proposition 1.5.2]{NSW}.
\end{proof}

We now go back to our usual notation and, following \cite[Notation 4.3.6]{heyer2023left}, introduce notation for a specific corestriction map that will be used regularly in the sequel.

\begin{notation}
Given any $g\in P$ and any $m\in M^+$, we use the following notation to denote a particular corestriction:
$$\cores^{(i)}_{m,g}: H^i(K_U^g\cap K_P,V)\to H^i(K_U^{mg}\cap K_P,V).$$
Since $m$ is positive, $K_U^{mg}\cap K_P\supseteq K_U^g\cap K_P$, hence this is well-defined.

Similarly, for any $g\in P$, pick any $m\in M^+$ such that $mg\in P^+$, and write 
$$\cores^{(i)}_g\coloneqq \cores^{(i)}_{m,g}:H^i(K_U^g\cap K_P,V)\to H^i(K_U,V).$$
Note that this is also just the usual corestriction from the open subgroup $K_U^g\cap K_P$ to $K_U$.  That this is well-defined and independent of the choice of $m$ follows from the observation that for any $g\in P^+$, we have $K_U^{g}\cap K_P=K_U$ just as in \cite[Lemma 4.3.5]{heyer2023left}.  (As Heyer explains, there always exists such a choice of $m\in M^+$.)
\end{notation}

\subsection{The double coset formula for \texorpdfstring{$\SVi$}{S(V,i)}} \label{double-coset-subsect}
For now, let $W\in \Rep_k(K_P)$ be any smooth representation of $K_P$.

\begin{prop} 
\label{Exp1}
There is an $M$-equivariant isomorphism 
\begin{align*} 
\kappa: \ind_{M^+}^M H^i (K_U, \ind_{K_P}^P W) &\overset{\sim}{\longrightarrow} \ind_{K_M}^M H^i(K_U, W)\\
1\otimes j_g^{(i)}(\xi) &\longmapsto [\mathrm{pr}_M(g), \cores_g^{(i)} (\xi)]
\end{align*}
for all $i \in \{0, \dots, \mathrm{dim}(U)\}$, where $\xi \in H^i(K_U^g\cap K_P,W), g\in K_U\backslash P/K_P$ and $\mathrm{pr}_{M}$ denotes the projection $P\twoheadrightarrow M$.
\end{prop}
We build towards a proof of this proposition by proving some intermediate lemmas.

By an analogous Mackey decomposition to that of Lemma \ref{cohomological Mackey}, we have natural isomorphisms
$$H^i(K_U,\ind_{K_P}^{P^+} W)\overset{\sim}{\longrightarrow} \bigoplus_{g\in K_U\backslash P^+/K_P} H^i(K_U^g\cap K_P,W) =  \bigoplus_{g\in K_U\backslash P^+/K_P} H^i(K_U,W).$$
For any $g\in K_U\backslash P^+/K_P$, we denote by $\pre +j^{(i)}_{g}$ the natural inclusion map $$\pre +j^{(i)}_g:H^i(K_U,W)=H^i(K_U^g\cap K_P,W)\hookrightarrow H^i(K_U,\ind^{P^+}_{K_P}W).$$ As in the previous section, $\pre +j^{(i)}_g: H^i(K_U,-)\rightarrow H^i(K_U,\ind^{P^+}_{K_P}(-))$ is a morphism of $\delta$-functors. Let $\iota^{(i)}_W$ denote the morphism
$$H^i(K_U,\ind^{P^+}_{K_P}(W))\longrightarrow H^i(K_U,\ind^P_{K_P}(W))$$ induced by the inclusion $\ind^{P^+}_{K_P}(W)\subseteq \ind ^P_{K_P}(W)$.

\begin{lem}
\label{jplus}
We have $(\iota^{(i)}_W\circ \pre +j^{(i)}_g)(\xi)=j^{(i)}_g(\xi)$ for any $g\in K_U\backslash P^+/K_P, \xi \in H^i(K_U,W)$, i.e., the diagram
\begin{center}
\begin{tikzcd}
{H^i(K_U,W)} \arrow[rr, "\pre +j^{(i)}_g"] \arrow[rrdd, "j^{(i)}_g"'] &  & {H^i(K_U,\ind^{P^+}_{K_P}(W))} \arrow[dd, "\iota^{(i)}_W"] \\
                                                                                      &  &                                                           \\
                                                                                      &  & {H^i(K_U,\ind^P_{K_P}(W))}                               
\end{tikzcd}
\end{center} commutes.
\end{lem}
\begin{proof}
First, we will apply Lemma \ref{delta3} to show that $(\iota^{(i)})_{i\geq 0}$ is a morphism of $\delta$-functors.  Consider the following four functors:
\begin{align*}
F_1=F_2&:\Rep_k(K_U)\to \Ab,\qquad\qquad V\mapsto V^{K_U}=H^0(K_U,V);\\
G_1\coloneqq \Res_{K_U}^{P^+} \ind^{P^+}_{K_P}{(-)}&:\Rep_k(K_P)\to \Rep_k(K_U);\\
G_2\coloneqq \Res_{K_U}^{P}\ind ^P_{K_P}(-)&:\Rep_k(K_P)\to \Rep_k(K_U).
\end{align*}    
Note that $F_1,F_2$ are left-exact, $G_1,G_2$ are exact, and by Mackey decomposition (\cite[Lemma 3.2.15]{heyer2023left} or Lemma \ref{mackey}) we have
$$G_1=\Res_{K_U}^{P^+} \ind^{P^+}_{K_P}{(-)} \cong \bigoplus_{g\in K_U\backslash P^+/K_P}\ind^{K_U}_{\pre gK_U}\Res^{\pre gK_P}_{\pre gK_U}\pre g(-).
$$ Each functor $\ind^{K_U}_{\pre gK_U}, \Res^{\pre gK_P}_{\pre gK_U}, \pre g(-)$ preserves injective objects (as they have exact left adjoints) and a direct sum of injective $K_U$-representations is a $F_1=F_2=H^0(K_U,-)$-acyclic object. Similarly, $G_2$ also sends injective objects to $F_1=F_2$-acyclic objects. The natural transformation $\mathrm{Id}\rightarrow \Res^{P}_{P^+}\ind^{P}_{P^+}$ induces a natural transformation $$G_1=\Res_{K_U}^{P^+} \ind^{P^+}_{K_P}{(-)}\rightarrow \Res_{K_U}^{P^+}\Res^{P}_{P^+}\ind^{P}_{P^+} \ind^{P^+}_{K_P}{(-)}=\Res_{K_U}^{P} \ind^{P}_{K_P}{(-)}=G_2$$ which further induces a natural transformation $\eta:F_1\circ G_1\rightarrow F_2\circ G_2$. By Lemma \ref{delta3}, we get an induced morphism of $\delta$-functors
$$\mathrm{R}^i\eta: \mathrm{R}^iF_1\circ G_1\longrightarrow \mathrm{R}^iF_2\circ G_2$$ which are precisely the morphisms $(\iota^{(i)})_{i\geq 0}: H^i(K_U,\ind^{P^+}_{K_P}(-))\rightarrow H^i(K_U,\ind^P_{K_P}(-))$. Thus, each arrow in the above diagram is a morphism of $\delta$-functors and $\{H^i(K_U^g\cap K_P, -): \Rep_k(K_P)\rightarrow \Ab\}_{i\geq 0}$ is a universal $\delta$-functor, hence it suffices to check that the diagram commutes for $i=0$; this trivially follows from $\pre +j^{(0)}_g(\xi)\coloneqq \sum_{u\in K_U/\pre gK_U}[ug,\xi]$ and $\iota^{(0)}_W([ug,\xi])=[ug,\xi]$.
\end{proof}
\begin{lem}
\label{Exp1Step1}
For any $W\in \Rep_{k}(K_P)$, the map 
$$\alpha: H^i(K_U, \ind^{P^+}_{K_P}W)\overset{\sim}{\longrightarrow} \ind^{M^+}_{K_M}H^i(K_U,W)$$
given by $$\pre +j^{(i)}_{g}(\xi)\longmapsto [\pr_{M}(g),\xi]$$ is an $M^{+}$-equivariant isomorphism, where $g\in P^{+}$ and $\xi \in H^i(K_U^{g}\cap K_P, W)= H^i(K_U,W)$
\end{lem}

\begin{proof}
For $X\in \Rep_k(P^+)$, we first recall that the action of $M^+$ on $H^i(K_U, X)$ is given by the \textit{Hecke action}:
$$H^i(K_U,X)\overset{\mathrm{conj}_m}{\longrightarrow}H^i(\pre mK_U, X)\overset{\cores^{K_U}_{\pre m K_U}}{\longrightarrow}H^i(K_U,X)$$
That this is an action easily follows from the fact that corestriction and conjugation maps commute (\cite[Proposition 1.5.4]{NSW}). We denote this action by $x\mapsto m\star x$. To show that the map $\alpha$ is $M^+$-equivariant, we first prove a slightly more general result:
\begin{claim}
\label{mplusaction}
    $m\star j^{(i)}_g(\xi) = j^{(i)}_{mg}(\cores_{m,g}^{(i)}(\xi))$ for any $g\in K_U\backslash P/K_P$. The same holds true for $\pre +j^{(i)}_{g}$ and $g\in K_U\backslash P^+/K_P$
\end{claim}
\begin{proof}
The $i=0$ case is proved in \cite[Lemma 4.3.7]{heyer2023left}. We then recall that 
$$j_{g}^{(i)}: H^i(K_U^g\cap K_P, -)\rightarrow H^i(K_U, \ind^P_{K_P}(-))$$
is a morphism of $\delta$-functors by Remark \ref{deltafunctorremark}. Moreover, 
$$\mathrm{conj}_m: H^i(K_U, \ind^P_{K_P}(-))\rightarrow H^i(\pre mK_U, \ind^P_{K_P}(-))$$ 
and 
$$\cores^{K_U}_{\pre mK_U}: H^i(\pre mK_U, \ind^P_{K_P}(-))\rightarrow H^i(K_U, \ind^P_{K_P}(-))$$
are also morphisms of $\delta$-functors by \cite[Proposition 1.5.2]{NSW} and Lemma \ref{delta2}. Thus, each arrow in the diagram below is a morphism of $\delta$-functors:
\begin{center}
\begin{tikzcd}
{H^i(K_U^g\cap K_P, -)} \arrow[rr, "j^{(i)}_g"] \arrow[dd, "\cores^{K_U^{mg}\cap K_P}_{K_U^g\cap K_P}"'] &  & {H^i(K_U, \ind^P_{K_P}(-))} \arrow[d, "\mathrm{conj}_m"]                   \\
                                                                                                         &  & {H^i(\pre mK_U, \ind^P_{K_P}(-))} \arrow[d, "\cores^{K_U}_{\pre m K_U}"] \\
{H^i(K_U^{mg}\cap K_P, -)} \arrow[rr, "j^{(i)}_{mg}"']                                                   &  & {H^i(K_U, \ind^P_{K_P}(-))}                                             
\end{tikzcd}
\end{center}
Since $\{H^i(K_U^g\cap K_P, -): \Rep_k(K_P)\rightarrow \Ab\}_{i\geq 0}$ is a universal $\delta$-functor and the above square commutes for $i=0$, we must have that the square commutes for any $i\geq 0$. The statement for $\pre +j^{(i)}_g$ with $g\in K_U\backslash P^+/K_P$ follows from the same argument.
\end{proof}
By the above claim, we see that for $m'\in M^+$, we have $m'\star \pre +j_g^{(i)}(\xi) = \pre +j^{(i)}_{m'g}(\cores^{(i)}_{m',g}(\xi)) = \pre +j^{(i)}_{m'g}(\xi)$ as $\cores^{(i)}_{m',g} = \cores^{K_U^{m'g}\cap K_P}_{K_U^g\cap K_P} = \cores^{K_U}_{K_U} = \id$ for any $g\in P^+$. Thus, $\alpha$ is an $M^+$-equivariant map. We also observe that as $k$-vector spaces, we have the isomorphism
$$\ind^{M^+}_{K_M}H^i(K_U,W) \overset{\sim}{\longrightarrow}\bigoplus_{m\in M^{+}/K_M}H^i(K_U,W)$$
given by $$[m,\xi]\longmapsto (\xi)_m$$
where $(\xi)_m$ denotes the element in the direct sum which is $\xi$ at the $m$-th place and $0$ everywhere else. Since $P^+=K_UM^+$ we have a bijection $K_U\backslash P^+/K_P\cong M^+/K_M$ given by $\pr_{M}$. 
Thus, we get an isomorphism of $k$-vector spaces
$$H^i(K_U,\ind^{P^+}_{K_P}W)\overset{\sim}{\longrightarrow}\bigoplus_{g\in K_U\backslash P^+/K_P}H^i(K_U,W) \overset{\sim}{\longrightarrow}\bigoplus_{m\in M^+/K_M}H^i(K_U,W)\iso \ind^{M^+}_{K_M}H^i(K_U,W)$$
given by $\pre +j^{(i)}_g(\xi)\mapsto (\xi)_{g} \mapsto (\xi)_{\pr_{M}(g)}\mapsto [\pr_{M}(g),\xi]$ which exactly matches with the map $\alpha$. This shows that $\alpha$ must be an isomorphism. 
\end{proof}

\begin{lem}
\label{Exp1Step2}
For any $W\in \Rep_{k}(K_P)$, the map
$$f: \ind^M_{M^+}H^i(K_U,\ind^P_{K_P}W)\longrightarrow \ind^M_{M^+}H^i(K_U,\ind^{P^+}_{K_P}W)$$
given by 
$$1\otimes j^{(i)}_g(\xi)\longmapsto m^{-1}\otimes \pre +j^{(i)}_{mg}(\cores^{(i)}_{m,g}(\xi))$$
is a well-defined $M$-equivariant isomorphism, where $g\in K_U\backslash P/K_P, \xi \in H^i(K_U^g\cap K_P,W)$ and $m\in M^+$ is chosen such that $mg\in P^+$ (Note: the same $m$ works for any double coset representative for $g\in K_U\backslash P/K_P$).
\end{lem}
\begin{proof}
First, we want to show that $f$ is independent of our choice of $m\in M^+$. Let $m_1,m_2\in M^+$ such that $m_1g\in P^+$ and $m_2g\in P^+$. We first claim that it suffices to check that the definition of $f$ is independent of $m$ when $m_2=mm_1$ for some $m\in M^+$. Indeed, we can find $m_1'\in M^+$ such that $m_1'm_1m_2^{-1}\in M^+$ (because $m_1m_2^{-1}\in M$ and \cite[Lemma 3.4.4(a)]{heyer2023left}). So, $m_1'm_1=m_2'm_2$ for some $m_1',m_2'\in M^+$. Now, if $m_2=mm_1$ then we have
\begin{align*}
m_2^{-1}\otimes \pre +j^{(i)}_{m_2g}(\cores^{(i)}_{m_2,g}(\xi))
&= m_1^{-1}m^{-1}\otimes \pre +j^{(i)}_{mm_1g}(\cores^{(i)}_{mm_1,g}(\xi)) \\
&= m_1^{-1}m^{-1}\otimes m\star \pre +j^{(i)}_{m_1g}(\cores^{(i)}_{m_1,g}(\xi))\\ 
&=  m_1^{-1}\otimes \pre +j^{(i)}_{m_1g}(\cores^{(i)}_{m_1,g}(\xi))   
\end{align*}
using Claim \ref{mplusaction} and that $m\in M^+$ and $\pre +j^{(i)}_{m_1g}(\cores^{(i)}_{m_1,g}(\xi))\in H^i(K_U,\ind^{P^+}_{K_P}W)$. This shows that our definition of $f$ does not depend on the choice of $m\in M^+$.\\
We also need to check that $f(1\otimes m'\star j^{(i)}_g(\xi))=f(m'\otimes j^{(i)}_g(\xi))$ for any $m'\in M^+$. Now, using Claim \ref{mplusaction}, the left side is given by
$$f(1\otimes j^{(i)}_{m'g}(\cores^{(i)}_{m',g}(\xi)))= m^{-1}\otimes \pre +j^{(i)}_{mm'g}(\cores^{(i)}_{m,m'g}(\cores^{(i)}_{m',g}(\xi)))=m^{-1}\otimes \pre +j^{(i)}_{mm'g}(\cores^{(i)}_{mm',g}(\xi))$$ by transitivity of corestriction maps. Here $m\in M^+$ is chosen so that $mm'g\in P^+$. For computing the right side, we need to find $m_1\in M^+$ such that $m_1g\in P^+$. Choosing $m_1=mm'\in M^+$ makes the right side become
\begin{align*}
f(m'\otimes j^{(i)}_g(\xi)) = m'\cdot f\left(1\otimes j^{(i)}_g(\xi)\right) &= m'\cdot\left(m_1^{-1}\otimes \pre +j^{(i)}_{m_1g}(\cores^{(i)}_{m_1,g}(\xi))\right)\\
&= m^{-1}\otimes \pre +j^{(i)}_{mm'g}\left(\cores^{(i)}_{mm',g}(\xi)\right),
\end{align*}
giving our desired equality. To show that $f$ is an isomorphism, we construct an inverse explicitly. Recall that the inclusion $\ind^{P^+}_{K_P}(W)\subseteq \ind^P_{K_P}(W)$ induces a map on cohomology 
$$\iota^{(i)}_W: H^i(K_U, \ind^{P^+}_{K_P}(W))\longrightarrow H^i(K_U,\ind^P_{K_P}(W))$$ which satisfies $$\iota^{(i)}_W(\pre +j^{(i)}_{g}(\xi))=j^{(i)}_{g}(\xi)$$ for any $g\in P^+, \xi \in H^i(K_U,W)$ by Lemma \ref{jplus}. Now, Claim \ref{mplusaction} implies that the morphism $\iota^{(i)}_{W}$ is $M^+$-equivariant.
This further induces the map 
$$\phi: \ind^M_{M^+}H^i(K_U, \ind^{P^+}_{K_P}(W))\longrightarrow \ind^M_{M^+}H^i(K_U,\ind^P_{K_P}(W))$$ given by
$$ 1\otimes \pre +j^{(i)}_g(\xi)\longmapsto 1\otimes j^{(i)}_{g}(\xi)$$ which is $M$-equivariant. We claim that $f$ and $\phi$ are mutual inverses. Indeed, 
$$f\circ \phi(1\otimes \pre +j^{(i)}_g(\xi))= f(1\otimes j^{(i)}_g(\xi))=m^{-1}\otimes \pre +j^{(i)}_{mg}(\cores^{(i)}_{m,g}(\xi))$$
where $m\in M^+$ such that $mg\in P^+$. Since $g\in P^+$, we can choose $m=1$, in which case we get $f\circ \phi(1\otimes \pre +j^{(i)}_{g}(\xi))=1\otimes \pre +j^{(i)}_g(\cores^{(i)}_{1,g}(\xi))=1\otimes \pre +j^{(i)}_g(\xi)$ as 
$$\cores^{(i)}_{1,g}=\cores^{K_U^{g}\cap K_P}_{K_U^{g}\cap K_P}:H^i(K_U^g\cap K_P,W)\rightarrow H^i(K_U^g\cap K_P,W)=\mathrm{Id}.$$
Also, 
\begin{align*}
\phi\circ f(1\otimes j^{(i)}_g(\xi))=\phi(m^{-1}\otimes \pre +j^{(i)}_{mg}(\cores^{(i)}_{m,g}(\xi)))&=m^{-1}\cdot\phi(1\otimes \pre +j^{(i)}_{mg}(\cores^{(i)}_{m,g}(\xi)))\\
&=m^{-1}\otimes (m\star j^{(i)}_g(\xi))=1\otimes j^{(i)}_g(\xi)
\end{align*}
because $j^{(i)}_g(\xi)\in H^i(K_U,\ind^{P}_{K_P}(W))$.
\end{proof}
We are now ready to prove Proposition \ref{Exp1}.
\begin{proof}[Proof of Proposition \ref{Exp1}] We note that this proposition is proved for $i=0$ in \cite[Proposition 4.3.8]{heyer2023left}. We consider the $M$-equivariant isomorphism
$$\ind^M_{M^+}(\alpha): \ind^M_{M^+}H^i(K_U,\ind^{P^+}_{K_P}W)\overset{\sim}{\longrightarrow}\ind^M_{M^+}\ind ^{M^+}_{K_M}H^i(K_U,W)\cong \ind^M_{K_M}H^i(K_U,W)$$
given by 
$$1\otimes \pre +j^{(i)}_{g}(\xi)\longmapsto [\pr_{M}(g),\xi]$$ for $g\in P^+$ and $\xi \in H^i(K_U^g\cap K_P,W)=H^i(K_U,W)$ coming from Lemma \ref{Exp1Step1} and $$f: \ind^M_{M^+}H^i(K_U,\ind^P_{K_P}W)\overset{\sim}{\longrightarrow}\ind^M_{M^+}H^i(K_U,\ind^{P^+}_{K_P}W)$$ coming from Lemma \ref{Exp1Step2}. For $g\in K_U\backslash P/K_P$ and $\xi\in H^i(K_U^g\cap K_P,W)$, we compute the composite
\begin{align*}
(\ind^M_{M^+}(\alpha)\circ f)(1\otimes j^{(i)}_g(\xi))
&=(\ind^M_{M^+}(\alpha))(m^{-1}\otimes \pre +j^{(i)}_{mg}(\cores^{(i)}_{m,g}(\xi))) \\
&= m^{-1}\cdot[\pr_{M}(mg),\cores^{(i)}_{m,g}(\xi)]\\ 
&=m^{-1}\cdot[m\pr_{M}(g),\cores^{(i)}_{m,g}(\xi)]\\&=[\pr_{M}(g),\cores^{(i)}_{m,g}(\xi)]\\
&=[\pr_{M}(g),\cores^{(i)}_{g}(\xi)].
\end{align*}
Thus, the composite $\kappa\coloneqq  \ind^{M}_{M^+}(\alpha)\circ f: \ind^M_{M^+}H^i(K_U,\ind^P_{K_P}W)\rightarrow \ind^M_{K_M}H^i(K_U,W)$ provides the required $M$-equivariant isomorphism.
\end{proof}
We note that the inverse of the isomorphism in Proposition \ref{Exp1} is given by
$$\kappa^{-1}: \ind^M_{K_M}H^i(K_U,W)\overset{\sim}{\longrightarrow} \ind^M_{M^+}H^i(K_U,\ind^P_{K_P}W)$$
$$[m,\xi]\longmapsto m\otimes j^{(i)}_1(\xi)$$ for $m\in M$ and $\xi\in H^i(K_U,W)$.

We can now give an explicit formula for the higher mod $p$ Satake transform, which should be compared with the analogous formula given in the proof of \cite[Theorem 4.3.2(b)]{heyer2023left}.

\begin{prop}\label{prop-explicit-formula} For $V\in \Rep_k(K)$, the higher mod $p$ Satake transform
$$
\mathcal{S}_V^i: \mathcal{H}_G(V) \longrightarrow \mathcal{H}_M(H^i(K_U, V))
$$
is explicitly given by 
\begin{align}\label{explicit-formula} 
\mathcal{S}_V^i(\Phi) ([1,\xi]) = \sum_{g \in K_P \backslash P / K_U} [\pr_M (g^{-1}), \cores^{(i)}_{g^{-1}} ( \pr_{g^{-1}}^{(i)} (\Phi[1, \xi ] ) ) ]
\end{align}
for $i \in \{0, \dots, \mathrm{dim}(U)\}.$
    
\end{prop}
Here, $\Phi\in \mathcal{H}_G(V)=\End_G(\ind^G_K V)$, and $\xi \in H^i(K_U,V)$ which gives us an element 
$$[1,\xi]\in \ind^M_{K_M}H^i(K_U,V).$$
By a slight abuse of notation, we write
$$\Phi([1,\xi])\coloneqq H^i(\Phi_1)\circ H^i(\eta)(\xi)$$
where $\eta: V\rightarrow \ind^P_{K_P}V$ is the $K_P$-equivariant map $x\mapsto [1,x]$ and $\Phi_1: \ind^P_{K_P}V\rightarrow \ind^P_{K_P}V$ is induced by the Iwasawa decomposition ($G=PK$) from the map $\Phi: \ind^G_KV\rightarrow \ind^G_K V$ (by noting that $\Res^G_P\ind^G_K\cong \ind^P_{K_P}\Res^K_{K_P}$). That is, we have the composition
$$H^i(\Phi_1)\circ H^i(\eta): H^i(K_U,V)\rightarrow H^i(K_U, \ind^P_{K_P}V)\rightarrow H^i(K_U,\ind^P_{K_P}V)$$ 
and hence $\Phi([1,\xi])\in H^i(K_U,\ind^P_{K_P}V)$, meaning that $\pr^{(i)}_{g^{-1}}(\Phi([1,\xi]))\in H^i(K_U^{g^{-1}}\cap K_P,V)$.  This implies that $\cores^{(i)}_{g^{-1}}(\pr^{(i)}_{g^{-1}}(\Phi([1,\xi])))\in H^i(K_U,V)$, so 
$$[\pr_{M}(g^{-1}), \cores^{(i)}_{g^{-1}}(\pr^{(i)}_{g^{-1}}(\Phi([1,\xi])))]\in \ind^M_{K_M}H^i(K_U,V),$$ 
showing that the right side of our formula in the above proposition is well-defined.

Before proving the proposition, we make an observation:
\begin{lem}\label{eta_lemma}
The map $j^{(i)}_1: H^i(K_U,V)\hookrightarrow H^i(K_U, \ind^P_{K_P}V)$ is exactly the map $H^i(\eta)$ induced by $\eta:V\rightarrow \ind^P_{K_P}V$.
\end{lem}
\begin{proof}
First, we note that $j^{(0)}_1(\xi)=H^0(\eta)(\xi)=[1,\xi]$ for any $\xi \in H^0(K_U,V)$. This follows from the explicit description of $j^{(i)}_g$ provided in Remark \ref{mackey-on-cochains}. Now, we will apply Lemma \ref{delta3} to show that $\{H^i(\eta)\}_{i\geq 0}$ is a morphism of $\delta$-functors.  Consider the following four functors:
\begin{align*}
F_1=F_2&:\Rep_k(K_U)\to \Ab,\qquad\qquad V\mapsto V^{K_U}=H^0(K_U,V);\\
G_1\coloneqq \Res_{K_U}^{K_P}{(-)}&:\Rep_k(K_P)\to \Rep_k(K_U);\\
G_2\coloneqq \Res_{K_U}^{P}\ind ^P_{K_P}(-)&:\Rep_k(K_P)\to \Rep_k(K_U).
\end{align*}    
Note that $F_1,F_2$ are left-exact, $G_1,G_2$ are exact, and $G_2$ sends injective objects to $F_1=F_2$ acyclic objects (as in the proof of Lemma \ref{jplus}). Similarly, $G_1$ also sends injective objects to $F_1=F_2$-acyclic objects (actually $G_1$ preserves injectives since it is right adjoint to an exact functor $\ind^{K_P}_{K_U}$). Also, the natural transformation $\eta: \mathrm{Id}\rightarrow \Res^{P}_{K_P}\ind^{P}_{K_P}$ induces a natural transformation $$G_1=\Res_{K_U}^{K_P} {(-)}\rightarrow \Res_{K_U}^{K_P}\Res^{P}_{K_P}\ind^{P}_{K_P} {(-)}=\Res_{K_U}^{P} \ind^{P}_{K_P}{(-)}=G_2$$ which further induces a natural transformation $\tilde{\eta}:F_1\circ G_1\rightarrow F_2\circ G_2$. By Lemma \ref{delta3}, we get an induced morphism of $\delta$-functors
$$\mathrm{R}^i\tilde{\eta}: \mathrm{R}^iF_1\circ G_1\longrightarrow \mathrm{R}^iF_2\circ G_2$$ which are precisely the morphisms $(H^i(\eta))_{i\geq 0}: H^i(K_U,-)\rightarrow H^i(K_U,\ind^P_{K_P}(-))$. Also, we have seen that $\{j^{(i)}_1\}_{i\geq 0}$ are also morphisms of $\delta$-functors. Since these two morphisms of $\delta$-functors agree for $i=0$ and $\{H^i(K_U,-):\Rep_k(K_P)\rightarrow \Ab\}_{i\geq 0}$ is a universal $\delta$-functor, we must have $j^{(i)}_1=H^i(\eta)$ for all $i\geq 0$. 
\end{proof}
\begin{proof}[Proof of Proposition \ref{prop-explicit-formula}]
Recall for clarity that our definition of $\Phi_1$ can be repackaged as saying that the following diagram commutes:
$$\begin{tikzcd}
\Res^G_P\ind^G_K V \arrow[rr,swap]{ur}{\cong}[swap]{\mathrm{Iwasawa}} \arrow[dd, "\Phi"']&  & \ind^P_{K_P}\Res^K_{K_P}V \arrow[dd, "\Phi_1"]\\
\\
\Res^G_P\ind^G_K V \arrow[rr,swap]{ur}{\cong}[swap]{\mathrm{Iwasawa}}                    &  & \ind^P_{K_P}\Res^K_{K_P}V
\end{tikzcd}$$

We also have by Lemma \ref{lem-kappa&satake} that the following diagram commutes:
\begin{equation}\label{SVi-comm-diagram}\begin{tikzcd}
{\ind^M_{M^+}H^i(K_U,\ind^P_{K_P}V)} \arrow[dd, "\ind^M_{M^+}H^i(\Phi_1)"'] \arrow[rr,swap]{ur}{\cong}[swap]{\kappa} &  & {\ind^M_{K_M}H^i(K_U,V)} \arrow[dd, "\mathcal{S}^i_V(\Phi)"] \\
\\
{\ind^M_{M^+}H^i(K_U,\ind^P_{K_P}V)} \arrow[rr,swap]{ur}{\cong}[swap]{\kappa}                                        &  & {\ind^M_{K_M}H^i(K_U,V)} 
\end{tikzcd}\end{equation}


Now, given $[1,\xi]\in \ind^M_{K_M}H^i(K_U,V)$, we get $\kappa^{-1}([1,\xi])=1\otimes j^{(i)}_1(\xi)$. Then 
$$\ind^M_{M^+}H^i(\Phi_1)(1\otimes j^{(i)}_1(\xi))=1\otimes H^i(\Phi_1)(j^{(i)}_1(\xi))=1\otimes H^i(\Phi_1)\circ H^i(\eta)(\xi)=1\otimes \Phi([1,\xi]),$$
where the penultimate equality holds by Lemma \ref{eta_lemma}. For any $x\in H^i(K_U, \ind^P_{K_P}V)$, by our Mackey decomposition (with $g$ replaced by $g^{-1}$ and double cosets $K_UgK_P$ and $K_Pg^{-1}K_U$ switched appropriately), we have 
$$x = \sum_{g\in K_P\backslash P/K_U} j^{(i)}_{g^{-1}}\circ \pr^{(i)}_{g^{-1}}(x).$$
Thus, after applying this to $x=\Phi([1,\xi])$, we get
$$\ind^M_{M^+}H^i(\Phi_1)(1\otimes j^{(i)}_1(\xi))=\sum_{g\in K_P\backslash P/K_U}1\otimes j^{(i)}_{g^{-1}}\circ \pr^{(i)}_{g^{-1}}(\Phi([1,\xi])).$$ 
Applying $\kappa$ to both sides then yields
\begin{align*}
\kappa \circ \ind^M_{M^+}H^i(\Phi_1)\circ \kappa^{-1}([1,\xi])&=\kappa(\ind^M_{M^+}H^i(\Phi_1)(1\otimes j^{(i)}_1(\xi))) \\
&=\kappa\left(\sum_{g\in K_P\backslash P/K_U}1\otimes j^{(i)}_{g^{-1}}\circ \pr^{(i)}_{g^{-1}}(\Phi([1,\xi]))\right) \\
&=\sum_{g\in K_P\backslash P/K_U}\kappa\left(1\otimes j^{(i)}_{g^{-1}}\circ \pr^{(i)}_{g^{-1}}(\Phi([1,\xi]))\right)\\
&= \sum_{g\in K_P\backslash P/K_U}[\pr_M(g^{-1}),\cores^{(i)}_{g^{-1}}(\pr^{(i)}_{g^{-1}}(\Phi([1,\xi])))].
\end{align*}
Since we have $\mathcal{S}^i_V(\Phi)=\kappa\circ \ind^M_{M^+}H^i(\Phi_1)\circ \kappa^{-1}$ by diagram \ref{SVi-comm-diagram}, we get the required formula.     
\end{proof}

\subsection{The cochain-level formula for \texorpdfstring{$\SVi$}{S(V,i)}} \label{cochain-subsec}

Recall from the previous two subsections that for any fixed $g\in P$, we have morphisms of $\delta$-functors $\Rep_k(K_P)\to \Ab$
$$\pr_{g\inv}^{(i)}:H^i\left(K_U,\Res_{K_U}^{P}\ind_{K_P}^P (-)\right)\to H^i\left(K_U^{g\inv}\cap K_P,\Res^{K_P}_{K_U^{g\inv}\cap K_P}(-)\right),$$
$$\cores_{g\inv}^{(i)}:H^i\left(K_U^{g\inv}\cap K_P,\Res^{K_P}_{K_U^{g\inv}\cap K_P}(-)\right)\to H^i\left(K_U,\Res_{K_U}^{K_P}(-)\right).$$

Composing yields a morphism of $\delta$-functors $\Rep_k(K_P)\to \Ab$
$$\cores_{g\inv}^{(i)}\pr_{g\inv}^{(i)}:H^i\left(K_U,\ind_{K_P}^P (-)\right)\to H^i\left(K_U,-\right).$$

\begin{lem}
\label{KMnote}
The maps $H^i\left(K_U,-\right)$ in fact form a $\delta$-functor $\Rep_k(K_P)\to \Rep_k(K_M)$.
\end{lem}

\begin{proof}
Here, for any $V\in \Rep_k(K_P)$, $K_M\iso K_P/K_U$ acts on $H^i(K_U,V)$ in the usual way via the conjugation map $\mathrm{conj}^{(i)}_m:H^i(K_U,V)\to H^i(K_U,V)$; see \cite[Proposition 1.6.3]{NSW} for details.  To check that this gives a $\delta$-functor $\Rep_k(K_P)\to \Rep_k(K_M)$, we note that $\mathrm{conj}^{(i)}_m$ is itself a morphism of $\delta$-functors $\Rep_k(K_P)\to \Ab$ by \cite[Proposition 1.5.2]{NSW}, so the following diagram commutes for any short exact sequence $0\to V''\to V\to V'\to 0$ in $\Rep_k(K_P)$ and any $m\in K_M$:
\[\begin{tikzcd}[cramped]
	\cdots & {H^i(K_U,V'')} & {H^i(K_U,V)} & {H^i(K_U,V')} & {H^{i+1}(K_U,V'')} & \cdots \\
	\cdots & {H^i(K_U,V'')} & {H^i(K_U,V)} & {H^i(K_U,V')} & {H^{i+1}(K_U,V'')} & \cdots
	\arrow[from=1-1, to=1-2]
	\arrow[from=1-2, to=1-3]
	\arrow["{\mathrm{conj}^{(i)}_m}"', from=1-2, to=2-2]
	\arrow[from=1-3, to=1-4]
	\arrow["{\mathrm{conj}^{(i)}_m}"', from=1-3, to=2-3]
	\arrow["\delta", from=1-4, to=1-5]
	\arrow["{\mathrm{conj}^{(i)}_m}"', from=1-4, to=2-4]
	\arrow[from=1-5, to=1-6]
	\arrow["{\mathrm{conj}^{(i+1)}_m}"', from=1-5, to=2-5]
	\arrow[from=2-1, to=2-2]
	\arrow[from=2-2, to=2-3]
	\arrow[from=2-3, to=2-4]
	\arrow["\delta", from=2-4, to=2-5]
	\arrow[from=2-5, to=2-6]
\end{tikzcd}\]
This exactly says that the action of $m\in K_M$ commutes with every map in the long exact sequence, meaning that the maps $H^i(K_U,-):\Rep_k(K_P)\to \Rep_k(K_M)$ indeed form a $\delta$-functor. 
\end{proof}

We now consider the composition of our morphism of $\delta$-functors 
$$\cores_{g\inv}^{(i)}\pr_{g\inv}^{(i)}:H^i\left(K_U,\ind_{K_P}^P(-)\right)\to H^i(K_U,-)$$
with the natural maps $H^i(K_U,-)\to \ind_{K_M}^M H^i(K_U,-)$ sending $\xi\mapsto \pr_M(g\inv)\otimes \xi$.

\begin{lem}
The maps 
$$\pr_M(g\inv)\otimes \cores_{g\inv}^{(i)}\pr_{g\inv}^{(i)}(-):H^i\left(K_U,\ind_{K_P}^P(-)\right)\to \ind_{K_M}^M H^i(K_U,-)$$
form a morphism of universal $\delta$-functors $\Rep_k(K_P)\to \Ab$.
\end{lem}

\begin{proof}
One can show that these are universal $\delta$-functors $\Rep_k(K_P)\to \Ab$ using the same form of argument as in the proof of Remark \ref{deltafunctorremark}.

Let now
$$0\to V''\to V\to V'\to 0$$
be a short exact sequence in $\Rep_k(K_P)$.  Because the original maps $\cores_{g\inv}^{(i)}\pr_{g\inv}^{(i)}$ formed a morphism of $\delta$-functors $\Rep_k(K_P)\to \Ab$, we know that the following diagram of morphisms in $\Ab$ commutes:
\[\begin{tikzcd}[cramped]
	{H^i\left(K_U,\ind_{K_P}^P V'\right)} & {H^{i+1}\left(K_U,\ind_{K_P}^P V''\right)} \\
	{H^i(K_U,V')} & {H^{i+1}(K_U,V'').}
	\arrow["\delta\circ \ind_{K_P}^P", from=1-1, to=1-2]
	\arrow["{\cores_{g\inv}^{(i)}\pr_{g\inv}^{(i)}}"', from=1-1, to=2-1]
	\arrow["{\cores_{g\inv}^{(i+1)}\pr_{g\inv}^{(i+1)}}", from=1-2, to=2-2]
	\arrow["\delta", from=2-1, to=2-2]
\end{tikzcd}\]
We now apply Lemma \ref{delta4} with the exact functors $F:\Rep_k(K_M)\to \Ab$ given by the forgetful functor, $G:\Rep_k(K_M)\to \Ab$ the composition of $\ind_{K_M}^M:\Rep_k(K_M)\to \Rep_k(M)$ with the forgetful functor $\Rep_k(M)\to \Ab$, the natural transformation $\eta:F\to G$ given by $\eta_W:W\to \ind_{K_M}^M W$, $w\mapsto \pr_M(g\inv)\otimes w$ for $W\in \Rep_k(K_M)$, and $H=H^i(K_U,-)$ our $\delta$-functor $\Rep_k(K_P)\to \Rep_k(K_M)$ (using Lemma \ref{KMnote}).  This demonstrates that the following diagram commutes in $\Ab$:
\[\begin{tikzcd}[cramped]
	{H^i(K_U,V')} & {H^{i+1}(K_U,V'')} \\
	{\ind_{K_M}^M H^i(K_U,V')} & {\ind_{K_M}^M H^{i+1}(K_U,V'').}
	\arrow["\delta", from=1-1, to=1-2]
	\arrow["{\pr_M(g\inv)\otimes(-)}"', from=1-1, to=2-1]
	\arrow["{\pr_M(g\inv)\otimes(-)}", from=1-2, to=2-2]
	\arrow["{\ind_{K_M}^M\circ \delta}", from=2-1, to=2-2]
\end{tikzcd}\]
Combining the two commuting diagrams above gives the commutative diagram needed to verify that the maps
$$\pr_M(g\inv)\otimes \cores_{g\inv}^{(i)}\pr_{g\inv}^{(i)}(-):H^i\left(K_U,\ind_{K_P}^P(-)\right)\to \ind_{K_M}^M H^i(K_U,-)$$
form a morphism of universal $\delta$-functors $\Rep_k(K_P)\to \Ab$:
\[\begin{tikzcd}[cramped]
	{H^i\left(K_U,\ind_{K_P}^P V'\right)} & {H^{i+1}\left(K_U,\ind_{K_P}^P V''\right)} \\
	{\ind_{K_M}^M H^i(K_U,V')} & {\ind_{K_M}^M H^{i+1}(K_U,V'').}
	\arrow["\delta", from=1-1, to=1-2]
	\arrow["{\pr_M(g\inv)\otimes \cores_{g\inv}^{(i)}\pr_{g\inv}^{(i)}}"', from=1-1, to=2-1]
	\arrow["{\pr_M(g\inv)\otimes \cores_{g\inv}^{(i+1)}\pr_{g\inv}^{(i+1)}}", from=1-2, to=2-2]
	\arrow["{\ind_{K_M}^M \circ \delta}", from=2-1, to=2-2]
\end{tikzcd}\]
\end{proof}

For any $V\in \Rep_k(K_P)$ and any $\xi\in H^i(K_U,\ind_{K_P}^P V)$, we know that $\pr_{g\inv}^{(i)}\xi$ is zero for all but finitely many double coset representatives $g\in K_P\backslash P/K_U$, hence we also get a well-defined morphism of $\delta$-functors $\Rep_k(K_P)\to \Ab$ given by 
$$\sum_{g\in K_P\backslash P/K_U} \pr_M(g\inv)\otimes \cores_{g\inv}^{(i)}\pr_{g\inv}^{(i)}(-):H^i\left(K_U,\ind_{K_P}^P(-)\right)\to \ind_{K_M}^M H^i(K_U,-).$$
Fix any $m\in M$, so that evaluating at $m$ yields similarly by Lemma \ref{delta4} a morphism of $\delta$-functors $\Rep_k(K_P)\to \Ab$,
$$\ind_{K_M}^M H^i(K_U,-)\xto{\mathrm{ev}_m} H^i(K_U,-).$$
Composing these yields a morphism of $\delta$-functors $\Rep_k(K_P)\to \Ab$
$$H^i\left(K_U,\ind_{K_P}^P (-)\right)\to H^i(K_U,-).$$
\begin{lem}
\label{lemma-explicit-morphism}
The above construction gives a morphism of universal $\delta$-functors $\Rep_k(K_P)\to \Ab$
$$H^i\left(K_U,\ind_{K_P}^P (-)\right)\to  H^i(K_U,-)$$
which is given when $i=0$ for any $V\in \Rep_k(K_P)$ by the well-defined map
$$H^0(K_U,\ind_{K_P}^P V)\to H^0(K_U,V),$$
$$f\mapsto \sum_{u\in U/K_U} f(mu).$$
\end{lem}

\begin{proof}
First, we note that one can prove that these are universal $\delta$-functors by the same form of argument as in the proof of Remark \ref{deltafunctorremark}.

Next, note that for any fixed $f\in H^0(K_U,\ind_{K_P}^P V)$, $f$ has compact support on $P$. Since $K_U$ is open in $U$, this sum is nonzero for only finitely many cosets $u\in U/K_U$.

After unwinding our construction, the claim that this morphism of universal $\delta$-functors is given by the claimed formula when $i=0$ is precisely provided in the proof of \cite[Theorem 4.3.2(b)]{heyer2023left}, so we omit it.
\end{proof}

We seek to obtain a similarly formed concrete description of this morphism of $\delta$-functors above the level of $i=0$.  To do this, we introduce a collection of maps which will ultimately be shown to form a morphism of $\delta$-functors agreeing with the one considered above.  The definition should be compared to the usual definition of corestriction; see \ref{cores-on-cochains} above.

\begin{definition}
\label{def-delta-functor}
Fix any $i\geq 0$, $V\in \Rep_k(K_P)$, and $m\in M$.  For each coset $u\in K_U\backslash U$, fix a representative $\ol u\in u$ once and for all.  Define
$$C_m^{(i)}:X^i(K_U,\ind_{K_P}^P V)^{K_U}\to X^i(K_U,V)^{K_U}$$
 for any $x\in X^i(K_U,\ind_{K_P}^P V)^{K_U}$ and any $\sigma_0,\ldots, \sigma_i\in K_U$ by
$$(C_m^{(i)}x)(\sigma_0,\ldots, \sigma_i)\coloneqq \sum_{u\in K_U\backslash U} x\left(\ol u m\inv \sigma_0 m(\ol{u m\inv \sigma_0 m})\inv,\ldots, \ol u m\inv\sigma_i m(\ol{u m\inv \sigma_i m})\inv\right)(m\ol u\inv).$$
\end{definition}

\begin{lem}\label{lem-Cmi-maps-of-cohom}
The maps $C_m^{(i)}$ are well-defined and induce maps
$$C_m^{(i)}:H^i(K_U,\ind_{K_P}^P V)\to H^i(K_U,V)$$
at the level of cohomology.
\end{lem}

\begin{proof}
Fix any $x\in X^i(K_U,\ind_{K_P}^P V)^{K_U}$.  We first justify why we can evaluate $x$ at the point of $P^{i+1}$ in the definition, i.e., why each of its entries actually lies in $K_U$.  First, $U\leq P$ being a normal subgroup means that $m\inv\sigma_jm\in U$ for $\sigma_j\in K_U$.  Now the terms $\ol u m\inv \sigma_j m$ and $\ol{u m\inv\sigma_jm}$ represent the same coset in $K_U\backslash U$, from which it follows that indeed
$$\ol u m\inv\sigma_j m(\ol{u m\inv \sigma_j m})\inv$$
lies in $K_U$ for each $j$.  It is now sufficient to verify the following three facts.
\begin{enumerate}[label = (\roman*)]
\item For our fixed $x$, there are only finitely many terms in this sum which are nonzero.
\item $C_m^{(i)}x$ is $K_U$-invariant, i.e., $\sigma\cdot (C_m^{(i)}x)=C_m^{(i)}x$ for all $\sigma\in K_U$.
\item The $C_m^{(i)}$ form a chain map, i.e., $dC_m^{(i)}=C_m^{(i+1)}d$.
\end{enumerate}
We will prove these in sequence.
\begin{enumerate}[label = (\roman*)]
\item Our fixed $x$ is a continuous function $x:K_U^{i+1}\to \ind_{K_P}^P V$ where $\ind_{K_P}^P V$ is endowed with the discrete topology.  Thus, the preimage of any set under $x$ will be open in $K_U^{i+1}$.  For each point $\sigma_\bullet=(\sigma_0,\ldots, \sigma_i)\in K_U^{i+1}$, we know that $x(\sigma_\bullet)\in \ind_{K_P}^P V$ has compact support $S_{\sigma_\bullet}\subseteq P$.  Let $\mc U_{\sigma_\bullet}\subseteq K_U^{i+1}$ be the open neighborhood of $\sigma_\bullet$ consisting of points whose image under $x$ has support contained in $S_{\sigma_\bullet}$, giving us an open cover $\{\mc U_{\sigma_\bullet}\}_{\sigma_\bullet\in K_U^{i+1}}$ of $K_U^{i+1}$.  By compactness, this has a finite subcover, producing for us a compact set $S\subseteq P$ such that $\supp(x(\sigma_\bullet))\subseteq S$ for every $\sigma_\bullet\in K_U^{i+1}$.  Now $(m\inv\cdot S)\cap U$ is compact in $U$, so it follows from $K_U\leq U$ being open that there are only finitely many cosets $u\in K_U\backslash U$ with $x(\sigma_\bullet)(m\ol u\inv )\neq 0$ for any $\sigma_\bullet\in K_U^{i+1}$.  Thus, only finitely many terms of our sum are nonzero for each $x\in X^i(K_U,\ind_{K_P}^P V)^{K_U}$.

\item Say $\sigma,\sigma_0,\ldots, \sigma_i\in K_U$.  Now for each coset $u\in K_U\backslash U$, $\ol u(m\inv \sigma\inv m)=\tau_{\sigma,u}\cdot\left(\ol{u(m\inv \sigma\inv m)}\right)$ for some $\tau_{\sigma,u}\in K_U$, so we can compute that
\begin{align*}
(\sigma\cdot& C_m^{(i)} x)(\sigma_0,\ldots, \sigma_i)\\
&= \sigma(C_m^{(i)}x)(\sigma\inv\sigma_0,\ldots, \sigma\inv \sigma_i)\\
&= \sum_{u\in K_U\backslash U} \sigma x\left(\ol u m\inv \sigma\inv\sigma_0 m(\ol{u m\inv \sigma\inv\sigma_0 m})\inv,\ldots\right)(m\ol u\inv)\\
&= \sum_{u\in K_U\backslash U} x\left(\ol u (m\inv \sigma\inv m)m\inv \sigma_0 m\left(\ol{u (m\inv \sigma\inv m)m\inv \sigma_0 m}\right)\inv,\ldots\right)(m(m\inv \sigma m)\ol u\inv)
\end{align*}
where the last equality follows from elements of $\ind_{K_P}^P V$ being left $K_P$-equivariant.  Using our $\tau_{\sigma,u}$ and then invoking the fact that $x\in X^i(K_U,\ind_{K_P}^P V)^{K_U}$ is $K_U$-invariant, this equals
\begin{align*}
\sum_{u\in K_U\backslash U}\! x\left(\tau_{\sigma,u} (\ol{u(m\inv\sigma \inv m)}) m\inv \sigma_0 m(\ol{u (m\inv \sigma\inv m)m\inv \sigma_0 m})\inv,\ldots\right)\left(m(\ol{u(m\inv \sigma\inv m) })\inv\tau_{\sigma,u}\inv\right)\\
=\sum_{u\in K_U\backslash U} x\left( (\ol{u(m\inv\sigma \inv m)}) m\inv \sigma_0 m(\ol{u (m\inv \sigma\inv m)m\inv \sigma_0 m})\inv,\ldots\right)\left(m(\ol{u(m\inv \sigma\inv m) })\inv\right).
\end{align*}
Since $u\mapsto u(m\inv\sigma \inv m)$ is a bijection on $K_U\backslash U$, we can do a simple substitution to show that the above equals
\begin{align*}
\sum_{u\in K_U\backslash U} &x\left(\ol u m\inv \sigma_0 m(\ol{u m\inv \sigma_0 m})\inv,\ldots, \ol u m\inv\sigma_i m(\ol{u m\inv \sigma_i m})\inv\right)(m\ol u\inv )\\
&=(C_m^{(i)}x)(\sigma_0,\ldots, \sigma_i).
\end{align*}
\item One can immediately check from the definitions of the maps $C_m^{(i)}$ and the differentials that $dC_m^{(i)}=C_m^{(i+1)}d$; there are no nontrivial steps in the computation (see \cite[pg. 12]{NSW} for a similar computation).
\end{enumerate}
\end{proof}

Note that the maps $C_m^{(i)}$ are functorial with respect to $V\in \Rep_k(K_P)$, so the following proposition makes sense.

\begin{prop}\label{prop-Cmi-univ-delta-functors}
The maps $C_m^{(i)}$ form a morphism of universal $\delta$-functors $\Rep_k(K_P)\to \Ab$,
$$C_m^{(i)}:H^i\left(K_U, \Res_{K_U}^P\ind_{K_P}^P(-)\right)\to H^i\left(K_U,\Res_{K_U}^{K_P}(-)\right).$$
\end{prop}

\begin{proof}
This argument is completely standard, and is identical to that given in the proof of \cite[Proposition 1.5.2]{NSW}.  We thus omit it.
\end{proof}

\begin{cor} \label{cor-independence}
The morphism of $\delta$-functors $C_m^{(i)}:H^i(K_U,\ind_{K_P}^P(-))\to H^i(K_U,-)$ is independent of the choices of coset representatives made in Definition \ref{def-delta-functor}.
\end{cor}

\begin{proof}
When $i=0$, the construction of $C_m^{(0)}$ ensures that it is explicitly given for any $V\in \Rep_k(K_P)$ by the map
$$C_m^{(0)}:H^0(K_U,\ind_{K_P}^P V)\to H^0(K_U,V),$$
$$f\mapsto \sum_{u\in K_U\backslash U} f(m\ol u\inv).$$
Fix any $u\in K_U\backslash U$, and suppose that $\tilde u\in u$ is any other representative of this coset.  Then $\tilde u=\sigma \ol u$ for some $\sigma\in K_U$, hence $f(m\tilde u\inv)=f(m\ol u\inv \sigma\inv)=(\sigma\inv\cdot f)(m\ol u\inv)$.  Since $f\in (\ind_{K_P}^P V)^{K_U}$ is $K_U$-invariant, this means that 
$$f(m\tilde u\inv)=(\sigma\inv\cdot f)(m\ol u\inv)=f(m\ol u\inv),$$
from which it immediately follows that $C_m^{(0)}$ is independent of the choice of coset representatives.  Because different choices of coset representatives give morphisms of the same universal $\delta$-functors which agree when $i=0$, it follows that they agree on all $i$.
\end{proof}

Altogether, these results prove the following:

\begin{theorem}
\label{thm-explicit-formula}
The higher mod $p$ Satake transform
$$\SVi:\mc H_G(V)\to \mc H_M(H^i(K_U,V))$$
is explicitly given by
$$\SVi(\Phi)([1,\xi])(m)=C_m^{(i)}(\Phi([1,\xi]))$$
for $i\in \{0,\ldots, \dim U\}$ for all $\Phi\in \mc H_G(V)$, $\xi\in H^i(K_U,V)$, and $m\in M$, where again $\Phi([1,\xi])$ is shorthand for $H^i(\Phi_1)\circ H^i(\eta)(\xi)\in H^i(K_U,\ind_{K_P}^P V)$ as described after Proposition \ref{prop-explicit-formula}.
\end{theorem}

\begin{proof}
Let 
$$D_m^{(i)}:H^i\left(K_U,\ind_{K_P}^P(-)\right)\to H^i(K_U,-)$$
denote the morphism of universal $\delta$-functors described in Lemma \ref{lemma-explicit-morphism}, which we noted was given when $i=0$ by
$$H^0(K_U,\ind_{K_P}^P V)\to H^0(K_U,V),$$
$$f\mapsto \sum_{u\in U/K_U} f(mu)$$
for any $V\in \Rep_k(K_P)$.  By the construction of $D_m^{(i)}$, Proposition \ref{prop-explicit-formula} tells us that $\SVi(\Phi)([1,\xi])(m)=D_m^{(i)}(\Phi([1,\xi]))$.  As noted in the previous corollary, directly from its explicit description, the morphism of universal $\delta$-functors
$$C_m^{(i)}:H^i\left(K_U,\ind_{K_P}^P(-)\right)\to H^i(K_U,-)$$
can be described when $i=0$ by
$$H^0(K_U,\ind_{K_P}^P V)\to H^0(K_U,V),$$
$$f\mapsto \sum_{u\in K_U\backslash U} f(mu\inv).$$
Clearly $u\mapsto u\inv$ induces a bijection $U/K_U\to K_U\backslash U$, showing that $C_m^{(0)}$ and $D_m^{(0)}$ agree as functions $H^0(K_U,\ind_{K_P}^PV)\to H^0(K_U,V)$.  Since they are morphisms of universal $\delta$-functors, we conclude that $D_m=C_m$.
\end{proof}

We now recall from \cite[\S1.1]{herzig2011satake} that there is an alternate description of the Hecke algebra, which is used more often in applications.  In particular, we have the following explicit correspondence for $\mc H_G(V)$, with a similar correspondence for $\mc H_M\left(H^i(K_U,V)\right)$:
\begin{align*}
\HH_G(V) \coloneqq  \End_G(\ind_K^G V)
&\overset{\sim}{\xleftrightarrow{\hspace{1.5cm}}}
\left\{ 
\Psi: G \rightarrow \End_k V \; \middle| \;
\begin{array}{c}
\Psi(h_1 g h_2) = h_1 \circ \Psi(g) \circ h_2 \text{ for all }\\ h_1, h_2 \in K, g \in G, \quad 
\text{supp}(\Psi) \text{ compact}
\end{array}
\right\}\\
\big([1,v] \mapsto f_v \big) &\loongmapsto \big( g \longmapsto ( v \mapsto f_v(g) ) \big).
\end{align*}
Note that the inverse of this correspondence is given by taking any $\Psi:G\to \End_k(V)$ lying in our Hecke algebra, and sending it to the map in $\End_G(\ind_K^GV)$ sending any $[1,v]\in \ind_K^G V$ to the map $G\to V$ in $\ind_K^G V$ defined by 
$$g\mapsto \Psi(g)(v).$$
We seek to reformulate our description of the Satake morphism using this new interpretation of the Hecke algebras.  Recall that the isomorphism $\ind_K^G V\to \ind_{K_P}^P V$ coming from the Iwasawa decomposition is merely given by restriction to $P$.  In particular, given any $v\in V$, $[1,v]\in \ind_K^G V$ is sent to $[1,v]\in \ind_{K_P}^P V$ under this isomorphism.

In our description of the Satake morphism above, we used a $K_U$-linear map $V\to \ind_{K_P}^P V$ coming from $\Phi\in \End_G(\ind_K^G V)$ to induce a map on cohomology, which was constructed as follows:
$$V\xto{v\mapsto [1,v]}\ind_{K_P}^P V\xto{\sim} \ind_K^G V\xto{\Phi} \ind_K^G V\xto{\sim} \ind_{K_P}^P V.$$
Using our above explicit correspondences, given the element $\Psi\in \mc H_G(V)\subseteq \mathrm{Map}(G,\End_k(V))$ corresponding to $\Phi\in \End_G(\ind_K^G V)$, we can produce this same map via
\begin{align*}
\Psi(-):V&\to \ind_{K_P}^P V,\\
v&\mapsto \left(p\mapsto \Psi(p)(v)\right).
\end{align*}
This allows us to reformulate Theorem \ref{thm-explicit-formula} as follows.

\begin{cor}
\label{explicit-formula-reinterpreted}
With the above interpretation of the Hecke algebras, the higher mod $p$ Satake transform
$$\SVi:\mc H_G(V)\to \mc H_M(H^i(K_U,V))$$
is explicitly given by
$$\SVi(\Psi)(m)(\xi)=C_m^{(i)}\left(\Psi(-)(\xi)\right)$$
for any $\Psi\in \mc H_G(V)\subseteq \mathrm{Map}(G,\End_k(V))$, $m\in M$, and $\xi\in H^i(K_U,V)$, where $\Psi(-):H^i(K_U,V)\to H^i(K_U,\ind_{K_P}^P V)$ is the map on cohomology induced by $\Psi(-):V\to \ind_{K_P}^P V$.
\end{cor}

\section{The higher mod \texorpdfstring{$p$}{p} Satake transforms for \texorpdfstring{$\GL_n(F)$}{GLn(F)}}\label{sec-GLn}

In this section, our goal will be to apply Corollary \ref{explicit-formula-reinterpreted} to the case when $\mathbf G=\GL_n$ and find simpler expressions for the higher mod $p$ Satake transforms.  In part, this will make precise the naïve qualitative connection remarked before Definition \ref{def-delta-functor} between the morphism of $\delta$-functors $C_m$ and the notion of corestriction.

For the rest of the paper, we will be working with the function model of the Hecke algebra used in Corollary \ref{explicit-formula-reinterpreted}, where elements of $\mc H_G(V)$ are in particular functions $G\to \End_k(V)$.  We will also assume for the remainder of the paper that $V\in \Rep_k(K)$ is now irreducible.

\subsection{Hecke generators and coefficient maps}\label{subsec-GLn-notation}

Let $F / \Q_p$ be a finite extension with ring of integers $\mathcal{O}_F$, uniformizer $\varpi$, and residue field $k_F \cong \F_{p^f}$. Let $G = \GL_n(F)$, and let $K = \GL_n(\mathcal{O}_F)$ denote its usual maximal compact subgroup. 

We fix the standard upper-triangular Borel subgroup $B = T U$, where $T$ is the diagonal torus and $U$ is the unipotent radical. We denote their integral versions by $K_T \coloneqq  T \cap K$ and $K_U \coloneqq  U \cap K$. 

For any $\lambda = (\lambda_1, \dots, \lambda_n) \in \Z^n$, we define the associated torus element
$$t_\lambda \coloneqq  \diag(\varpi^{\lambda_1}, \dots, \varpi^{\lambda_n}) \in T.$$
We call $\Lambda^+ \coloneqq  \{ \lambda \in \Z^n \mid \lambda_1 \ge \lambda_2 \ge \dots \ge \lambda_n \}$ the \textit{dominant cone}.

For any subset $I \subset \{1, \dots, n\}$, we let $\omega_I \in \Z^n$ be such that $(\omega_I)_\ell = 1$ if $\ell \in I$ and $0$ otherwise. The corresponding torus element is denoted by $t_I \coloneqq  t_{\omega_I}$. We let $N_I$ denote the unipotent radical of the parabolic subgroup associated to $I$; i.e., $N_I$ is generated by the root groups $U_{a,b}$ where $a \in I$ and $b \not\in I$. This parabolic subgroup has Levi given by the centralizer of $t_I$; see \cite[3.2.15, 13.4, 15.4.4]{Springer}.  Note that the opposite unipotent group of $N_I$ is exactly $N_{I^c}$, the unipotent radical associated to the complement subset $I^c$. For each nonempty $I \subset \{1, \dots, n\}$, we also define the maps:
\begin{itemize}
    \item $q_I: V \twoheadrightarrow V_{N_I \cap K}$, the natural quotient map to the $(N_I \cap K)$-coinvariants,
    \item $i_I:  V^{N_I \cap K} \hookrightarrow V$, the natural inclusion of the $(N_I \cap K)$-invariants. 
\end{itemize}
For each integer $j \in \{1, \dots, n\}$, we also let $\omega_j \coloneqq  (1, \dots, 1, 0, \dots, 0) \in \Z^n$ consisting of exactly $j$ ones, and we write $t_j \coloneqq  t_{\omega_j}$ for simplicity. 

The Hecke algebra $\mathcal{H}_G(V)$ has a $k$-basis $\{T_\lambda\}_{\lambda \in \Lambda^+}$, where $T_\lambda$ is uniquely determined by the properties:
\bal
\Supp(T_\lambda) = K t_\lambda K \quad \text{and} \quad T_\lambda(t_\lambda) \in \End_k(V) \text{ is a linear projection.}
\nal
As a $k$-algebra, $\mathcal{H}_G(V)$ is isomorphic to $k[T_1, \dots, T_{n-1}, T_n^{\pm 1}],$ where $T_j \coloneqq  T_{\omega_j}$. For any $j \in \{1, \dots, n\}$, the linear projection $T_j(t_j)$ is given by the composition of $K_T$-equivariant maps:
\bal 
T_j(t_j): V \overset{q_J}{\longtwoheadrightarrow} V_{N_J \cap K} \overset{\eta_J}{\longrightarrow} V^{N_{J^c} \cap K} \overset{i_{J^c}}{\longhookrightarrow} V,
\nal 
where $J \coloneqq  \{1, \dots, j\}$ and $\eta_J: V_{N_J \cap K} \overset{\sim}{\longrightarrow} V^{N_{J^c} \cap K}$ is the canonical isomorphism of \cite[Lemma 2.5(ii), (2.8)]{herzig2011satake}.

For any subset $I \subset \{1,\dots, n\}$ of size $j$, there exists a permutation matrix $w \in K$ such that $t_I = w t_j w^{-1}$. By the $K$-biequivariance of $T_j$, we have the factorization of $K_T$-equivariant maps:
\bal 
T_j(t_I): V \overset{q_I}{\longtwoheadrightarrow} V_{N_I \cap K} \overset{\eta_I}{\longrightarrow} V^{N_{I^c} \cap K} \overset{i_{I^c}}{\longhookrightarrow} V,
\nal 
where $\eta_I \coloneqq  w \circ \eta_J \circ w^{-1}$ is an isomorphism independent of the choice of $w$. For $I=\emptyset$, we set $N_{\emptyset} = \{1\}$ and $q_{\emptyset} = i_{\emptyset} = \id_V$. 


    
\subsection{The transforms of the standard Hecke generators}
In order to describe $\SVi(T_j): T \rightarrow \End_k(H^i(K_U,V))$, we will need the following definitions.

\begin{definition}\label{defn-cI} For any nonempty subset $I \subset \{1, \dots, n\}$, define
\bal 
K_{U,I} \coloneqq  t_I K_U t_I^{-1} \cap K_U.
\nal 
\end{definition}

\begin{lem}\label{lem-KUI-compatible-pair} The pair of maps $\conjtI: K_{U,I} \rightarrow K_U$ and $T_j(t_I): V\rightarrow V$ induce a $K_T$-equivariant map on cohomology, which we denote by
\begin{align}\label{eqn-induced-map-conj-Tj-tI}
T_{j,I}: H^i(K_U,V) \rightarrow H^i(K_{U,I},V).
\end{align}
At the level of cocycles, $T_{j,I}(x) = T_j(t_I)_* \circ x \circ \conjtI^*$ for all $x \in C^i(K_U,V)$.
\end{lem}

\begin{proof}
It is straightforward to verify that
\bal 
T_j(t_I)(\conjtI(h) \cdot v) = h\cdot T_j(t_I)(v) \quad \text{for all }h \in \KUI, \, v \in V.
\nal 
By \cite[pg. 45]{NSW}, it follows that the pair $(\conjtI, T_j(t_I))$ induces maps $H^i(K_U,V) \rightarrow H^i(\KUI,V)$ for all $i \geq 0$ which satisfy the desired cocycle formula. Furthermore, it is easy to show that these maps are $K_T$-equivariant via this cocycle formula, using that $T$ is abelian.
\end{proof}

\begin{theorem}\label{thm-GLn-SVi-support-formula} Let $F / \Qp$ be a finite extension and let $G = \GL_n(F)$ for some $n \geq 1$. For any Serre weight $V$ and integer $i \geq 0$, we have 
\begin{align}\label{eqn-GLn-SVi-support}
    \Supp \SVi(T_j) \subseteq \bigsqcup\limits_{\substack{I \subset \{1, \dots, n\},\\ \#I = j}} K_T t_I
\end{align}
for all $j \in \{1, \dots, n\}$. Moreover, for each $I \subset \{1, \dots, n\}$ with $\#I = j$, we have 
\begin{align}\label{eqn-GLn-SVi-formula}
\SVi(T_j)(t_I) =\cores_{K_{U,I}}^{K_U} \circ T_{j,I} \in \End_k(H^i(K_U,V)).
\end{align}
\end{theorem}

\begin{proof} Let $j \in \{1, \dots, n\}$ and let $m = \diag(m_1, \dots, m_n) \in T$. By Corollary \ref{explicit-formula-reinterpreted}, we have the following formula for $\SVi(T_j)(m)$ at the level of cochains:
\begin{align}\label{eqn-SVi-Tj-m-cochains}
\SVi(T_j)(m)(x)(\sigma_\bullet) = \sum_{u \in K_U\backslash U}T_j(m \overline{u}^{-1}) \left( x\left( \left( \overline{u} m^{-1} \sigma_\ell m (\overline{u m^{-1} \sigma_\ell m})^{-1} \right)_{\ell=0}^i\right) \right). 
\end{align}
Thus, if $T_j(m u) = 0$ for all $u \in U$, then $\SVi(T_j)(m)=0 \in \End_k(H^i(K_U,V))$. In other words:
\bal 
m \in \Supp(\SVi(T_j)) \implies mu \in \Supp(T_j) = K t_j K \text{ for some }u \in U.
\nal 

Now assume there exists $u \in U$ such that $mu \in K t_j K$. Then $mu \in \Matn(\OF)$ since $t_j \in \Matn(\OF)$. Because $u \in U$ is upper unitriangular, the matrix $mu$ is upper-triangular with diagonal entries $m_1, \dots, m_n$. Then $mu \in \Matn(\OF)$ implies that $\valpi(m_r)\geq 0$ for all $r \in \{1, \dots, n \}.$ Similarly, $\varpi (m u)^{-1} \in K (\varpi t_j^{-1}) K \subset \Matn(\OF)$ is upper-triangular with diagonal entries $ \varpi m_1^{-1}, \dots, \varpi m_n^{-1}$, and so $\valpi(\varpi m_r^{-1}) = 1 - \valpi(m_r) \geq 0$ for all $r \in \{1, \dots, n\}$. Thus, 
\bal 
0 \leq \valpi(m_r) \leq 1 \quad \text{for all }r \in \{1, \dots, n\}.
\nal 

Taking the determinant yields $\valpi(\det(mu)) = \valpi(\det(m))= \valpi(\det(t_j)) = j$, and so we also have $\sum_{r=1}^n \valpi(m_r) = j$. Therefore, exactly $j$ diagonal entries of $m$ must have valuation $1$, and the remaining $n-j$ entries must have valuation $0$. Consequently, $m \in K_T t_I$ for some $I \subset \{1, \dots, n\}$ with $\#I = j$. This establishes \eqref{eqn-GLn-SVi-support}.

Now fix some subset $I \subset \{1, \dots, n\}$ satisfying $\#I = j$. To prove (\ref{eqn-GLn-SVi-formula}), we will plug $m = t_I$ into (\ref{eqn-SVi-Tj-m-cochains}). First, suppose there exists $ \overline{u} \in U$ such that $t_I \overline{u}^{-1} \in K t_j K$. Then $ \overline{u}\in U \cap K t_I^{-1} K t_I$ since $K t_j^{-1} K = K t_I^{-1} K$. By Lemma \ref{lem-U-cap-KtKt}, we then have $\overline{u} \in K_U(t_I^{-1} K_U t_I)$. Thus, when $m= t_I$, the only nonzero summands in (\ref{eqn-SVi-Tj-m-cochains}) correspond to $u$ that belong to the set $K_U \backslash K_U ( t_I^{-1} K_U t_I)$. Moreover, there is a bijection
\begin{align}\label{eqn-bij-cosets}
    K_{U,I} \backslash K_U \rightarrow K_U \backslash K_U ( t_I^{-1} K_U t_I), \quad K_{U,I} v\mapsto K_U(t_I^{-1} v t_I).
\end{align}

For each coset $K_{U,I}v\in K_{U,I}\backslash K_U$, fix a representative $\widetilde{v}\in K_U$. By Corollary \ref{cor-independence}, we are free to choose any set of lifts of $K_U\backslash U$. For the coset $u=K_U(t_I^{-1}vt_I)$ corresponding to $K_{U,I}v$ under \eqref{eqn-bij-cosets}, choose the representative $\overline{u}=t_I^{-1}\widetilde{v}t_I$. We write $\widetilde{v\sigma_\ell}$ for the chosen representative of the coset $K_{U,I}v\sigma_\ell$. Evaluating (\ref{eqn-SVi-Tj-m-cochains}) at $m=t_I$ then yields:
\bal 
\SVi(T_j)(t_I)(x)(\sigma_\bullet) &= \sum_{u \in K_U \backslash K_U ( t_I^{-1} K_U t_I)} T_j(t_I \overline{u}^{-1}) \left( x\left( \left( \overline{u} t_I^{-1} \sigma_\ell  t_I(\overline{u t_I^{-1} \sigma_\ell t_I})^{-1} \right)_{\ell=0}^i\right) \right) \\
&=\sum_{v \in K_{U,I} \backslash K_U} T_j\left(\widetilde{v}^{-1} t_I\right) \left( x\left( \left( t_I^{-1} \widetilde{v} \sigma_\ell t_I \left(\ \overline{ t_I^{-1} \widetilde{v} \sigma_\ell t_I} \ \right)^{-1} \right)_{\ell=0}^i \right) \right).
\nal 
Note that $T_j(\widetilde{v}^{-1} t_I) = \widetilde{v}^{-1} \circ T_j(t_I)$ by the $K$-equivariance of $T_j$. Also, since $K_{U,I} \widetilde{v} \sigma_\ell = K_{U,I} \widetilde{v \sigma_\ell}$, there exists $h \in K_{U,I}$ such that $\widetilde{v} \sigma_\ell = h \widetilde{v \sigma_\ell}$. But then 
\bal 
K_U t_I^{-1} \widetilde{v} \sigma_\ell t_I = K_U (t_I^{-1} h t_I) (t_I^{-1} \widetilde{v \sigma_\ell} t_I) = K_U (t_I^{-1} \widetilde{v \sigma_\ell}t_I)
\nal 
since $h \in K_{U,I} \subset t_I K_U t_I^{-1}$. Thus, $\overline{t_I^{-1} \widetilde{v} \sigma_\ell t_I} = t_I^{-1} \widetilde{v \sigma_\ell} t_I$ by our choice of lifts. Hence, 
\bal 
\SVi(T_j)(t_I)(x)(\sigma_\bullet) &=\sum_{v \in K_{U,I} \backslash K_U} \widetilde{v}^{-1} \cdot T_j\left(t_I\right) \left( x\left( \left( t_I^{-1} \widetilde{v} \sigma_\ell t_I\left( t_I^{-1} \widetilde{v \sigma_\ell} t_I \right)^{-1} \right)_{\ell=0}^i \right) \right) \\
&= \sum_{v \in K_{U,I} \backslash K_U} \widetilde{v}^{-1} \cdot T_j(t_I) \left( x\left( \left( t_I^{-1} \widetilde{v} \sigma_\ell (\widetilde{v \sigma_\ell})^{-1} t_I \right)_{\ell=0}^i \right) \right) \\
&= \sum_{v \in K_{U,I} \backslash K_U} \widetilde{v}^{-1} \cdot T_{j,I} (x)\left( \left( \widetilde{v} \sigma_\ell (\widetilde{v \sigma_\ell})^{-1} \right)_{\ell=0}^i \right)  \\
&= \bigl(\cores_{K_{U,I}}^{K_U}\circ T_{j,I}\bigr)(x)(\sigma_\bullet).
\nal 
\end{proof}

\begin{lem}\label{lem-U-cap-KtKt}
For all $I \subset \{1, \dots, n\}$ with $j \coloneqq  \# I \in \{1,\dots,n\}$, we have
\bal
U \cap K t_I^{-1} K t_I = K_U ( t_I^{-1} K_U t_I).
\nal
\end{lem}
\begin{proof}
The $(\supseteq)$ containment is trivial. For the reverse containment, let $u \in U \cap K t_I^{-1} K t_I$ and set $A = u t_I^{-1}$. It suffices to show that there exist $h_1, h_2 \in K_U$ such that $h_1 A h_2 = t_I^{-1}$.

Let $\mathbf{S} \coloneqq  \{(a,b) \in \Z^2: 1 \leq a < b \leq n\}$,
totally ordered by $(a_1,b_1) \prec (a_2,b_2)$ if
$b_1-a_1 < b_2-a_2$, or if
$b_1-a_1 = b_2-a_2$ and $a_1<a_2$. We proceed by induction on $\mathbf{S}$ with respect to $\prec$. 

We claim that for every $(r,c) \in \mathbf{S}$, there exist $x_1, x_2 \in K_U$ such that $A'\coloneqq  x_1 A x_2$ satisfies $A'_{a,b} = 0$ for all $(a,b) \prec (r,c)$, $A'_{a,a} = A_{a,a}$ for all $1 \leq a \leq n$, and $\varpi A', (A')^{-1} \in \Matn(\OF)$. This holds for the base case $(r,c) = (1,2)$ with $x_1 = x_2 = \Id$.

Assume the claim holds for some $(r,c) \in \mathbf{S}$. To show it holds for the successor of $(r,c)$, it suffices to find $z_1, z_2 \in K_U$ such that $A'' \coloneqq  z_1 A' z_2$ annihilates the $(r,c)$-entry while preserving the previous zeros and the diagonal. For any $\alpha\in F$, let $E_{a,b}(\alpha)$ denote the elementary matrix with $1$s on the
diagonal and $\alpha$ in the $(a,b)$ position.
Since $A'_{a,a} = (t_I^{-1})_{a,a} \in \{1,\varpi^{-1}\}$
for every $1\leq a\leq n$, we consider three (non-mutually exclusive) cases:
\begin{itemize}
    \item If $A'_{c,c} = \varpi^{-1}$, let $z_1 = E_{r,c}(-\varpi A'_{r,c})$ and $z_2 = \Id$. 
    \item If $A'_{r,r} = \varpi^{-1}$, let $z_1 = \Id$ and $z_2 = E_{r,c}(-\varpi A'_{r,c})$.
    \item If $A'_{r,r} = A'_{c,c} = 1$, consider the $(r,c)$-entry of the identity $\Id= A' (A')^{-1}$:
    $$0 = (A' (A')^{-1})_{r,c} = \sum_{\ell=r}^c A'_{r,\ell} ((A')^{-1})_{\ell,c} = ((A')^{-1})_{r,c} + A'_{r,c} + \sum_{\ell=r+1}^{c-1} A'_{r,\ell} ((A')^{-1})_{\ell,c}.$$
    Since $(r,\ell) \prec (r,c)$ for all $r<\ell<c$, the sum vanishes by the induction hypothesis. Then $A'_{r,c} = - (A')^{-1}_{r,c} \in \OF.$ Let $z_1 = E_{r,c}(-A'_{r,c})$ and $z_2 = \Id$.
\end{itemize}

In each case, $z_1, z_2 \in K_U$ and $A'' \coloneqq  z_1 A' z_2$ satisfies $A''_{a,b} = A'_{a,b}$ if $a=b$ or $(a,b) \prec (r,c)$. Moreover, $\varpi A'', (A'')^{-1} \in \Matn(\OF)$ since $z_1, z_2 \in K_U \subset K$. By induction, there exist $h_1, h_2 \in K_U$ such that $h_1 A h_2=t_I^{-1}$. Hence, $u = A t_I \in K_U t_I^{-1} K_U t_I.$
\end{proof}

\section{The image of \texorpdfstring{$\SVi$}{S(V,i)} for \texorpdfstring{$\GL_2(F)$}{GL2(F)}} \label{sec-GL_2} 

The goal of this section is to determine the image of the Satake transform
\bal 
\SVi: \mathcal{H}_G(V) \longrightarrow \HT(H^i(K_U,V))
\nal 
when $G=\GL_2(F)$ for $V$ a Serre weight. Since $\SVi$ is a $k$-algebra homomorphism and $\mathcal{H}_G(V)=k[T_1,T_2^{\pm 1}]$, the image is generated by $\SVi(T_1)$ and $\SVi(T_2)^{\pm1}$. As is explained in the proof of Theorem \ref{thm-GL2-SVi-support-formula}, the second generator is already completely understood: by Theorem \ref{thm-GLn-SVi-support-formula}, it is supported on $K_Tt_{\{1,2\}}$ and satisfies
\bal
\SVi(T_2)(t_{\{1,2\}})=\id_{H^i(K_U,V)}.
\nal
Thus, the main task is to understand $\SVi(T_1)$. By Theorem \ref{thm-GLn-SVi-support-formula}, we know that
\bal
\Supp\SVi(T_1)\subseteq K_T t_{\{1\}}\sqcup K_T t_{\{2\}}.
\nal
Since $\SVi(T_1)$ is $K_T$-equivariant, it is determined by its values at $t_{\{1\}}$ and $t_{\{2\}}$, which we call
\bal
\AVi\coloneqq \SVi(T_1)(t_{\{1\}})
\qquad\text{and}\qquad
\BVi\coloneqq \SVi(T_1)(t_{\{2\}}).
\nal
We first study the vanishing of $\AVi$, then the nilpotence of $\BVi$, and finally their interaction in order to describe the image of $\SVi$.

\subsection{Serre weights and cohomological coordinates}\label{subsec-GL2-notation}\subsubsection{Basic notation}
We retain the notation of Section~\ref{sec-GLn} with $n=2$. Thus, $G=\GL_2(F)$, $K=\GL_2(\OF)$, and $B=TU$ is the standard upper-triangular Borel subgroup. For $x \in F$, let 
\bal
u_\alpha(x)=\begin{pmatrix}1&x\\0&1\end{pmatrix}.
\nal
We identify $K_U = u_{\alpha}(\OF)$ and $u_{\alpha}(\varpi \OF)$ with the additive groups $(\OF,+)$ and $(\varpi \OF,+)$, respectively, and adopt additive notation throughout this section. Thus, we write $H^i(\OF,V)$, $H^i(\varpi\OF,V)$, and $\cores_{\varpi\OF}^{\OF}$ for the corresponding cohomology groups and corestriction map. We continue to use the subgroup notation $K_U$ in expressions involving invariants, coinvariants, and coefficient maps,
where the representation-theoretic meaning is clearer. For $x\in\OF$ and $v\in V$, the notation $x\cdot v$ means $u_\alpha(x)\cdot v$.

Throughout this section, we compute continuous group cohomology using inhomogeneous cochains. Thus, for a continuous $\OF$-module $E$, we write
\bal 
C^i(\OF,E)=\operatorname{Map}_{\mathrm{cts}}(\OF^i,E)
\nal 
for the continuous functions $\OF^i\to E$, with the usual inhomogeneous differential and cup product. We identify this complex with the homogeneous cochain complex used in the preceding sections via the standard isomorphism
\bal 
f(\alpha_1,\ldots,\alpha_i)
\longmapsto
\widetilde f(\alpha_0,\ldots,\alpha_i) =
\alpha_0\cdot
f(\alpha_1-\alpha_0,\ldots,\alpha_i-\alpha_{i-1}).
\nal 
When we use cochain-level formulas from the preceding sections, we implicitly pass to the corresponding inhomogeneous formulas via this identification.

For $t=\diag(\beta,\gamma)\in T$, conjugation on $U$ corresponds to scalar multiplication:
\bal
t u_\alpha(x)t^{-1}=u_\alpha(\beta\gamma^{-1}x).
\nal
For $r,s\in\Z$, set $t_{r,s}\coloneqq \diag(\varpi^r,\varpi^s)$. Thus, in the notation of Section \ref{sec-GLn}, we have
\bal
t_{1,0}=t_{\{1\}},\qquad t_{0,1}=t_{\{2\}},\qquad
t_{1,1}=t_{\{1,2\}}.
\nal
We write $T_1\coloneqq T_{\{1\}}$ and $T_2\coloneqq T_{\{1,2\}}$ for the two Hecke generators.
When $a\in k_F$, we write $a\cdot v$ for $[a]\cdot v$, where $[a]\in\OF$ denotes the Teichm\"uller lift. For $\alpha \in \OF$, $\overline{\alpha} \in k_F$ is the reduction modulo $\varpi$.

\subsubsection{Serre weights} Once and for all, fix an embedding $\kappa_0:k_F\hookrightarrow k$, and let $\kappa_j$ denote $\kappa_0\circ(x\mapsto x^{p^j})$ for $j\in\Z/f\Z$. For ease of notation, we identify $k_F$ with its image under $\kappa_0$.  Every Serre weight for $\GL_2(F)$ is of the form 
\bal
F(\ul a,\ul b)\coloneqq 
\bigotimes_{j=0}^{f-1}
\left(
\left(
\Sym^{a_j-b_j}(k_F^2)\otimes\mathrm{det}^{b_j}
\right)
\otimes_{k_F,\kappa_j}k
\right)
\nal
for some $\ul a=(a_j)_j,\ul b=(b_j)_j\in\Z^f$ satisfying $0\leq a_j-b_j\leq p-1$ for every $j$. 
Here $\GL_2(k_F)$ acts in the usual way on $k_F^2$, and we view $F(\ul a,\ul b)$ as a representation of $K$ by inflation along the reduction map $K\rightarrow\GL_2(k_F)$. For a Serre weight $V=F(\ul a,\ul b)$, we let
\bal
d_j&\coloneqq a_j-b_j, \qquad d\coloneqq \sum_{j=0}^{f-1}d_j p^j, \qquad b\coloneqq  \sum_{j=0}^{f-1} b_j p^j,\\
n_r(V)&\coloneqq \#\{j\in\{0,\dots,f-1\}:d_j=r\} \quad \text{for }r \in \{0, p-1\}.
\nal 
We call $V$ \textit{Steinberg} if $d_j = p-1$ for all $j$, and \textit{non-Steinberg} otherwise. 

Via the natural $\GL_2(k_F)$-equivariant injection $V\hookrightarrow \Sym^d(k^2)\otimes\mathrm{det}^b$ (see \cite[Example 2.2]{herzigmod}), we identify $V$ with the subrepresentation spanned by the monomials  $X^{d-c}Y^c$, where $c=\sum_{j=0}^{f-1}c_j p^j$ with $0 \leq c_j \leq d_j$ for each $j$. Under this identification,
\bal
V^{K_U}=kX^d
\qquad\text{and}\qquad
V_{K_U}=k\big(Y^d+(X)\big),
\nal
where $Y^d+(X)$ simply means the image of $Y^d$ in $V_{K_U}$.

Recall that $\KUbar$ denotes the lower unipotent subgroup of $K$. Let $q_{K_U}: V \twoheadrightarrow V_{K_U}$ denote the natural projection and $i_{K_U}: V^{K_U} \hookrightarrow V$ denote the natural injection, and similarly define $q_{\KUbar}, i_{\KUbar}$. Then $q_{K_U}\circ i_{\KUbar}:V^{\KUbar}\rightarrow V_{K_U}$ and $q_{\KUbar}\circ i_{K_U}:V^{K_U}\rightarrow V_{\KUbar}$ are isomorphisms. Let $\eta_Y$ and $\eta_X$ denote their inverses, respectively. Lastly, define the maps
\bal
\Pry\coloneqq  T_1(t_{1,0})&: V\overset{q_{K_U}}{\longtwoheadrightarrow} V_{K_U}
\overset{\eta_Y}{\longrightarrow}V^{\KUbar}\overset{i_{\KUbar}}{\longhookrightarrow} V,\\
\Prx \coloneqq  T_1(t_{0,1})&: V \overset{q_{\KUbar}}{\longtwoheadrightarrow} V_{\KUbar}
\overset{\eta_X}{\longrightarrow} V^{K_U}\overset{i_{K_U}}\longhookrightarrow V.
\nal
The images of $\Pry$ and $\Prx$ are the lines spanned by $Y^d$ and $X^d$, respectively. 

\subsubsection{Cohomological coordinates} Define the sets
\bal
\scrJ &\coloneqq \{(m,j):0\leq m\leq e-1,\ 0\leq j\leq f-1\},\\
\scrJi &\coloneqq  \{ S \subset \scrJ: \#S = i\} \quad \text{for all }i\geq 1.
\nal
Every element of $\OF/p\OF$ has a unique Teichm\"uller expansion $\sum_{m=0}^{e-1}[x_m]\varpi^m$ with $x_m\in k_F$. For $(m,j)\in\scrJ$, define the map $\kappa_{m,j}: \OF/ p\OF \rightarrow k$ by
\bal
\kappa_{m,j}\left(\sum_{r=0}^{e-1}[x_r]\varpi^r\right)
\coloneqq \kappa_j(x_m) = x_m^{p^j}.
\nal
The maps $\kappa_{m,j}$ form a $k$-basis of
$\Hom_{\Fp}(\OF/p\OF,k)$. We also write
\bal
\kappa_{m,j}:\OF\longrightarrow k,
\qquad \kappa_{m,j}(\alpha) \coloneqq \kappa_{m,j}(\alpha+p\OF),
\nal
for the corresponding inflated map. It is a continuous additive
homomorphism, hence an inhomogeneous $1$-cocycle. For a subset
$S=\{(m_1,j_1),\dots,(m_i,j_i)\}\subseteq\scrJ$, with its elements
ordered lexicographically, set
\bal
\kappa_S\coloneqq  \kappa_{m_1,j_1}\cup\cdots\cup\kappa_{m_i,j_i}.
\nal
Since $k$ has trivial $\OF$-action, this cup product is represented by
\bal
\kappa_S(\alpha_1,\ldots,\alpha_i) =\prod_{r=1}^i\kappa_{m_r,j_r}(\alpha_r).
\nal
By Lazard's theorem (see \cite[Ch. V, Proposition 2.5.7.1]{lazard65}), the classes $[\kappa_S]$, with $S\in\scrJi$,
form a $k$-basis of $H^i(\OF,k)$. Since
$H^0(\OF,V^{K_U})=V^{K_U}=kX^d$, the classes
$[\kappa_S]\cup X^d$ form a $k$-basis of
$H^i(\OF,V^{K_U})$.

Finally, define
\bal
\AVi\coloneqq \SVi(T_1)(t_{1,0}),
\qquad
\BVi\coloneqq \SVi(T_1)(t_{0,1}),
\qquad
\Bki\coloneqq \Ski(T_1)(t_{0,1}),
\nal
where we recall that we are assuming $V$ is a Serre weight.  We write $\AV$, $\BV$, and $\Bk$ for the corresponding graded operators on $H^*(\OF,V)$ and $H^*(\OF,k)$. Whenever $\AVi$ or $\BVi$ is nilpotent, we denote its nilpotence index by
\bal
M_i(V)&\coloneqq \min\{M\geq1:\AVi^M=0\},\\
N_i(V)&\coloneqq \min\{N\geq1:\BVi^N=0\}.
\nal

\subsection{The operators \texorpdfstring{$\AVi$}{A(V,i)} and \texorpdfstring{$\BVi$}{B(V,i)}}\label{subsec-explicit-description-SVi-GL2}

\begin{theorem}\label{thm-GL2-SVi-support-formula}
Let $F/\Q_p$ be a finite extension and let $G=\GL_2(F)$. For any Serre weight $V$ and integer $i\geq0$, we have
\bal
\Supp\SVi(T_2)\subseteq K_Tt_{1,1},\qquad
\Supp\SVi(T_1)\subseteq K_Tt_{1,0}\sqcup K_Tt_{0,1}.
\nal
Moreover, $\SVi(T_2)(t_{1,1})$ is the identity endomorphism of $H^i(\OF,V)$, while   
\begin{align}\label{eqn-GL2-SVi-formula}
\AVi  &= \cores_{\varpi\OF}^{\OF}\circ(\Pry)_*\circ m_{\varpi^{-1}}^*,\\
\BVi &= (\Prx)_*\circ m_\varpi^* \circ \res_{\varpi \OF}^{\OF},\nonumber
\end{align}
where $m_{\varpi^{-1}}:\varpi\OF\xrightarrow{\sim}\OF$ and $m_\varpi:\OF\overset{\sim}{\rightarrow}\varpi \OF$ are multiplication by $\varpi^{-1}$ and $\varpi$, respectively. 
\end{theorem}

\begin{proof}
The support containments follow immediately from Theorem \ref{thm-GLn-SVi-support-formula}. For $t_{1,1}$, the subgroup $K_{U,\{1,2\}}$ is $\OF$ (recall that we're identifying $\OF$ with $u_\alpha(\OF)= K_U$), conjugation is trivial, and $T_2(t_{1,1})$ is the identity. For $t_{0,1}$, we have $K_{U,\{2\}}=\OF$, conjugation by $t_{0,1}^{-1}$ is multiplication by $\varpi$, and the coefficient map is $\Prx$. For $t_{1,0}$, we have $K_{U,\{1\}}=\varpi\OF$, conjugation by $t_{1,0}^{-1}$ is multiplication by $\varpi^{-1}$, and the coefficient map is $\Pry$. Substituting these descriptions into \eqref{eqn-GLn-SVi-formula} gives \eqref{eqn-GL2-SVi-formula}. 
\end{proof}

In \eqref{eqn-GL2-SVi-formula}, pullback is understood with pulled-back coefficients. More precisely,
\bal
m_{\varpi^{-1}}^*:
H^i(\OF,V)
&\longrightarrow
H^i\bigl(\varpi\OF,m_{\varpi^{-1}}^*V\bigr),\\
m_\varpi^*:
H^i(\varpi\OF,V|_{\varpi\OF})
&\longrightarrow
H^i\bigl(\OF,m_\varpi^*(V|_{\varpi\OF})\bigr).
\nal
Here $\Pry$ is regarded as a $\varpi\OF$-equivariant map $\Pry: m_{\varpi^{-1}}^*V \to V|_{\varpi\OF},$ while $\Prx$ is regarded as an $\OF$-equivariant map $\Prx :m_\varpi^*(V|_{\varpi\OF}) \to V.$ 

Thus, the two composites in \eqref{eqn-GL2-SVi-formula} are precisely the maps induced by the compatible pairs appearing in Lemma~\ref{lem-KUI-compatible-pair}. We will henceforth omit these coefficient decorations when the relevant coefficient modules are clear from context.

\subsection{Vanishing of \texorpdfstring{$\AVi$}{A(V,i)}}\label{subsec-A-vanishing}
In this subsection, we prove that $\AVi$ vanishes for a range of cohomological degrees. The key idea is to factor $\AVi$ through
$H^i(\OF,V_{K_U})$. Unlike $H^i(\OF,V)$, this space has an explicit
basis represented by scalar-valued cocycles, which allows us to
analyze the corestriction sum directly.
\begin{lem}\label{lem-AVi-factorization} The endomorphism $\AVi$ factors as $\AVi = \AVi' \circ (q_{K_U})_*$ for $q_{K_U}:V\onto V_{K_U}$ as above, where   
\bal 
\AVi': H^i(\OF,V_{K_U}) \longrightarrow H^i(\OF,V)
\nal
is defined by
\bal
\AVi' =\cores_{\varpi\OF}^{\OF} \circ(i_{\KUbar})_* \circ(\eta_Y)_* \circ m_{\varpi^{-1}}^*.    
\nal 
\end{lem}
\begin{proof}  Since $\Pry=i_{\KUbar}\circ\eta_Y\circ q_{K_U}$, we have by Theorem \ref{thm-GL2-SVi-support-formula}
\bal
\AVi = \cores_{\varpi\OF}^{\OF}
\circ(i_{\KUbar})_*
\circ(\eta_Y)_*
\circ(q_{K_U})_*
\circ m_{\varpi^{-1}}^*,
\nal
and $(q_{K_U})_*\circ m_{\varpi^{-1}}^* = m_{\varpi^{-1}}^*\circ(q_{K_U})_*$ by naturality of pullback with respect to coefficient maps. 
\end{proof}
By Lemma \ref{lem-AVi-factorization}, it suffices to understand when $\AVi'$ vanishes. To study this map, we write it explicitly at the level of inhomogeneous cocycles. For $a\in k_F$, define
\bal
\theta_a:\OF\longrightarrow\OF,
\qquad
\theta_a(\alpha) = \bigl([a]+\alpha-[a+\overline{\alpha}]\bigr)\varpi^{-1}.
\nal
For $\alpha_\bullet=(\alpha_1,\dots,\alpha_i)\in\OF^i$, set
\bal
s_0\coloneqq 0, \qquad s_r\coloneqq \alpha_1+\cdots+\alpha_r \quad(1\leq r\leq i).
\nal
Then
\bal
\theta_a(s_r)-\theta_a(s_{r-1}) =\theta_{a+\overline{s_{r-1}}}(\alpha_r)
\qquad(1\leq r\leq i).
\nal
In terms of inhomogeneous cocycles, the corestriction formula from the preceding sections becomes
\begin{align}\label{eqn-AVi-prime-cocycle}
\AVi'(x)(\alpha_1,\dots,\alpha_i)
&= \sum_{a\in k_F}
-a\cdot (i_{\KUbar}\circ\eta_Y)
\bigl(
x( \theta_a(\alpha_1),
\theta_{a+\overline{s_1}}(\alpha_2),
\dots,
\theta_{a+\overline{s_{i-1}}}(\alpha_i)
)
\bigr)
\end{align}
for every $x\in Z^i(\OF,V_{K_U})$, where for $i=0$ the input to $x$ is empty.

For $S=\{(m_1,j_1),\dots,(m_i,j_i)\}\in\scrJi$, with its elements ordered lexicographically, let
$x_S\in Z^i(\OF,V_{K_U})$ be the cocycle
\bal
x_S(\alpha_1,\dots,\alpha_i)
= \prod_{r=1}^i
\kappa_{m_r,j_r}(\alpha_r)
\bigl(Y^d+(X)\bigr).
\nal
The classes $[x_S]$, with $S\in\scrJi$, form a $k$-basis of
$H^i(\OF,V_{K_U})$. Since
\bal
\eta_Y\bigl(Y^d+(X)\bigr)=Y^d
\qquad\text{and}\qquad
(-a)\cdot Y^d=(Y-aX)^d,
\nal
equation \eqref{eqn-AVi-prime-cocycle} yields 
\begin{align}\label{eqn-A-xS}
\AVi'(x_S)(\alpha_\bullet)
&=
\sum_{\ell=0}^d \binom{d}{\ell}(-1)^\ell X^\ell Y^{d-\ell}
\sum_{a\in k_F}C_{a,\alpha_\bullet}a^\ell,
\end{align}
where $C_{a, \alpha_{\bullet}} \in k$ is defined as
\bal
C_{a,\alpha_\bullet}
\coloneqq 
\prod_{r=1}^i
\kappa_{m_r,j_r}
\bigl(
\theta_{a+\overline{s_{r-1}}}(\alpha_r)
\bigr).
\nal

We will prove the vanishing of $\AVi'$ by applying the finite field
power sum formula to the inner sum in
\eqref{eqn-A-xS}. For this, we must control the powers of $a$ that
can occur in $C_{a,\alpha_\bullet}$. The following lemma provides
the required control.

\begin{lem}\label{lem-theta-mod-p}
Let $x\in\OF$, let $b\in k_F$, and let $(m,j)\in\scrJ$. Let
$j'\in\{0,\dots,f-1\}$ such that
\bal
j'\equiv j-1\pmod f.
\nal
As a function of $a\in k_F$, the expression
$\kappa_{m,j}\bigl(\theta_{a+b}(x)\bigr)$ is independent of $a$
if $m<e-1$. If $m=e-1$, it is a $k$-linear combination of the
monomials
\bal
a^{cp^{j'}},
\qquad
0\leq c\leq p-1.
\nal
\end{lem}
\begin{proof}
There exist unique $x_0,\dots,x_{e-1}\in k_F$ such that
\begin{align}\label{eqn-x-teich}
x\equiv x' \coloneqq  \sum_{r=0}^{e-1}[x_r]\varpi^r\pmod p.
\end{align}
Observe that $\theta_a(x) - \theta_a(x') = (x- x') \varpi^{-1}$ is congruent to an $a$-independent multiple of $\varpi^{e-1}$ modulo $p$. We may therefore assume
that $x = x'$. 

Let $W_1$ denote the first Witt carry polynomial, 
\bal
W_1(X,Y) = \frac{X^p+Y^p-(X+Y)^p}{p} \; \in\Z[X,Y],
\nal
and let $\overline{W_1}(X,Y) \in k_F[X,Y]$ be its reduction modulo $p$. Witt vector addition yields
\begin{align}\label{eqn-W1-congruence}
[a]+[x_0]-[a+x_0]
\equiv
p[\overline{W_1}(a,x_0)^{1/p}]
\pmod{p^2},
\end{align}
where $1/p$ denotes that we are applying the inverse Frobenius map.  If we let $u \in \OF$ such that $p = u \varpi^e$, then dividing \eqref{eqn-W1-congruence} by
$\varpi$ yields 
\bal
\theta_a(x)
\equiv
\sum_{r=0}^{e-2}[x_{r+1}]\varpi^r
+
[\overline u \overline{W_1}(a,x_0)^{1/p}]\varpi^{e-1}
\pmod p.
\nal
Here division preserves the congruence modulo $p$ because
$p^2\varpi^{-1}\OF\subseteq p\OF$. In particular, after replacing
$a$ by $a+b$, every coefficient of $\varpi^m$ with $m<e-1$
remains independent of $a$. Hence
$\kappa_{m,j}(\theta_{a+b}(x))$ is independent of $a$ whenever
$m<e-1$.

It remains to consider $m=e-1$. Let $\gamma_1(x_0), \dots, \gamma_{p-1}(x_0) \in k_F$ such that
\bal
\overline{W_1}(X,x_0) =\sum_{c=1}^{p-1}\gamma_c(x_0)X^c. 
\nal
Since the inverse Frobenius on $k_F$ is the same as the $p^{f-1}$ power map, we have
\bal
\overline{W_1}(a+b,x_0)^{1/p}
=
\sum_{c=1}^{p-1}
\gamma_c(x_0)^{p^{f-1}}
(a+b)^{cp^{f-1}}.
\nal
Expanding each term shows that the $\varpi^{e-1}$ coefficient of
$\theta_{a+b}(x)$ is a linear combination of the monomials $a^{rp^{f-1}}$, with $0\leq r\leq p-1$. Finally, $\kappa_{e-1,j}$ applies $\kappa_j$ to this coefficient, and
therefore raises each such monomial to its $p^j$-th power. Thus, its dependence on $a$ is a
$k$-linear combination of monomials $a^{rp^{f-1+j}}$ with $0 \leq r \leq p-1$. The desired result follows since $a^{p^f}=a$ for $a\in k_F$. 
\end{proof}
We can now prove the desired vanishing result. Recall that $n_{p-1}(V)$ is the number of indices $j$ for which $d_j = p-1$. 

\begin{theorem}\label{thm-A-vanishing}
Let $V=F(\underline a,\underline b)$ be a Serre weight. Then
\bal
\AVi=0
\qquad\text{for every}\qquad
 i<f-n_{p-1}(V).
\nal
\end{theorem}
\begin{proof} We first consider $i=0$. In this case, the assumption
$0<f-n_{p-1}(V)$ implies $V$ is non-Steinberg, and hence
$d<p^f-1$. Since $H^0(\OF,V_{K_U})$ is spanned by
$Y^d+(X)$, the cocycle formula \eqref{eqn-AVi-prime-cocycle} gives
\[
\mathbf{A}_{V,0}'\bigl(Y^d+(X)\bigr)
=
\sum_{a\in k_F}(-a)\cdot Y^d
=
\sum_{\ell=0}^d
\binom{d}{\ell}(-1)^\ell X^\ell Y^{d-\ell }
\sum_{a\in k_F}a^\ell .
\]
For $1\leq \ell \leq d<p^f-1$, the inner sum vanishes by the finite
field power-sum formula, while for $\ell =0$ it vanishes since
$\#k_F=p^f=0$ in $k$. Thus $\mathbf{A}_{V,0}'=0$, and hence
$\mathbf{A}_{V,0}=0$ by Lemma~\ref{lem-AVi-factorization}.

Now suppose that $1 \leq i<f-n_{p-1}(V)$. By Lemma
\ref{lem-AVi-factorization} and the basis of
$H^i(\OF,V_{K_U})$ described above, it suffices to prove that $\AVi'([x_S])=0$ for every $S\in\scrJi$.

Fix $S \in \scrJi$ and $\alpha_\bullet=(\alpha_1,\dots,\alpha_i)\in\OF^i$. In
\eqref{eqn-A-xS}, Lucas's theorem shows that $\binom{d}{\ell }$ can be
nonzero only when $\ell =\sum_{j} \ell _j p^j$ with $0 \leq \ell _j \leq d_j$ for each $j$. By Lemma \ref{lem-theta-mod-p}, the $r$-th factor of
$C_{a,\alpha_\bullet}$ is independent of $a$ unless $m_r=e-1$.
When $m_r=e-1$, its monomials have exponents
\bal
cp^{j_r'},
\qquad
0\leq c\leq p-1,
\qquad
j_r'\equiv j_r-1\pmod f.
\nal
The indices $j_r'$ arising in this way are distinct.

It follows that every monomial in
$C_{a,\alpha_\bullet}a^\ell $ has exponent
\bal
E=\sum_{j=0}^{f-1}(\ell _j+c_j)p^j,
\qquad
0\leq c_j\leq p-1,
\nal
where $c_j=0$ unless $j=j_r'$ for some
$(e-1,j_r)\in S$. Since $i<f$, the sum $\sum_j c_j p^j$ is supported
on fewer than $f$ indices. Hence
\bal
0\leq E<2(p^f-1).
\nal
The finite field power sum formula therefore shows that
$\sum_{a\in k_F}a^E$ can be nonzero only when
$E=p^f-1$, as the sum vanishes when $E=0$ since
$\#k_F=p^f=0$ in $k$.

Suppose that $E=p^f-1$. Since
$0\leq \ell _j+c_j\leq2(p-1)$, reducing successively modulo $p$ in
\bal
\sum_{j=0}^{f-1}(\ell _j+c_j)p^j
=
\sum_{j=0}^{f-1}(p-1)p^j
\nal
shows that $\ell _j + c_j = p-1$ for every $j$. If $d_j<p-1$, then $\ell _j\leq d_j<p-1$, and therefore $c_j>0$.
Thus every such $j$ must be one of the distinct indices $j_r'$
associated to a pair $(e-1,j_r)\in S$. Consequently,
\bal
f-n_{p-1}(V) \leq \#\{(e-1,j)\in S\} \leq i,
\nal
contrary to the assumption.

Therefore every inner sum in \eqref{eqn-A-xS} vanishes, so
$\AVi'([x_S])=0$ for every $S\in\scrJi$. Hence $\AVi'=0$, and
so $\AVi=0$ by Lemma \ref{lem-AVi-factorization}.
\end{proof}

\subsection{Nilpotence of \texorpdfstring{$\BVi$}{B(V,i)}}\label{subsec-B-nilpotence}

Let $\mu_\varpi: \OF \rightarrow \OF$ denote the map given by multiplication by $\varpi$. Since $\mu_{\varpi}$ factors as 
\bal
\mu_{\varpi}: \OF\xrightarrow{\,m_\varpi\,}\varpi\OF\hookrightarrow\OF,
\nal
we obtain the factorization $\mu_\varpi^* = m_\varpi^*\circ\res_{\varpi\OF}^{\OF}$ of maps on cohomology. Thus, by Theorem \ref{thm-GL2-SVi-support-formula},
\bal
\BV=(\Prx)_*\circ\mu_\varpi^*
\qquad\text{and}\qquad
\Bk=\mu_\varpi^*.
\nal
The cup product makes $H^*(\OF,k)$ a graded $k$-algebra and
$H^*(\OF,V)$ a graded module over it. The fact that $\BV$ and $\Bk$
are induced by the same pullback implies that $\BV$ is
$\Bk$-semilinear with respect to this module structure. More
precisely, we have the following.

\begin{lem}\label{lem-BV-semilinear}
For $[\phi]\in H^r(\OF,k)$ and $[x]\in H^s(\OF,V)$, we have
\bal
\mathbf{B}_{V,r+s}([\phi]\cup[x])=\mathbf{B}_{k,r}([\phi])\cup \mathbf{B}_{V,s}([x]).
\nal
\end{lem}

\begin{proof}
At the level of cochains, pullback along $\mu_\varpi$ distributes over
the cup product, and $\Prx(cv)=c\Prx(v)$ for $c\in k$. The formula
follows.
\end{proof}
We next show that, after one application of $\BV$, all further
iterates are governed by $\Bk$. 
\begin{lem}\label{lem-BV-image}
For every $[x]\in H^i(\OF,V)$, there exists
$[\phi_x]\in H^i(\OF,k)$ such that
\begin{align}\label{eqn-BV-Bk-iterates}
\BVi^N([x])=\Bki^{N-1}([\phi_x])\cup X^d
\end{align}
for every $N\geq1$.
\end{lem}

\begin{proof}
The factorization of $\Prx$ through $V^{K_U}$ shows that the image of
$\BVi$ lies in the image of
\bal
(i_{K_U})_*:H^i(\OF,V^{K_U})\longrightarrow H^i(\OF,V).
\nal
Since $V^{K_U}=kX^d$, every class in $H^i(\OF,V^{K_U})$ has the form
$[\phi]\cup X^d$ for some $[\phi]\in H^i(\OF,k)$. Thus,
\bal
\BVi([x])=[\phi_x]\cup X^d
\nal
for some $[\phi_x]\in H^i(\OF,k)$. Since $\mu_\varpi^*$ is the identity
in degree zero and $\Prx(X^d)=X^d$, we have $\BV(X^d)=X^d$. The result
now follows by induction from Lemma \ref{lem-BV-semilinear}.
\end{proof}
Thus, the nilpotence of $\BVi$ depends on the vanishing of the powers
of $\Bki$ on $H^i(\OF,k)$ and on the additional vanishing that may
occur after cupping with $X^d$. We first compute the action of $\Bk$
on the basis $[\kappa_S]$.

\begin{lem}\label{lem-Bk-action}
For $(m,j)\in\scrJ$, we have
\begin{align}\label{eqn-Bk-kappa}
\BBB_{k,1}([\kappa_{m,j}])= \begin{cases} [\kappa_{m-1,j}] &\text{if }m \geq 1,\\
0&\text{if }m=0.
\end{cases}
\end{align}
Consequently, if $S\in \scrJi$, then 
\bal
\Bki^N ([\kappa_S]) = 0 \iff  N > \min_{(m,j )\in S} m.
\nal
\end{lem}

\begin{proof} At the level of inhomogeneous cochains, for a scalar $i$-cochain $\phi$ we have
\bal
(\Bki\phi)(\alpha_1,\dots,\alpha_i) =\phi(\varpi\alpha_1,\dots,\varpi\alpha_i).
\nal
From the definition of $\kappa_{m,j}$, it is clear that, for all $\alpha \in \OF$,   
\bal
\kappa_{m,j}(\varpi\alpha) = \begin{cases}
\kappa_{m-1,j}(\alpha)&\text{if }m\geq1,\\
0&\text{if }m=0.
\end{cases}
\nal
The formula $\Bk=\mu_\varpi^*$ now gives
\eqref{eqn-Bk-kappa}. Since $\Bk$ distributes over cup products,
iteration gives
\bal
\Bki^N([\kappa_S])= \begin{cases}
[\kappa_{S-N}] &\text{if }m\geq N\text{ for every }(m,j)\in S,\\
0&\text{otherwise},
\end{cases}
\nal
where $S-N\coloneqq \{(m-N,j):(m,j)\in S\}.$ This proves the second statement.
\end{proof}

The preceding lemma describes when a scalar class reaches the layer
$m=0$. The next lemma determines which basis classes in this layer
vanish after cupping with $X^d$ in $V$-valued cohomology.  

\begin{lem}\label{lem-kappa0-y1-vanishing}
For every $j\in\{0,\dots,f-1\}$, we have
\bal
[\kappa_{0,j}]\cup X^d=0
\qquad\Longleftrightarrow\qquad
d_j>0.
\nal
Consequently, if $S\in\scrJi$ contains $(0,j)$ for some $j$ with
$d_j>0$, then
\bal
[\kappa_S]\cup X^d=0
\qquad \text{in }H^i(\OF,V).
\nal
\end{lem}

\begin{proof}
Suppose first that $d_j>0$. Then the polynomial $v_j\coloneqq X^{d-p^j}Y^{p^j}$ is an element of $V$. We have
\bal
(\kappa_{0,j}\cup X^d)(\alpha) =\kappa_j(\bar\alpha)X^d.
\nal
Since the determinant twist is trivial on $K_U$, we have
\bal
\alpha\cdot v_j-v_j = X^{d-p^j}
\left(
\bigl(Y+\bar\alpha X\bigr)^{p^j}
- Y^{p^j} \right) = \kappa_j(\bar\alpha)X^d.
\nal
Thus, this cocycle is a coboundary.

Conversely, suppose that $[\kappa_{0,j}]\cup X^d=0.$ Let $F_0\subseteq F$ be the maximal unramified subextension, which has residue field $k_F$. The restriction of the cocycle $\kappa_{0,j}\cup X^d$ to $\mathcal O_{F_0}$ is
\bal
\alpha\longmapsto\kappa_j(\bar\alpha)X^d.
\nal 
Since this restriction is a coboundary, there exists 
$v=X^dP(Z)\in V$, with $Z\coloneqq Y/X$, such that
\bal
P\bigl(Z+a\bigr)-P(Z) = \kappa_j(a) = a^{p^j}
\qquad
\text{for every }a\in k_F.
\nal
Thus, the polynomial $Q(Z)\coloneqq P(Z)-Z^{p^j}$ satisfies
\bal 
Q\bigl(Z+a\bigr)=Q(Z)
\qquad
\text{for every }a\in k_F.
\nal
It follows that $Q-Q(0)$ vanishes on the $p^f$ elements of $k_F$. Moreover,
\bal
\deg\bigl(Q-Q(0)\bigr) \leq \max\{d,p^j\} < p^f.
\nal
Hence $Q$ is constant, and so $P(Z)=Z^{p^j}+c$ for some $c\in k$. Since $v, X^d \in V$, we obtain 
\bal 
v - c X^d = X^d (Z^{p^j} + c) - c X^d = X^{d-p^j} Y^{p^j} \in V.
\nal
As $V$ is spanned by monomials $X^{d-c} Y^c$ with $c = \sum_{r=0}^{f-1} c_r p^r$ and $0 \leq c_r \leq d_r$, this forces $d_j>0$.

The final statement follows from associativity and graded
commutativity of the cup product.
\end{proof}

For the lower bound on the nilpotence index, we will also need a
criterion ensuring that a class remains nonzero after passing to
$V$-valued cohomology.

\begin{lem}\label{lem-i_U0-nonvanishing}
Let $i\geq1$. If $[\phi]\in H^i(\OF,k)$ satisfies
$\Bki([\phi])\neq0$, then the image of
$[\phi]\cup X^d \in H^i(\OF,V^{K_U})$ under the map
\bal
(i_{K_U})_*:H^i(\OF,V^{K_U})\longrightarrow H^i(\OF,V)
\nal
is nonzero.
\end{lem}

\begin{proof}
Write $\res$ for restriction from $\OF$ to $\varpi\OF$. Since
\[
\Bki=m_\varpi^*\circ\res
\]
and $m_\varpi^*$ is an isomorphism, the condition
$\Bki([\phi])\neq0$ implies $\res([\phi]) \neq 0$ in $H^i(\varpi\OF,k)$.

Consider the commutative diagram
\[
\begin{tikzcd}
H^i(\OF,V^{K_U}) \arrow[r,"{(i_{K_U})_*}"] \arrow[d,"\res"'] &
H^i(\OF,V) \arrow[d,"\res"] \\
H^i(\varpi\OF,V^{K_U}) \arrow[r,"{(i_{K_U})_*}"'] &
H^i(\varpi\OF,V).
\end{tikzcd}
\]
The bottom horizontal map is injective. Indeed, $\varpi\OF$ acts trivially on both $V^{K_U}$ and $V$, so under the identifications
\[
H^i(\varpi\OF,V^{K_U})
\cong H^i(\varpi\OF,k)\otimes_kV^{K_U},
\qquad
H^i(\varpi\OF,V)
\cong H^i(\varpi\OF,k)\otimes_kV,
\]
it is identified with $\id_{H^i(\varpi\OF,k)}\otimes i_{K_U},$ which is injective.

Now $\res([\phi]\cup X^d)$ corresponds to
$\res([\phi])\otimes X^d$, which is nonzero. Hence, its image under the bottom arrow is nonzero. By commutativity, the restriction of
$(i_{K_U})_*([\phi]\cup X^d)$ is nonzero, so
\[
(i_{K_U})_*([\phi]\cup X^d)\neq0.
\]
\end{proof}

We now combine these results to bound the nilpotence index of
$\BVi$, which we recall is defined by
\bal
N_i(V)\coloneqq \min\{N\geq1:\BVi^N=0\}.
\nal
Also recall that $n_0(V)$ is the number of indices $j$ for which $d_j = 0$. 
\begin{theorem}\label{thm-nilpotence-bound}
Let $1\leq i\leq ef$ and $q=\lfloor (i-1) / f \rfloor$, so that  $i = qf + s$ for some $s \in \{1, \dots, f\}$. 
The nilpotence index of $\BVi$ satisfies
\bal
\max\{1,e-q-1\}
\leq N_i(V)\leq
e-q+1.
\nal
If $s>n_0(V)$, then the upper bound improves to $N_i(V)\leq e-q.$
\end{theorem}

\begin{proof}
By Lemma \ref{lem-BV-image}, the operator $\BVi^N$ vanishes whenever
$\Bki^{N-1}$ vanishes on $H^i(\OF,k)$. Let
\bal
m_S\coloneqq \min_{(m,j)\in S}m \qquad \text{for all }S \in \scrJi.
\nal
Then every $S \in \scrJi$ is contained in the set $\{(m,j)\in\scrJ:m\geq m_S\},$ which has cardinality $(e- m_S)f$. Thus, 
\bal 
\#S = i \leq (e-m_S)f,
\nal 
and it follows that $m_S\leq e-q-1$. Lemma \ref{lem-Bk-action} therefore shows that
$\Bki^{N-1}=0$ for $N\geq e-q+1$. This proves the upper bound.

For the lower bound, we may assume that $q\leq e-2$, since otherwise
the claimed bound is $N_i(V) \geq 1.$  Define the set $S^* \subset \scrJ$ by
\bal
S^* \coloneqq  \{(e-q-1,j):0\leq j\leq s-1\}
\sqcup \{(m,j)\in\scrJ:m\geq e-q\}.
\nal
Then $\#S^*=qf+s=i$, so $S^* \in \scrJi$, and $m_{S^*} =e-q-1$. Let 
\bal
[\phi^*]\coloneqq \Bki^{e-q-2}([\kappa_{S^*}]) \in H^i(\OF, k), \qquad [\xi^*]\coloneqq [\kappa_{S^*}]\cup X^d \in H^i(\OF,V).
\nal
By Lemma \ref{lem-Bk-action}, $\Bki([\phi^*])$ is nonzero. Then Lemma \ref{lem-i_U0-nonvanishing} implies 
$[\phi^*]\cup X^d$ is nonzero in $H^i(\OF,V)$. Since $\BV(X^d)=X^d$, repeated applications of Lemma
\ref{lem-BV-semilinear} then yield
\bal
\BVi^{e-q-2}([\xi^*]) = [\phi^*]\cup X^d \neq 0.
\nal
Thus, $N_i(V)\geq e-q-1$. Combining this with $N_i(V)\geq1$ proves the lower bound.

It remains to prove the improved upper bound. Suppose that
$s>n_0(V)$. For $[x]\in H^i(\OF,V)$, Lemma
\ref{lem-BV-image} gives
\bal
\BVi^{e-q}([x]) =\Bki^{e-q-1}([\phi_x])\cup X^d
\nal
for some $[\phi_x]\in H^i(\OF,k)$. Consider a basis element
$[\kappa_S]$ occurring in $[\phi_x]$. If
$\Bki^{e-q-1}([\kappa_S])\neq0$, then every $(m,j)\in S$ satisfies
$m\geq e-q-1$. There are only $qf$ elements of $\scrJ$ with
$m\geq e-q$, so at least $s$ elements of $S$ have first coordinate
$e-q-1$. After applying $\Bki^{e-q-1}$, these give at least $s$
distinct factors $[\kappa_{0,j}]$. Since $s>n_0(V)$, one of these
satisfies $d_j>0$, and Lemma \ref{lem-kappa0-y1-vanishing} shows that
\bal
\Bki^{e-q-1}([\kappa_S])\cup X^d=0.
\nal
This holds for every basis element occurring in $[\phi_x]$, and therefore
$\BVi^{e-q}=0$. Hence $N_i(V)\leq e-q$.
\end{proof}

\subsection{Adjointness of \texorpdfstring{$\AVi$}{A(V,i)} and \texorpdfstring{$\BBB_{V^\vee,ef-i}$}{B(V dual, ef-i)}} Since $\OF$ is a free $\Zp$-module of rank $[F:\Qp]=ef$, there are
topological group isomorphisms
\bal
(\OF,+)\cong\Zp^{ef}
\qquad\text{and}\qquad
(\varpi\OF,+)\cong\Zp^{ef}.
\nal
Thus, both groups are Poincar\'e pro-$p$ groups of dimension $ef$; see
\cite[Propositions~1.7.7 and~3.7.6 and
Theorem~3.7.4]{NSW}.

For a topological group $H$ and a finite-dimensional continuous $k[H]$-module $E$, define
\bal 
E^{\vee}\coloneqq \Hom_k(E,k).
\nal 
The group $H$ acts on $E^{\vee}$ via
$(h\cdot\lambda)(x)=\lambda(h^{-1}\cdot x)$
for $h\in H$, $\lambda\in E^{\vee}$, and $x\in E$.

\begin{lem}\label{lem-perfect-pairing}
Let $H$ denote either $(\OF,+)$ or $(\varpi\OF,+)$. Let
$k_0\subseteq k$ be a finite extension of $\Fp$, let $E_0$ be a
finite-dimensional continuous $k_0[H]$-module, and set
$E\coloneqq E_0\otimes_{k_0}k$. For $r\in\{0,\dots,ef\}$ and any
$k$-linear isomorphism
\bal
\operatorname{tr}_H:
H^{ef}(H,k)\overset{\sim}{\longrightarrow}k,
\nal
the cup product induces a perfect pairing
\begin{align}\label{eqn-k-pairing}
H^r(H,E)\times H^{ef-r}(H,E^\vee)
&\overset{\cup}{\longrightarrow}
H^{ef}(H,E\otimes_k E^\vee)
\overset{\operatorname{ev}_*}{\longrightarrow}
H^{ef}(H,k)
\overset{\operatorname{tr}_H}{\longrightarrow}
k.
\end{align}
\end{lem}

\begin{proof}
Since $H$ is a Poincar\'e pro-$p$ group of dimension
$ef$ with trivial dualizing character, Poincar\'e duality
\cite[Theorem~3.4.6]{NSW} gives a perfect pairing
\bal
H^r(H,E_0)\times
H^{ef-r}(H,\Hom_{\Fp}(E_0,\Fp))
\longrightarrow\Fp,
\nal
induced by cup product, evaluation, and the trace map
\bal
\operatorname{tr}_{H,\Fp}:
H^{ef}(H,\Fp)\overset{\sim}{\longrightarrow}\Fp.
\nal

The trace pairing of the finite extension $k_0/\Fp$ gives an
$H$-equivariant isomorphism
\bal
\Hom_{k_0}(E_0,k_0)
\overset{\sim}{\longrightarrow}
\Hom_{\Fp}(E_0,\Fp),
\qquad
\lambda\longmapsto\Tr_{k_0/\Fp}\circ\lambda.
\nal
Under
$H^{ef}(H,k_0)\cong H^{ef}(H,\Fp)\otimes_{\Fp}k_0$, let
\bal
\operatorname{tr}_{H,k_0}^0
\coloneqq
\operatorname{tr}_{H,\Fp}\otimes\id_{k_0}.
\nal
It follows that cup product, evaluation, and
$\operatorname{tr}_{H,k_0}^0$ induce a perfect $k_0$-bilinear pairing
\bal
H^r(H,E_0)\times
H^{ef-r}(H,\Hom_{k_0}(E_0,k_0))
\longrightarrow k_0.
\nal

Extending scalars from $k_0$ to $k$ gives natural isomorphisms
\bal
H^i(H,E_0)\otimes_{k_0}k
&\overset{\sim}{\longrightarrow}
H^i(H,E),\\
H^i(H,\Hom_{k_0}(E_0,k_0))\otimes_{k_0}k
&\overset{\sim}{\longrightarrow}
H^i(H,E^\vee),
\nal
where we use the natural $H$-equivariant isomorphism
$\Hom_{k_0}(E_0,k_0)\otimes_{k_0}k\cong E^\vee$.
These isomorphisms are compatible with cup products and evaluation.
Thus, under
$H^{ef}(H,k)\cong H^{ef}(H,k_0)\otimes_{k_0}k$, the map $\operatorname{tr}_H^0
\coloneqq \operatorname{tr}_{H,k_0}^0\otimes\id_k$ 
induces a perfect $k$-bilinear pairing \eqref{eqn-k-pairing}.

Finally, $H^{ef}(H,k)$ is one-dimensional over $k$, so any other
$k$-linear isomorphism $\operatorname{tr}_H:H^{ef}(H,k)\to k$ is a
nonzero scalar multiple of $\operatorname{tr}_H^0$. The resulting
pairing is therefore also perfect.
\end{proof}

\begin{lem}\label{lem-cores-restriction-adjoint} Let
$k_0\subseteq k$ be a finite extension of $\Fp$, let $E_0$ be a
finite-dimensional continuous $k_0[\OF]$-module, and set
$E\coloneqq E_0\otimes_{k_0}k$. Fix a $k$-linear isomorphism
\bal
\operatorname{tr}_{\OF}:H^{ef}(\OF,k)\overset{\sim}{\longrightarrow}k,
\nal
and define
\bal
\operatorname{tr}_{\varpi \OF} \coloneqq \operatorname{tr}_{\OF}\circ\operatorname{cores}_{\varpi \OF}^{\OF} :
H^{ef}(\varpi \OF,k)\longrightarrow k.
\nal
Then $\operatorname{tr}_{\varpi \OF}$ is an isomorphism. Let
$\langle\cdot,\cdot\rangle_{\OF}$ and
$\langle\cdot,\cdot\rangle_{\varpi \OF}$ denote the perfect pairings of
Lemma \ref{lem-perfect-pairing} defined using
$\operatorname{tr}_{\OF}$ and $\operatorname{tr}_{\varpi \OF}$, respectively.
Then for every $r\in\{0,\ldots,ef\}$, the corestriction and restriction maps
are adjoint with respect to these pairings; that is,
\bal
\left\langle
\operatorname{cores}_{\varpi \OF}^{\OF}(\xi)  , \eta
\right\rangle_{\OF}
=\left\langle
\xi , \operatorname{res}_{\varpi \OF}^{\OF}(\eta)
\right\rangle_{\varpi \OF} 
\nal 
for all $\xi \in H^r(\varpi \OF, E)$ and $\eta \in H^{ef -r}(\OF, E^{\vee}).$ 
\end{lem}
\begin{proof}
Since $\varpi\OF$ is an open normal subgroup of $\OF$,
$\operatorname{cd}_p(\OF)=ef$, and the action of $\OF/\varpi\OF$ on $H^{ef}(\varpi\OF,k_0)$ is trivial, \cite[Proposition~3.3.11]{NSW} gives an isomorphism
\bal
\operatorname{cores}_{\varpi\OF}^{\OF}:
H^{ef}(\varpi\OF,k_0)\overset{\sim}{\longrightarrow}
H^{ef}(\OF,k_0).
\nal
Extending scalars from $k_0$ to $k$ therefore gives an isomorphism
\bal
\operatorname{cores}_{\varpi\OF}^{\OF}:
H^{ef}(\varpi\OF,k)\overset{\sim}{\longrightarrow}
H^{ef}(\OF,k).
\nal
Hence, the composition
$\operatorname{tr}_{\varpi\OF}
= \operatorname{tr}_{\OF}\circ
\operatorname{cores}_{\varpi\OF}^{\OF}$
is also an isomorphism.

Now let $\xi\in H^r(\varpi\OF,E)$ and
$\eta\in H^{ef-r}(\OF,E^\vee)$. By \cite[Proposition~1.5.3(iv)]{NSW}, 
\begin{align}\label{eqn-NSW-cores-cup}
\operatorname{cores}_{\varpi\OF}^{\OF}
\left( \xi\cup\operatorname{res}_{\varpi\OF}^{\OF}(\eta)
\right)
=\operatorname{cores}_{\varpi\OF}^{\OF}(\xi)\cup\eta.
\end{align}
Corestriction is functorial in the coefficient module
\cite[Proposition~1.5.2]{NSW}, so it commutes with the map induced by
evaluation $E \otimes_k E^\vee\to k$. Therefore,
\bal
\left\langle \operatorname{cores}_{\varpi\OF}^{\OF}(\xi) ,  \eta
\right\rangle_{\OF} &= \operatorname{tr}_{\OF}
\left( \operatorname{ev}_*
\bigl( \operatorname{cores}_{\varpi\OF}^{\OF}(\xi)\cup\eta
\bigr) \right) \\ 
&= \operatorname{tr}_{\OF}
\left( \operatorname{cores}_{\varpi\OF}^{\OF}
\left( \operatorname{ev}_*
\bigl( \xi\cup\operatorname{res}_{\varpi\OF}^{\OF}(\eta)
\bigr) \right)
\right)\\
&= \operatorname{tr}_{\varpi\OF}\left( \operatorname{ev}_*
\bigl( \xi\cup\operatorname{res}_{\varpi\OF}^{\OF}(\eta)
\bigr) \right)\\
&= \left\langle \xi , \operatorname{res}_{\varpi\OF}^{\OF}(\eta)
\right\rangle_{\varpi\OF}.
\nal
Thus, with the normalization
$\operatorname{tr}_{\varpi\OF}=\operatorname{tr}_{\OF}\circ
\operatorname{cores}_{\varpi\OF}^{\OF}$,
corestriction is exactly adjoint to restriction.
\end{proof}

Recall from Subsection~\ref{subsec-GL2-notation} the definitions of the
coefficient maps $\Pry$ and $\Prx$. For the remainder of this
subsection, if $V$ is a Serre weight, we adorn these maps and the constituent maps
appearing in their definitions with a subscript $V$ to indicate the
coefficient representation. Thus,
\bal
\operatorname{Pr}_{V,Y}=i_{V,\KUbar}\circ\eta_{V,Y}\circ q_{V,K_U}
\qquad\text{and}\qquad
\operatorname{Pr}_{V,X}=i_{V,K_U}\circ\eta_{V,X}\circ q_{V,\KUbar},
\nal
where $\eta_{V,Y}$ and $\eta_{V,X}$ are the inverses of the isomorphisms
$q_{V,K_U}\circ i_{V,\KUbar}$ and
$q_{V,\KUbar}\circ i_{V,K_U}$, respectively.


\begin{lem}\label{lem-coefficient-duality}
For every Serre weight $V$, we have $\left(\operatorname{Pr}_{V,Y}\right)^\vee
=\operatorname{Pr}_{V^\vee,X}.$ Equivalently,  
\bal
\lambda\bigl(\operatorname{Pr}_{V,Y}(v)\bigr)
=\bigl(\operatorname{Pr}_{V^\vee,X}(\lambda)\bigr)(v) \qquad \text{for all }v \in V, \ \lambda \in V^{\vee}. 
\nal
\end{lem}
\begin{proof}
By \cite[proof of Proposition 1.30]{herzigmod} and finite-dimensional
biduality, for any finite-dimensional representation $W$ of
$H\in\{K_U,\KUbar\}$, we have canonical isomorphisms
\bal
(W^H)^\vee\cong(W^\vee)_H
\qquad\text{and}\qquad
(W_H)^\vee\cong(W^\vee)^H.
\nal
Under these identifications, we have
\bal
\left(q_{W,H}\right)^\vee=i_{W^\vee,H}
\qquad\text{and}\qquad
\left(i_{W,H}\right)^\vee=q_{W^\vee,H}.
\nal
Since $\eta_{V,Y}=(q_{V,K_U}\circ i_{V,\KUbar})^{-1}$, we obtain
\bal
\left(\eta_{V,Y}\right)^\vee
= \left(\left(i_{V,\KUbar}\right)^\vee
\circ\left(q_{V,K_U}\right)^\vee\right)^{-1}
=\left(q_{V^\vee,\KUbar}\circ i_{V^\vee,K_U}\right)^{-1}
=\eta_{V^\vee,X}.
\nal
Note that $V^\vee$ is also a Serre weight. Therefore,
\bal
\left(\operatorname{Pr}_{V,Y}\right)^\vee&=\left(i_{V,\KUbar}\circ\eta_{V,Y}\circ q_{V,K_U}\right)^\vee\\
&=\left(q_{V,K_U}\right)^\vee\circ\left(\eta_{V,Y}\right)^\vee\circ \left(i_{V,\KUbar}\right)^\vee\\
&=i_{V^\vee,K_U}\circ\eta_{V^\vee,X}\circ q_{V^\vee,\KUbar}
=\operatorname{Pr}_{V^\vee,X}.
\nal
\end{proof}

\begin{lem}\label{lem-multiplication-trace}
There exists a unique scalar $c_{\varpi}\in k^\times$ such that
\bal
\operatorname{cores}_{\varpi\OF}^{\OF}\circ m_{\varpi^{-1}}^*
=
c_{\varpi}\,\id_{H^{ef}(\OF,k)}.
\nal
Equivalently, $\operatorname{tr}_{\varpi\OF}\circ m_{\varpi^{-1}}^* =c_{\varpi}\,\operatorname{tr}_{\OF}.$
\end{lem}

\begin{proof} Since $m_{\varpi^{-1}}:\varpi\OF\overset{\sim}{\longrightarrow}\OF$ is an isomorphism, so is the induced map $m_{\varpi^{-1}}^*$. The corestriction $\operatorname{cores}_{\varpi\OF}^{\OF}$ is also an isomorphism by Lemma \ref{lem-cores-restriction-adjoint}. Thus, their composition is an automorphism of the
1-dimensional $k$-vector space $H^{ef}(\OF,k)$, and hence equals
$c_{\varpi}\,\id$ for a unique $c_{\varpi}\in k^\times$. The second
identity follows from the definition
$\operatorname{tr}_{\varpi\OF}
=\operatorname{tr}_{\OF}\circ
\operatorname{cores}_{\varpi\OF}^{\OF}$.
\end{proof}

\begin{remark}
One can show using a direct computation that, if $F/\Q_p$ is unramified and $\varpi=p$, we have $c_\varpi=c_p=1$.  The idea of the proof is as follows.  Recall that $H^{ef}(\OF,k)=H^f(\OF,k)$ is 1-dimensional with basis 
$$[\kappa_\scrJ]=[\kappa_{0,0}]\cup\cdots\cup [\kappa_{0,f-1}].$$
As an inhomogeneous cochain, 
$$\cores_{p\OF}^{\OF}(m_{p\inv}^*(\kappa_\scrJ))(x_0,\ldots, x_{f-1})=\sum_{a\in k_F} \prod_{j=0}^{f-1} \overline{\theta_{a+\overline{s_{j-1}}}(x_j)}^{p^j}$$
where $s_{j-1}=x_0+\cdots+x_{j-1}$ (and is zero if $j=0$). 
One can use an argument similar to that in the proof of Lemma \ref{lem-theta-mod-p}, notably using Witt vector computations, to show that this expression is precisely equal to
$$(-1)^{f+1}\prod_{j=0}^{f-1}\overline{x_j}^{p^{j-1}},$$
which is a computation that heavily relies on $x_j-[\overline{x_j}]$ being a multiple of $p$.  This proves that 
$$\cores_{p\OF}^{\OF}(m_{p\inv}^*(\kappa_\scrJ))=(-1)^{f+1}\kappa_{0,f-1}\cup\kappa_{0,0}\cup\cdots \cup \kappa_{0,f-2}$$
as cochains. One can conclude using the graded commutativity of the cup product that, after taking cohomology, we have 
$$\cores_{p\OF}^{\OF}(m_{p\inv}^*([\kappa_\scrJ]))=[\kappa_\scrJ],$$
hence $c_\varpi=c_p=1$ if $F/\Q_p$ is unramified.  By contrast, in the ramified setting, the constant will depend in general on our choice of uniformizer, so the computation of $c_\varpi$ is more intricate.
\end{remark}

\begin{theorem}\label{thm-AV-BV-duality}
Let $V$ be a Serre weight and let $i\in\{0,\dots,ef\}$. For every
$\xi\in H^i(\OF,V)$ and
$\eta\in H^{ef-i}(\OF,V^\vee)$, we have
\begin{align}\label{eqn-AB-pairing}
\left\langle \AVi(\xi),\eta\right\rangle_{\OF}
=c_{\varpi} \left\langle \xi, \BBB_{V^\vee,ef-i}(\eta)\right\rangle_{\OF}.
\end{align}
\end{theorem}

\begin{proof}
By the definition of a Serre weight, there is a finite-dimensional
$k_F[K]$-module $V_0$ such that $V\cong V_0\otimes_{k_F}k$; the same
remains true after duality, restriction, and pullback.
By Theorem \ref{thm-GL2-SVi-support-formula},
\bal
\AVi =\operatorname{cores}_{\varpi\OF}^{\OF}
\circ(\operatorname{Pr}_{V,Y})_*
\circ m_{\varpi^{-1}}^*.
\nal
Thus, Lemma \ref{lem-cores-restriction-adjoint} and
Lemma \ref{lem-coefficient-duality} give us
\bal
\left\langle \AVi(\xi),\eta\right\rangle_{\OF}
&=\left\langle(\operatorname{Pr}_{V,Y})_* \ m_{\varpi^{-1}}^*(\xi),
\operatorname{res}_{\varpi\OF}^{\OF}(\eta)
\right\rangle_{\varpi\OF}\\
&=\left\langle
m_{\varpi^{-1}}^*(\xi),
(\operatorname{Pr}_{V^\vee,X})_*
\operatorname{res}_{\varpi\OF}^{\OF}(\eta)
\right\rangle_{\varpi\OF}.
\nal
The second equality follows from the identity
\bal
\operatorname{ev}\bigl(\operatorname{Pr}_{V,Y}(v),\lambda\bigr)
=
\operatorname{ev}\bigl(v,\operatorname{Pr}_{V^\vee,X}(\lambda)\bigr)
\nal
and the naturality of cup products with respect to maps of coefficient
modules. Thus,
\bal
\left\langle \AVi(\xi),\eta\right\rangle_{\OF}
= \left\langle m_{\varpi^{-1}}^*(\xi),\zeta \right\rangle_{\varpi\OF},
\nal
where, using the isomorphism $\bigl(m_{\varpi^{-1}}^*V\bigr)^\vee \cong m_{\varpi^{-1}}^*(V^\vee),$ we set
\bal
\zeta \coloneqq  (\operatorname{Pr}_{V^\vee,X})_* \operatorname{res}_{\varpi\OF}^{\OF}(\eta) \in H^{ef-i}\bigl( \varpi\OF,m_{\varpi^{-1}}^*(V^\vee) \bigr).
\nal
Since $m_{\varpi^{-1}}$ and $m_\varpi$ are inverse isomorphisms, we have
\bal
m_\varpi^*(\zeta)\in H^{ef-i}(\OF,V^\vee)
\qquad\text{and}\qquad
\zeta = m_{\varpi^{-1}}^* m_\varpi^* (\zeta).
\nal
By the naturality of cup products and evaluation, we therefore have
\bal
\operatorname{ev}_*
\bigl(m_{\varpi^{-1}}^*(\xi)\cup\zeta\bigr)
= m_{\varpi^{-1}}^* \operatorname{ev}_* \bigl(\xi\cup m_\varpi^*(\zeta)\bigr).
\nal
Applying the trace maps and Lemma
\ref{lem-multiplication-trace}, we obtain
\bal
\left\langle
m_{\varpi^{-1}}^*(\xi),\zeta
\right\rangle_{\varpi\OF}
=
c_{\varpi}
\left\langle
\xi,m_\varpi^*(\zeta)
\right\rangle_{\OF}.
\nal
Finally, naturality of pullback with respect to coefficient maps and
Theorem \ref{thm-GL2-SVi-support-formula} yield
\bal
m_\varpi^*(\zeta)
=(\operatorname{Pr}_{V^\vee,X})_* m_\varpi^*
\operatorname{res}_{\varpi\OF}^{\OF}(\eta)
=\BBB_{V^\vee,ef-i}(\eta).
\nal
\end{proof}

\begin{cor}\label{cor-AV-BV-nilpotence}
Let $i\in\{0,\dots,ef\}$. For every $n\geq1$ and every
$\xi\in H^i(\OF,V)$, $\eta\in H^{ef-i}(\OF,V^\vee)$, 
\begin{align}\label{eqn-AVi-BVefi-pairing}
\left\langle \AVi^n(\xi),\eta\right\rangle_{\OF}
= c_{\varpi}^n
\left\langle
\xi,\BBB_{V^\vee,ef-i}^n(\eta)
\right\rangle_{\OF}.
\end{align}
It follows that 
\begin{align}\label{eqn-AVi-BVe-n-power-vanishing}
\AVi^n=0
\iff 
\BBB_{V,ef-i}^n=0.
\end{align}
In particular, we have 
\begin{align}\label{eqn-Mi-Nefi}
M_i(V)=N_{ef-i}(V)
\qquad
\text{for every }0\leq i\leq ef-1. 
\end{align}
\end{cor}

\begin{proof}
The identity \eqref{eqn-AVi-BVefi-pairing} follows from Theorem
\ref{thm-AV-BV-duality} via induction on $n$. Since the pairing is perfect, this implies 
\begin{align}\label{eqn-in-proof-AVi-BVefvee}
\AVi^n=0
\iff
\BBB_{V^\vee,ef-i}^n=0.
\end{align}

By \cite[Exercise~3.22]{herzigmod}, there is a character
$\chi_V:k_F^\times\to k^\times$ and a $K$-equivariant isomorphism
\bal
\Phi:V^\vee\overset{\sim}{\longrightarrow}
W\coloneqq V\otimes(\chi_V\circ\det).
\nal
Since the character $\chi_V\circ\det$ is trivial on $K_U$ and $\KUbar$, the restrictions of $W$ and $V$ to these subgroups are equal as representations. It follows that
\bal
W^{K_U}=V^{K_U},
\qquad
W_{\KUbar}=V_{\KUbar},
\qquad
\operatorname{Pr}_{W,X}=\operatorname{Pr}_{V,X}.
\nal
The actions of $\OF$ and $\varpi\OF$, viewed as subgroups of $K_U$,
are likewise equal on $W$ and $V$. Thus the restriction and pullback
maps in the formula for $\BBB_{W,j}$ and $\BBB_{V,j}$ agree, which yields
\bal
\BBB_{W,j}=\BBB_{V,j}.
\nal
On the other hand, naturality with respect to the $K$-equivariant
isomorphism $\Phi$ gives
\bal
\Phi_*\circ\BBB_{V^\vee,j}
=
\BBB_{W,j}\circ\Phi_*
=
\BBB_{V,j}\circ\Phi_*.
\nal
Therefore $\BBB_{V^\vee,j}$ and $\BBB_{V,j}$ are conjugate, and so
\bal
\BBB_{V^\vee,j}^n=0 \iff \BBB_{V,j}^n=0
\nal
for every $n\geq1$. Hence,
\eqref{eqn-AVi-BVe-n-power-vanishing} follows from
\eqref{eqn-in-proof-AVi-BVefvee}.

If $0\leq i\leq ef-1$, then $ef-i\geq1$, so Theorem
\ref{thm-nilpotence-bound} shows that
$\BBB_{V^\vee,ef-i}$ is nilpotent. Thus $\AVi$ is nilpotent
with the same nilpotence index, which proves \eqref{eqn-Mi-Nefi}. 
\end{proof}

\subsection{The image of \texorpdfstring{$\SVi$}{S(V,i)}}\label{subsec-image}
We now combine the preceding vanishing, nilpotence, and adjointness results to describe the image of $\SVi$. This description is made possible by the observation that the operators $\AVi$ and $\BVi$ annihilate one another when $V$ is non-Steinberg. Consequently, every convolution power of $\SVi(T_1)$ has only a purely dominant and a purely antidominant contribution. The nilpotence of $\BVi$ truncates the antidominant contribution, and by Corollary \ref{cor-AV-BV-nilpotence}, the dominant contribution is also truncated in every degree $1\leq i\leq ef-1$. Thus, in the interior cohomological degrees, the image of $\SVi$ is supported on a finite nilpotent thickening of the diagonal.

\begin{lem}\label{lem-AV-BV-orthogonal}
If $V$ is non-Steinberg (i.e., if $d < p^f-1$), then
\bal
\AVi\circ\BVi=\BVi\circ\AVi=0.
\nal
\end{lem}

\begin{proof}
By Theorem \ref{thm-GL2-SVi-support-formula} and the naturality of
pullback with respect to coefficient maps, we have
\bal
\AVi\circ\BVi
=
\cores_{\varpi\OF}^{\OF}
\circ(\Pry\circ\Prx)_*
\circ\res_{\varpi\OF}^{\OF}.
\nal
If $d>0$, then $\Pry\circ\Prx=0$, so
$\AVi\circ\BVi=0$. If $d=0$, then $V$ is a one-dimensional
determinant character; hence, $\Prx=\Pry=\id_V$, and $\OF$ acts
trivially on $V$. Therefore, $d=0$ implies
\bal
\AVi\circ\BVi =\cores_{\varpi\OF}^{\OF} \circ\res_{\varpi\OF}^{\OF}=[\OF:\varpi\OF]\,\id
=p^f\,\id=0.
\nal

For the reverse composition, Theorem
\ref{thm-GL2-SVi-support-formula} gives us
\bal
\BVi\circ\AVi
= (\Prx)_* \circ m_\varpi^* \circ\res_{\varpi\OF}^{\OF}
\circ\cores_{\varpi\OF}^{\OF} \circ(\Pry)_*
\circ m_{\varpi^{-1}}^*.
\nal
Since $\OF$ is abelian, the restriction--corestriction formula yields
\bal
\res_{\varpi\OF}^{\OF}
\circ\cores_{\varpi\OF}^{\OF}
=\sum_{a\in k_F}(a \cdot)_*
\; \in \End_k H^i(\varpi\OF,V).
\nal
Since pullback commutes with coefficient maps and
$m_\varpi^*\circ m_{\varpi^{-1}}^*=\id$, it follows that
$\BVi\circ\AVi$ is the map on cohomology induced by the coefficient
endomorphism
\bal
\sum_{a\in k_F}\Prx\circ a\circ\Pry
\in\End_k(V).
\nal
If $d>0$ and $\Pry(v)=cY^d$, then
\bal
\sum_{a\in k_F}\Prx\bigl(a\cdot\Pry(v)\bigr)
=cX^d\sum_{a\in k_F}a^d =0,
\nal
because $d<p^f-1$. If $d=0$, then $\Prx=\Pry=\id_V$ and every
$a\in k_F$ acts trivially on $V$, so the preceding coefficient
endomorphism is
\bal
\sum_{a\in k_F}\id_V
=
p^f\,\id_V
=
0.
\nal
Thus, $\BVi\circ\AVi=0$.
\end{proof}

The adjointness results determine the operators
in the extreme cohomological degrees.

\begin{lem}\label{lem-extreme-degree-operators}
If $V$ is non-Steinberg, then
\bal
\mathbf{A}_{V,0}=0,
\qquad
\mathbf{B}_{V,0}=\id,
\qquad
\mathbf{A}_{V,ef}=c_\varpi\,\id,
\qquad
\mathbf{B}_{V,ef}=0.
\nal
\end{lem}

\begin{proof}
Since $d<p^f-1$, we have $n_{p-1}(V)<f$, so Theorem
\ref{thm-A-vanishing} implies $\mathbf{A}_{V,0}=0$. Moreover,
\bal
H^0(\OF,V)=V^{K_U}.
\nal
Pullback along $\mu_\varpi$ fixes $V^{K_U}$ and
$\Prx$ restricts to the identity on $V^{K_U}$, so $\mathbf{B}_{V,0}=\id$.

The same argument gives $\mathbf{B}_{V^\vee,0}=\id$. By applying
Theorem \ref{thm-AV-BV-duality} with $i=ef$, we obtain
\bal
\left\langle
\mathbf{A}_{V,ef}(\xi),\eta
\right\rangle_{\OF}
=
c_\varpi
\left\langle
\xi,\eta
\right\rangle_{\OF}.
\nal
Since the pairing is perfect,
$\mathbf{A}_{V,ef}=c_\varpi\,\id$. Finally, $M_0(V)=1$ since $\mathbf{A}_{V,0}=0$, so Corollary
\ref{cor-AV-BV-nilpotence} gives us $N_{ef}(V)=1$. Thus, $\mathbf{B}_{V,ef}=0$.
\end{proof}

Next we determine the convolution powers of $\AVi$ and $\BVi$, along with the diagonal translations arising from
$\SVi(T_2)$.

\begin{cor}\label{cor-SVi-convolution-formulas}
Let $V$ be a non-Steinberg Serre weight. Then for every $m, n, r, s \in\Z$ with $n\geq1$,  
\bal
\left(
\SVi(T_2)^{\star m}\star\SVi(T_1)^{\star n}
\right)(t_{r,s})
=
\begin{cases}
\AVi^n,
&\text{if }(r,s)=(m+n,m),\\
\BVi^n,
&\text{if }(r,s)=(m,m+n),\\
0,
&\text{otherwise}.
\end{cases}
\nal
For $n=0$, the function $\SVi(T_2)^{\star m}$ is supported on
$K_Tt_{m,m}$ and satisfies
\bal
\SVi(T_2)^{\star m}(t_{m,m})=\id.
\nal
\end{cor}

\begin{proof}
The statement for $T_2$ follows from Theorem
\ref{thm-GL2-SVi-support-formula}, since $T_2$ is invertible and
$\SVi(T_2)(t_{1,1})=\id$. For $T_1$, the case $n=1$ is the same
theorem. Since $T/K_T\cong\Z^2$, every convolution summand in
$\SVi(T_1)^{\star n}$ corresponds to a word of length $n$ in the two
support classes $t_{1,0}$ and $t_{0,1}$, and its value is the
corresponding composition of $\AVi$ and $\BVi$. By Lemma
\ref{lem-AV-BV-orthogonal}, every mixed composition is zero. The
all-$\AVi$ term is supported at $t_{n,0}$, while the all-$\BVi$ term
is supported at $t_{0,n}$. Convolution with
$\SVi(T_2)^{\star m}$ translates these support classes by
$t_{m,m}$ without changing their values.
\end{proof}

We can now determine the support of the image exactly. We find that $\mathcal{S}_V^0$ is supported on the antidominant cone, $\mathcal{S}_V^{ef}$ is supported on the dominant cone, and for all interior degrees, $\SVi$ is supported on a strip of finite width along the diagonal.
\begin{theorem}\label{thm-GL2-image}
Let $V$ be a non-Steinberg Serre weight. Then, for every
$0\leq i\leq ef$,
\bal
\Supp(\operatorname{im}\SVi)
=
\begin{cases}
K_T\{t_{r,s}:r-s\leq0\},
&\text{if }i=0,\\[.4em]
K_T\{t_{r,s}:-N_i(V)<r-s<N_{ef-i}(V)\},
&\text{if }1\leq i\leq ef-1,\\[.4em]
K_T\{t_{r,s}:r-s\geq0\},
&\text{if }i=ef.
\end{cases}
\nal
\end{theorem}

\begin{proof}
Since $\mathcal{H}_G(V)=k[T_1,T_2^{\pm1}]$, the image of $\SVi$ is
spanned by
\bal
\SVi(T_2)^{\star m}\star\SVi(T_1)^{\star n},
\qquad
m\in\Z,
\quad
n\geq0.
\nal
Suppose first that $1\leq i\leq ef-1$. By Corollary
\ref{cor-AV-BV-nilpotence}, $M_i(V)=N_{ef-i}(V).$ Hence, 
\bal
\AVi^n\neq0
\iff
0\leq n<N_{ef-i}(V) \quad \text{and} \quad
\BVi^n\neq0
\iff
0\leq n<N_i(V),
\nal
where the case $n=0$ denotes the identity contribution on the
diagonal. Corollary \ref{cor-SVi-convolution-formulas} now shows that
the dominant support classes are precisely those with
\bal
0\leq r-s<N_{ef-i}(V),
\nal
while the antidominant support classes are precisely those with
\bal
-N_i(V)<r-s\leq0.
\nal
Their union is the asserted support. 

When $i=0$, Lemma \ref{lem-extreme-degree-operators} gives
$\mathbf{A}_{V,0}=0$ and $\mathbf{B}_{V,0}=\id$, so only the
antidominant classes occur. When $i=ef$, it gives
$\mathbf{A}_{V,ef}=c_\varpi\,\id$ and $\mathbf{B}_{V,ef}=0$, so only
the dominant classes occur.  
\end{proof}

For any Serre weight $V$ and any $1\leq i\leq ef-1$, set
\bal
R_i(V)\coloneqq \max\{N_i(V),N_{ef-i}(V)\}.
\nal
Thus, $R_i(V)=R_{ef-i}(V)$.

\begin{cor}\label{cor-exact-image}
Let $V$ be a non-Steinberg Serre weight. Then for every $0\leq i\leq ef$, there is a $k$-algebra isomorphism
\bal
\operatorname{im}(\SVi)
\overset{\sim}{\longrightarrow}
\begin{cases}
k[x,y^{\pm1}],
&\text{if }i\in\{0,ef\},\\[.4em]
k[x,y^{\pm1}]/\bigl(x^{R_i(V)}\bigr),
&\text{if }1\leq i\leq ef-1,
\end{cases}
\nal
which sends $\SVi(T_1)$ to $x$ and $\SVi(T_2)$ to $y$. Moreover, for every $1\leq i\leq ef-1$, 
\begin{align}\label{eqn-Ri-bounds}
R_i(V)\in
\{e-1-r_i,\ e-r_i,\ e+1-r_i\},
\qquad
r_i\coloneqq 
\left\lfloor
\frac{\min\{i,ef-i\}-1}{f}
\right\rfloor.
\end{align}
\end{cor}

\begin{proof}
Assume first that $1\leq i\leq ef-1$. By Corollary
\ref{cor-SVi-convolution-formulas},
\bal
\SVi(T_1)^{\star n}=0
\iff
\AVi^n=\BVi^n=0.
\nal
Using $M_i(V)=N_{ef-i}(V)$, this is equivalent to $n\geq R_i(V).$ Thus, $x^{R_i(V)}$ is the only possible relation involving the image
of $T_1$. Indeed, for distinct pairs $(m,n)$ with $m\in\Z$ and
$n\geq0$, Corollary \ref{cor-SVi-convolution-formulas} shows that
the nonzero images of the monomials $T_2^m T_1^n$ have disjoint
supports. They are therefore linearly independent. This proves the
description of the image in the interior degrees.

If $i=0$ or $i=ef$, Lemma \ref{lem-extreme-degree-operators} shows
that no positive convolution power of $\SVi(T_1)$ vanishes. The same
support argument shows that there are no further relations, and so $\operatorname{im}(\SVi) \cong k[x,y^{\pm 1}]$.

Finally, \eqref{eqn-Ri-bounds} follows by applying Theorem
\ref{thm-nilpotence-bound} to $i$ and $ef-i$ and taking the maximum.
\end{proof}

In the unramified case, the image in every interior degree is either
$k[y^{\pm1}]$ or a square-zero thickening of it. For each Serre weight $V$, define the integer $n(V)\coloneqq \min\{n_0(V),n_{p-1}(V)\}$. 

\begin{cor}\label{cor-unramified-image}
Let $F/\Q_p$ be unramified and let $V$ be a Serre weight. Then for every $1 \leq i \leq f-1$, there is a $k$-algebra isomorphism
    \begin{align}\label{eqn-in-cor-unram-image}
       \operatorname{im}(\SVi) \overset{\sim}{\longrightarrow} k[x,y^{\pm1}]/\bigl(x^{R_i(V)}\bigr)
    \quad \text{with}\quad R_i(V)\in\{1,2\},
    \end{align}
    which sends $\SVi(T_1)$ to $x$ and $\SVi(T_2)$ to $y$. Moreover, we have $R_i(V) = 1$, and hence $\operatorname{im}\SVi \cong k[y^{\pm 1}]$, if either of the following conditions holds: 
    \begin{enumerate}
        \item $n(V)<i<f-n(V)$; \quad \text{or}
        \item  $p\geq5$ and $(d_0,\ldots,d_{f-1})$ does not contain a cyclically contiguous subsequence of the form
    \begin{align}\label{eqn-bad-subsequence}
    (d_s, d_{s+1}, \dots, d_{s+t-1}, d_{s+t}) = (p-1,p-2,\dots,p-2,0)
    \end{align}
    for some $s\in\Z/f\Z$ and $t\geq1$, where the indices are taken
    modulo $f$.

    \end{enumerate}
\end{cor}

\begin{proof} Suppose first that $V$ is Steinberg. Then $d_j=p-1$ for every $j$, so $n_0(V)=0$. Since $e=1$, the improved bound of Theorem
\ref{thm-nilpotence-bound} yields $N_i(V)\leq 1$ for every
$1\leq i\leq f$. By definition, $N_i(V) \geq 1$, and so
$N_i(V)=1$ for every $1\leq i\leq f$. In particular,
$R_i(V)=1$ for every $1\leq i\leq f-1$.
Corollary \ref{cor-AV-BV-nilpotence} then gives $M_i(V)=1$ for every
$0\leq i\leq f-1$. Thus, for every $1\leq i\leq f-1$, we have
$\AVi=\BVi=0$, and so Theorem \ref{thm-GL2-SVi-support-formula} gives
$\SVi(T_1)=0$.

Moreover, by Theorem \ref{thm-GL2-SVi-support-formula},
$\SVi(T_2)^{\star m}$ is supported on $K_T t_{m,m}$ and is nonzero there
for every $m\in\Z$. Hence these Laurent powers are linearly independent.
Since $\mathcal H_G(V)=k[T_1,T_2^{\pm1}]$ and $\SVi(T_1)=0$, we obtain $\operatorname{im}(\SVi)\cong k[y^{\pm1}],$ with $\SVi(T_2)$ corresponding to $y$.

We now assume that $V$ is non-Steinberg. Since $e=1$, Theorem \ref{thm-nilpotence-bound} gives
$N_j(V)\in\{1,2\}$ for every $1\leq j\leq f$. By Corollary
\ref{cor-AV-BV-nilpotence}, we also have $M_i(V)=N_{f-i}(V)$ for every $1\leq i\leq f-1$. Now \eqref{eqn-in-cor-unram-image} follows from Corollary \ref{cor-exact-image}. 

We next prove the assertion involving $n(V)$. The improved bound in
Theorem \ref{thm-nilpotence-bound} gives $N_j(V)=1$ whenever
$j>n_0(V)$. If $j>n_{p-1}(V)$, then
$f-j<f-n_{p-1}(V)$, so Theorem \ref{thm-A-vanishing} shows that
$\mathbf{A}_{V,f-j}=0$. Corollary
\ref{cor-AV-BV-nilpotence} then gives us
$N_j(V)=M_{f-j}(V)=1$. Consequently, $N_j(V)=1$ whenever
$j>n(V)$. If $n(V)<i<f-n(V)$, then both $i$ and $f-i$ are greater than
$n(V)$. Hence $N_i(V)=N_{f-i}(V)=1$, so $R_i(V)=1$. 

It remains to prove that $R_i(V)=1$ under the second condition. If no subsequence of the form \eqref{eqn-bad-subsequence} exists, and moreover $(d_0, \dots, d_{f-1}) \neq (p-2, \dots, p-2)$, then \cite[Lemma 2.5.4]{koziol2024derived}\footnote{See also Lemma 0.0.2 and Remark 0.0.3 in the erratum to \cite{koziol2024derived}.} decomposes $H^i(\OF,V)$ into a direct sum of $\binom{f}{i}$ distinct $K_T$-characters. It then follows from \cite[Proposition~6.2]{herzig2011classification}
that $\mathcal{H}_T(H^i(\OF,V))$ decomposes as a direct product of
$\binom{f}{i}$ Hecke algebras, each isomorphic to
$k[x^{\pm1},y^{\pm1}]$. In particular, $\mathcal{H}_T(H^i(\OF,V))$ is reduced in this case. Thus, 
$$\operatorname{im}(\SVi) \subset \mathcal{H}_T(H^i(\OF,V))$$
must also be reduced (which forces $R_i(V) = 1$) if the second condition holds and $(d_0, \dots, d_{f-1}) \neq (p-2, \dots, p-2).$ 

Lastly, note that $n(V) = 0$ if $d_j = p-2$ for all $j$, so in that case we also have $R_i(V)=1$ for all $i$ by the first condition.  
\end{proof}

\subsection{Examples and sharpness of the nilpotence bounds}

We now give examples illustrating how the nilpotence bounds govern the thickness of the image along the diagonal. The first realizes the middle value in \eqref{eqn-Ri-bounds} in every interior cohomological degree, while the latter two realize the upper and lower values for $R_1(V)$.

We first consider Serre weights that are determinant twists of the trivial character. In this case, the exact formula for $R_i(V)$ gives a reduced image when $F/\Qp$ is unramified and a nontrivial nilpotent thickening in every interior degree when $F/\Qp$ is ramified. In the ramified case, the thickening is largest in the extreme interior degrees and becomes thinner in steps toward the middle.

\begin{example}
Let $F/\Qp$ be a nontrivial extension, let $V$ be a Serre weight for
$\GL_2(F)$ with $d_j=0$ for every $j$, and let
$i\in\{1,\dots,ef-1\}$. Since $d=0$, the representation $V$ is one-dimensional, $\OF$ acts trivially on $V$, and $\Prx=\id_V$. Moreover, $X^d$ spans $V$, so
every class in $H^i(\OF,V)$ has the form $[\phi]\cup X^d$ for a unique
$[\phi]\in H^i(\OF,k)$. Then Lemma~\ref{lem-BV-semilinear} gives us
\bal
\BVi^N\bigl([\phi]\cup X^d\bigr)
= \Bki^N([\phi])\cup X^d.
\nal
Thus, $N_i(V)$ is the nilpotence index of $\Bki$, which the proof of
Theorem~\ref{thm-nilpotence-bound} shows is
\begin{align}\label{eqn-d0-nilpotence-index}
N_i(V) = e-\left\lfloor\frac{i-1}{f}\right\rfloor = 1+\left\lfloor\frac{ef-i}{f}\right\rfloor.
\end{align}

Applying \eqref{eqn-d0-nilpotence-index} with $ef-i$ in place of
$i$ similarly yields $N_{ef-i}(V) = 1 + \lfloor i / f \rfloor$. Hence,
\begin{align}\label{eqn-d=0-Ri}
R_i(V) =  1+ \max\left\{
\left\lfloor\frac{ef-i}{f}\right\rfloor,\,
\left\lfloor\frac{i}{f}\right\rfloor
\right\} = 1+\left\lfloor
\frac{\max\{i,ef-i\}}{f}
\right\rfloor.
\end{align}
Corollary~\ref{cor-exact-image} therefore yields
\bal
\operatorname{im}(\mathcal S_V^i)
\cong
k[x,y^{\pm 1}] \big/ 
\big( x^{1+\left\lfloor
\frac{\max\{i,ef-i\}}{f}
\right\rfloor } \big) \qquad
\text{for all }1\leq i\leq ef-1.
\nal
Note that \eqref{eqn-d=0-Ri} implies 
\bal
1+\left\lfloor\frac{e}{2}\right\rfloor
\leq R_i(V)\leq e.
\nal
Thus, if $e=1$, then $R_i(V)=1$ and
$\operatorname{im}(\mathcal S_V^i)\cong k[y^{\pm1}]$ for every
$1\leq i\leq f-1$. If $e\geq2$, then $R_i(V)\geq 2$ and $\operatorname{im}(\mathcal S_V^i)$ is nonreduced for every
$1\leq i\leq ef-1$.
\end{example}

The next example realizes the upper value $R_1(V)=e+1$ for every
$e \geq 1$ and $f\geq2$. Consequently,
$\operatorname{im}(\mathcal S_V^1)$ is a nilpotent thickening of
order $e+1$.

\begin{example} Let $e \geq 1$ and $f \geq 2$. Let $F_0/\Q_p$ be the unramified extension of degree $f$ and set $F=F_0(\varpi)$, where $\varpi^e=p$. Thus,
$k_F=\F_{p^f}$. Let $V$ be a Serre weight for $\GL_2(F)$ with
\bal
(d_0,\dots,d_{f-1})=(0,\dots,0,p-1).
\nal
Thus, $d = (p-1)p^{f-1}$. We will show that $N_1(V)=e+1$.

Since Theorem~\ref{thm-nilpotence-bound} proves $N_1(V)\leq e+1$,
it is enough to construct a class $[z]\in H^1(\OF,V)$ such that
$\BBB_{V,1}^{\,e}([z])\neq0$. Since $\BV(X^d)=X^d$, Lemmas~\ref{lem-BV-semilinear} and~\ref{lem-Bk-action} show that
\begin{align}\label{eqn-BV-kappa-action}
\BBB_{V,1}
\bigl([\kappa_{m,0}]\cup X^d\bigr)
=
[\kappa_{m-1,0}]\cup X^d
\qquad
\text{for }1\leq m\leq e-1.
\end{align}
Moreover, $d_0 = 0$ implies $[\kappa_{0,0}]\cup X^d \neq 0$ by Lemma \ref{lem-kappa0-y1-vanishing}, so it suffices to construct a class $[z]\in H^1(\OF,V)$ which satisfies 
\begin{align}\label{eqn-BV1-z-WTS}
\BBB_{V,1}([z])
= [\kappa_{e-1,0}]\cup X^d. 
\end{align}

For this, it is enough to construct a continuous $1$-cocycle $z_0:\mathcal O_{F_0}\rightarrow V$ such that
\begin{align}\label{eqn-z0-p-property}
z_0(p\beta) = \bar\beta X^d
\qquad
\text{for every }\beta\in\mathcal O_{F_0}.
\end{align}
Indeed, every $\alpha\in\OF$ can be uniquely expressed as $\alpha
=\sum_{m=0}^{e-1}a_m\varpi^m$ for some $a_m \in \mathcal{O}_{F_0}.$ Define $z(\alpha)\coloneqq z_0(a_0)$. The map $\alpha\mapsto a_0$ is additive, and $\alpha$ and $a_0$ have the same reduction modulo $\varpi$, so they act identically on $V$.
It follows that $z$ is a continuous $1$-cocycle on $\OF$. Moreover,
\begin{align}\label{eqn-varpi-alpha-expansion}
\varpi\alpha
=
pa_{e-1}
+
\sum_{m=1}^{e-1}a_{m-1}\varpi^m.
\end{align}
The cochain formula for $\BBB_{V,1}$ and
\eqref{eqn-z0-p-property} therefore yield
\bal
(\BBB_{V,1}z)(\alpha)=
\Prx\bigl(z(\varpi\alpha)\bigr) = 
\Prx\bigl(z_0(pa_{e-1})\bigr) = \bar a_{e-1}X^d.
\nal
By the definition of $\kappa_{e-1,0}$, this is
$(\kappa_{e-1,0}\cup X^d)(\alpha)$. Hence, \eqref{eqn-BV1-z-WTS} holds.

To construct $z_0$, define the carry polynomial
\bal 
C_p(A,B)
\coloneqq
\overline{\left(\frac{A^p+B^p-(A+B)^p}{p}\right)}
\in k[A,B],
\nal
where the bar denotes coefficient-wise reduction modulo $p$ of the entire
parenthesized expression. It satisfies the polynomial cocycle identity
\begin{align}\label{eqn-carry-polynomial-identity}
C_p(A,B)+C_p(A+B,C)-C_p(A,B+C)
= C_p(B,C)
\end{align}
in $k[A,B,C]$. For $a,b\in k_F$, let $\gamma(a,b)\in\mathcal O_{F_0}$ be defined by
\bal
[a]+[b]=[a+b]+p\gamma(a,b).
\nal
Since inverse Frobenius on $\F_{p^f}$ is the
$p^{f-1}$-power map, the length-two Witt addition formula gives
\bal
\overline{\gamma(a,b)}
= C_p(a,b)^{p^{f-1}}
= C_p\bigl(a^{p^{f-1}},b^{p^{f-1}}\bigr).
\nal
Thus, we obtain
\begin{align}\label{eqn-witt-carry-general-f}
\overline{\gamma(a,b)}
= C_p\bigl(\kappa_{f-1}(a),\kappa_{f-1}(b)\bigr).
\end{align}

Set $Z\coloneqq (Y / X)^{p^{f-1}}$ and define, for all $a \in k_F$,
\bal
h(a)\coloneqq  X^d C_p(Z,\kappa_{f-1}(a)).
\nal
The monomials occurring in $h(a)$ are of the form $X^{d-rp^{f-1}}Y^{rp^{f-1}}$ with $1 \leq r \leq p-1$, so $h(a)\in V$. Moreover, $[a]$ sends $Z$ to
$Z+\kappa_{f-1}(a)$. Therefore,
\eqref{eqn-carry-polynomial-identity} and
\eqref{eqn-witt-carry-general-f} yield 
\begin{align}\label{eqn-h-carry}
h(a)+[a]\cdot h(b)-h(a+b) = \overline{\gamma(a,b)}X^d.
\end{align}

Every element of $\mathcal O_{F_0}$ has a unique expression
$[a]+p\beta$, with $a\in k_F$ and
$\beta\in\mathcal O_{F_0}$. Define
\bal
z_0([a]+p\beta) \coloneqq  h(a)+\bar\beta X^d.
\nal
Equation~\eqref{eqn-h-carry} shows that $z_0$ is a $1$-cocycle. It
factors through $\mathcal O_{F_0}/p^2\mathcal O_{F_0}$ and is
therefore continuous. Since $h(0)=0$, it also satisfies
\eqref{eqn-z0-p-property}.

Combining \eqref{eqn-BV-kappa-action} and
\eqref{eqn-BV1-z-WTS}, we obtain
\bal
\BBB_{V,1}^{e}([z])
=
[\kappa_{0,0}]\cup X^d
\neq0.
\nal
Hence, $N_1(V) = e+1$. Consequently, $R_1(V)=R_{ef-1}(V)=e+1$, and
Corollary~\ref{cor-exact-image} yields
\bal
\operatorname{im}(\mathcal S_V^1)
\cong
\operatorname{im}(\mathcal S_V^{ef-1})
\cong
k[x,y^{\pm1}]/(x^{e+1}).
\nal
\end{example}

We conclude this subsection with a family of Serre weights realizing the lower value
$R_1(V)=e-1$.

\begin{example}
Assume that $e\geq2$. Let $F_0/\Qp$ be the unramified extension of
degree $f$, and set $F=F_0(\varpi)$, where $\varpi^e=p$. Let $V$ be
a Serre weight for $\GL_2(F)$ such that
\bal
d_j>0\quad\text{for every }0\leq j\leq f-1 \qquad \text{and}
\qquad
\sum_{j=0}^{f-1}d_j<p-1.
\nal

To show that the lower bound $N_1(V)\geq e-1$ provided by Theorem \ref{thm-nilpotence-bound} is sharp, it suffices
to prove that $\BBB_{V,1}^{e-1}=0$. Let
$z:\OF\rightarrow V$ be a continuous inhomogeneous $1$-cocycle and set
$w\coloneqq \BBB_{V,1}^{e-1}z$. Then
\bal
w(\varpi\beta)
=
\Prx\bigl(z(p\beta)\bigr)
=
\sum_{a\in\F_p}
\Prx\bigl((a\beta)\cdot z(\beta)\bigr)
\qquad
\text{for every }\beta\in\OF.
\nal
The weight $V$ is spanned by the monomials $X^{d-c}Y^c$, where $c=\sum_{j=0}^{f-1}c_j p^j$ with $0 \leq c_j \leq d_j$ for each $j$. For each such monomial,
\bal
\Prx\bigl((a\beta)\cdot X^{d-c}Y^c\bigr)
\nal
is a scalar multiple of
$a^{\sum_j c_j}X^d$. Hence, for every $v\in V$, the function
$a\mapsto\Prx((a\beta)\cdot v)$ is polynomial in $a$ of degree at
most $\sum_j d_j<p-1$. Since the sum $\sum_{a \in \Fp} a^r$ vanishes for all $0 \leq r < p-1$, the cocycle $w$ vanishes on $\varpi\OF$. Moreover, $w$ takes values in $k X^d$, on which $\OF$ acts trivially. Thus, its cohomology class has the form
\bal
\sum_{j=0}^{f-1}
\lambda_j\bigl([\kappa_{0,j}]\cup X^d\bigr)
\qquad
\text{for some }\lambda_0,\dots,\lambda_{f-1}\in k.
\nal
Since $d_j>0$ for every $j$, all these classes vanish by
Lemma~\ref{lem-kappa0-y1-vanishing}. Hence
$\BBB_{V,1}^{e-1}=0$, and therefore $N_1(V)=e-1$. Since $n_0(V)=0$, Theorem~\ref{thm-nilpotence-bound} also gives $N_{ef-1}(V)=1$. Thus,
\bal
R_1(V)=R_{ef-1}(V)=e-1,
\nal
and so Corollary~\ref{cor-exact-image} yields
\bal
\operatorname{im}(\mathcal S_V^1)
\cong
\operatorname{im}(\mathcal S_V^{ef-1})
\cong
k[x,y^{\pm1}]/(x^{e-1}).
\nal
\end{example}

\subsection{Nonnilpotence in higher rank}
\label{subsec-GLn-nonnilpotence}

We conclude by contrasting the preceding description for $\GL_2(F)$ with the situation in higher rank. A key feature of the $\GL_2(F)$ analysis is the nilpotence of the two support components of the single noncentral Hecke generator in interior cohomological degrees. This nilpotence phenomenon does not persist in higher rank. Already for $\GL_3(F)$, an individual support component can be nonnilpotent.

\begin{example}[Nonnilpotence for $\GL_3(F)$]
\label{ex-GL3-nonnilpotence}
For a finite extension $F / \Qp$, let $G = \GL_3(F)$. Let $V=k$ be the
trivial weight and take $i = 1$. For $x,y,z \in F$, define
\bal
u(x,y,z)\coloneqq 
\begin{pmatrix}
1 & x & y\\
0 & 1 & z\\
0 & 0 & 1
\end{pmatrix}
\in U.
\nal
By Theorem \ref{thm-GLn-SVi-support-formula}, we have
\begin{align}\label{eqn-ex-supp}
\Supp \mathcal{S}_k^1(T_1)
\subseteq
K_T t_{\{1\}}
\sqcup
K_T t_{\{2\}}
\sqcup
K_T t_{\{3\}}.
\end{align}
We consider the component indexed by $I \coloneqq  \{3\}$. In this case,
$t_I = \diag(1,1,\varpi)$ and
\begin{align}\label{eqn-ex-tI-conj}
t_I u(x,y,z)t_I^{-1}
=
u(x,\varpi^{-1}y,\varpi^{-1}z).
\end{align}
Thus, $K_U \subset t_I K_U t_I^{-1}$, and hence
$K_{U,I}=K_U$. Since $V=k$ is trivial, the coefficient map
$T_1(t_I)$ is the identity. Therefore, Theorem
\ref{thm-GLn-SVi-support-formula} gives
\bal
\mathcal{S}_k^1(T_1)(t_I)
=
\bigl(\conj_{t_I^{-1}}\bigr)^*
\in
\End_k\bigl(H^1(K_U,k)\bigr).
\nal

Choose a nonzero $\F_p$-linear map
$\lambda:k_F\longrightarrow k$, and let $\overline{x}\in k_F$ denote
the reduction of $x\in\OF$ modulo $\varpi$. Define
$\xi:K_U\to k$ by
\bal
\xi\bigl(u(x,y,z)\bigr)\coloneqq \lambda(\overline{x});
\nal
one can easily see that this is a continuous group homomorphism, and hence gives a nonzero class $[\xi]\in H^1(K_U,k)$ as the action of $K_U$ on $k$ is trivial.

For every $x,y,z\in\OF$, we have
\begin{align*}
\bigl(\conj_{t_I^{-1}}\bigr)^*(\xi)
\bigl(u(x,y,z)\bigr)
&=
\xi\bigl(t_I^{-1}u(x,y,z)t_I\bigr) \\
&=
\xi\bigl(u(x,\varpi y,\varpi z)\bigr) \\
&=
\lambda(\overline{x}) \\
&=
\xi\bigl(u(x,y,z)\bigr).
\end{align*}
It follows that
$\bigl(\conj_{t_I^{-1}}\bigr)^*([\xi])=[\xi]$. Therefore,
\bal
\left(\mathcal{S}_k^1(T_1)(t_I)\right)^m([\xi])
=
[\xi]\neq 0
\qquad\text{for every }m\geq 1.
\nal
Hence, $\mathcal{S}_k^1(T_1)(t_I)$ is not nilpotent.

In fact, $\mathcal{S}_k^1(T_1)$ itself is not nilpotent. Fix
$m\geq1$. By iterating the convolution formula and using
\eqref{eqn-ex-supp}, every potentially nonzero summand in
\bal
\mathcal{S}_k^1(T_1)^{\star m}(t_{\{3\}}^m)
\nal
is indexed by a word of length $m$ in the three support cosets
\bal
K_T t_{\{1\}},\qquad
K_T t_{\{2\}},\qquad
K_T t_{\{3\}},
\nal
whose product is $K_Tt_{\{3\}}^m$. Comparing the three diagonal
valuations shows that every letter in such a word must be
$K_Tt_{\{3\}}$. Thus, there is only one potentially nonzero summand,
and
\bal
\mathcal{S}_k^1(T_1)^{\star m}(t_{\{3\}}^m)
=
\left(
\mathcal{S}_k^1(T_1)(t_{\{3\}})
\right)^m.
\nal
The right-hand side fixes $[\xi]\neq0$, and is therefore nonzero.
Hence
\bal
\mathcal{S}_k^1(T_1)^{\star m}\neq0
\nal
for every $m\geq1$, so $\mathcal{S}_k^1(T_1)$ is not nilpotent.
\end{example}

The same mechanism occurs for every $n>2$. Indeed, take
$I=\{n\}$. Conjugation by $t_I^{-1}$ fixes the root subgroups
$U_{a,b}(\OF)$ with $b<n$, while it sends $U_{a,n}(\OF)$ into
$U_{a,n}(\varpi\OF)$. Thus, the corresponding component of
$\mathcal{S}_k^1(T_1)$ fixes the upper-left copy of
$U_{n-1}(\OF)$ and, in particular, fixes nonzero degree-one
cohomology classes arising from the simple-root coordinates of this
subgroup. This upper-left subgroup is the upper-unipotent subgroup
of the standard Levi $\GL_{n-1}\times\GL_1$.

Thus, for $n>2$, even the pure powers of an individual support
component need not vanish. Moreover, convolution powers may involve
compositions of components supported on different cosets. A
description of the image for $\GL_n(F)$ therefore requires control
both of the individual support components and of their interactions
under convolution.

By contrast, for $n=2$, the only proper standard Levi is the
diagonal torus, whose unipotent subgroup is trivial. The mechanism
of nonnilpotence exhibited above is therefore absent. In addition,
the single noncentral standard Hecke generator has only two support
components, and we showed that these components annihilate one
another. Their nilpotence orders, together with the adjointness
between complementary degrees, are consequently sufficient to
determine the image of $\mathcal{S}_V^i$ for $\GL_2(F)$. The mechanism underlying the exact image theorem for $\GL_2(F)$ is therefore genuinely rank two.

\appendix 

\section{Standard facts about cohomological delta-functors}

\begin{lem}
\label{delta1}
Let $\mc A, \mc A'$ be abelian categories, and let $H=(H^i)_{i\geq 0}$ be any $\delta$-functor $\mc A\to \mc A'$.  Suppose that $F,G:\mc A\to \mc A$ are any exact additive functors equipped with a natural transformation $\eta:F\to G$.  Then $H\circ F\coloneqq (H^i\circ F)_{i\geq 0}$ and $H\circ G\coloneqq (H^i\circ G)_{i\geq 0}$ are $\delta$-functors $\mc A\to \mc A'$, and we have a natural morphism of $\delta$-functors induced by $\eta$:
$$H(\eta):H\circ F\to H\circ G,\qquad (H(\eta))^i\coloneqq H^i(\eta):H^i\circ F\to H^i\circ G.$$
\end{lem}

\begin{proof}
That $H\circ F$ and $H\circ G$ are $\delta$-functors is easy: since $F$ is exact, $H\circ F$ indeed functorially turns short exact sequences into long exact sequences, as required, and similarly for $H\circ G$.  To see that $\eta$ induces a morphism of $\delta$-functors, let
$$0\to A\to B\to C\to 0$$
be a short exact sequence in $\mc A$.  We then get a commuting diagram in $\mc A$ with exact rows:
\[\begin{tikzcd}
	0 & FA & FB & FC & 0 \\
	0 & GA & GB & GC & 0.
	\arrow[from=1-1, to=1-2]
	\arrow[from=1-2, to=1-3]
	\arrow["{\eta_A}", from=1-2, to=2-2]
	\arrow[from=1-3, to=1-4]
	\arrow["{\eta_B}", from=1-3, to=2-3]
	\arrow[from=1-4, to=1-5]
	\arrow["{\eta_C}", from=1-4, to=2-4]
	\arrow[from=2-1, to=2-2]
	\arrow[from=2-2, to=2-3]
	\arrow[from=2-3, to=2-4]
	\arrow[from=2-4, to=2-5]
\end{tikzcd}\]
Since $H$ is a $\delta$-functor, we conclude that for any $i\geq 0$, there is a commuting square
\[\begin{tikzcd}
	{H^i(FC)} & {H^{i+1}(FA)} \\
	{H^i(GC)} & {H^{i+1}(GA)}.
	\arrow["\delta", from=1-1, to=1-2]
	\arrow["{H^i(\eta_C)}"', from=1-1, to=2-1]
	\arrow["{H^{i+1}(\eta_A)}", from=1-2, to=2-2]
	\arrow["\delta", from=2-1, to=2-2]
\end{tikzcd}\]
Since the top and bottom $\delta$-homomorphisms are exactly the $\delta$-homomorphisms coming from the $\delta$-functors $H\circ F$ and $H\circ G$, respectively, we conclude that $\eta$ indeed gives us a morphism of $\delta$-functors, as claimed.
\end{proof}

\begin{lem}
\label{delta2}
Let $\mc A,\mc A',\mc A''$ be abelian categories.  Suppose that $H,\tilde H$ are $\delta$-functors $\mc A'\to \mc A''$ equipped with a $\delta$-functor morphism $\Phi:H\to \tilde H$.  Let $F:\mc A\to \mc A'$ be any exact functor.  Then
$$\Phi_F:H\circ F\to \tilde H\circ F$$
given by $\Phi_{F,i}=\Phi_{i,{F(-)}}:H^i\circ F\to \tilde H^i\circ F$ is a $\delta$-functor morphism.
\end{lem}

\begin{proof}
The proof of this is the same idea as that of the previous lemma: given a short exact sequence 
$$0\to A\to B\to C\to 0$$
in $\mc A$, we also get a short exact sequence
$$0\to FA\to FB\to FC\to 0$$
in $\mc A'$.  One then needs to verify that the following diagram commutes:
\[\begin{tikzcd}
	{H^i(FC)} & {H^{i+1}(FA)} \\
	{\tilde H^i(FC)} & {\tilde H^{i+1}(FA)}
	\arrow["\delta", from=1-1, to=1-2]
	\arrow["{\Phi_{i,FC}}"', from=1-1, to=2-1]
	\arrow["{\Phi_{i+1,FA}}", from=1-2, to=2-2]
	\arrow["\delta", from=2-1, to=2-2]
\end{tikzcd}\]
Here, the top $\delta$ is the $\delta$-homomorphism obtained by applying $H\circ F$ to our first short exact sequence, while the bottom $\delta$ is obtained by applying $\tilde H\circ F$. The proof of Lemma \ref{delta1} shows that these are the same as the $\delta$-homomorphisms obtained by applying $H$ and $\tilde H$ to the second short exact sequence $0\to FA\to FB\to FC\to 0$. Since $\Phi$ is a morphism of $\delta$-functors, the diagram therefore commutes.
\end{proof}

\begin{lem}
\label{delta3}
Let $\mc A,\mc A_1, \mc A_2, \mc A'$ all be abelian categories, and suppose that we have additive functors
$$G_i:\mc A\to \mc A_i,\qquad F_i:\mc A_i\to \mc A'$$
for $i=1,2$ such that the $F_i$ are left-exact, the $G_i$ are exact, and the $G_i$ send injective objects to $F_i$-acyclic objects, respectively.  Assume also that there is a natural transformation $\eta:F_1\circ G_1\to F_2\circ G_2$ of functors $\mc A\to \mc A'$.  Then we have natural identifications of $\delta$-functors $(\mathrm{R}F_i)\circ G_i=\mathrm{R}(F_i\circ G_i)$, and $\eta$ induces a morphism of $\delta$-functors
$$\mathrm{R}\eta:\mathrm{R}(F_1)\circ G_1\to \mathrm{R}(F_2)\circ G_2$$
given by composing the aforementioned isomorphisms with the natural maps $\mathrm{R}^i\eta:\mathrm{R}^i(F_1\circ G_1)\to \mathrm{R}^i(F_2\circ G_2)$.
\end{lem}

\begin{proof}
The claim that there is a natural identification of $\delta$-functors $(\mathrm{R}F_i)\circ G_i=\mathrm{R}(F_i\circ G_i)$ is an easy application of the Grothendieck spectral sequence: the $G_i$ are exact so $\mathrm{R}^j(G_i)=0$ for $j>0$, and they send injective objects to $F_i$-acyclic objects, respectively.  The remaining claim now follows from the following, where we have merely replaced $F_i$ by the (still left-exact) functor $F_i\circ G_i$: given a natural transformation $\eta:F_1\to F_2$ of functors $\mc A\to \mc A'$, the natural maps $\mathrm{R}^i\eta:\mathrm{R}^iF_1\to \mathrm{R}^iF_2$ form a morphism of $\delta$-functors $\mathrm{R}F_1\to \mathrm{R}F_2$.

To verify this claim, let $0\to A\to B\to C\to 0$ be a short exact sequence in $\mc A$.  Pick injective resolutions $A\to I_A^\bullet$ and $C\to I_C^\bullet$; by the horseshoe lemma, we can get an injective resolution $B\to I_B^\bullet$ such that there is a short exact sequence of complexes
$$0\to I_A^\bullet\to I_B^\bullet\to I_C^\bullet\to 0.$$
Applying $F_1$ and $F_2$ to this short exact sequence gives us two short exact sequences of complexes
\[\begin{tikzcd}
	0 & {F_1(I_A^\bullet)} & {F_1(I_B^\bullet)} & {F_1(I_C^\bullet)} & 0 \\
	0 & {F_2(I_A^\bullet)} & {F_2(I_B^\bullet)} & {F_2(I_C^\bullet)} & 0
	\arrow[from=1-1, to=1-2]
	\arrow[from=1-2, to=1-3]
	\arrow["{\eta_{I_A^\bullet}}", from=1-2, to=2-2]
	\arrow[from=1-3, to=1-4]
	\arrow["{\eta_{I_B^\bullet}}", from=1-3, to=2-3]
	\arrow[from=1-4, to=1-5]
	\arrow["{\eta_{I_C^\bullet}}", from=1-4, to=2-4]
	\arrow[from=2-1, to=2-2]
	\arrow[from=2-2, to=2-3]
	\arrow[from=2-3, to=2-4]
	\arrow[from=2-4, to=2-5]
\end{tikzcd}\]
with maps in every entry from the top to the bottom induced by the natural transformation $\eta$. The $\delta$-homomorphisms in the induced long exact sequences with respect to $F_1$ and $F_2$ are exactly those obtained by applying the snake lemma to the appropriate diagrams arising from the above short exact sequences of complexes. Moreover, the chain maps induced by $\eta$ yield, after taking cohomology, the maps $\mathrm{R}^i\eta: \mathrm{R}^iF_1\to \mathrm{R}^i F_2$.  The fact that these commute with the maps the snake lemma provides follows directly from the naturality of the snake lemma; see for example \cite[Proposition 1.3.4]{Weibel}.
\end{proof}

\begin{lem}
\label{delta4}
Let $\mc A$, $\mc A'$, $\mc A''$ be abelian categories.  Let $H=(H^i)_{i\geq 0}$ be any $\delta$-functor $\mc A\to \mc A'$, and $F,G:\mc A'\to \mc A''$ be any exact additive functors equipped with a natural transformation $\eta:F\to G$.  Then $F\circ H\coloneqq (F\circ H^i)_{i\geq 0}$ and $G\circ H\coloneqq (G\circ H^i)_{i\geq 0}$ are $\delta$-functors $\mc A\to \mc A''$, and we have a natural morphism of $\delta$-functors induced by $\eta$:
$$\eta:F\circ H\to G\circ H,\qquad \eta^i\coloneqq \eta_{H^i(-)}:F(H^i(-))\to G(H^i(-)).$$
\end{lem}

\begin{proof}
The proof is essentially the same as that of Lemma \ref{delta1}.  The $\delta$-functor $H$ functorially maps short exact sequences to long exact sequences, and applying $F$ or $G$ preserves this exactness, so $F\circ H$ and $G\circ H$ are automatically $\delta$-functors.

Given a short exact sequence
$$0\to A\to B\to C\to 0$$
in $\mc A$, we get for any $i\geq 0$ a $\delta$-homomorphism $\delta:H^i(C)\to H^{i+1}(A)$ in $\mc A'$.  Since $\eta$ is a natural transformation, we get a commuting diagram
\[\begin{tikzcd}
	{F(H^i(C))} & {F(H^{i+1}(A))} \\
	{G(H^i(C))} & {G(H^{i+1}(A))}
	\arrow["{F\delta}", from=1-1, to=1-2]
	\arrow["{\eta_{H^i(C)}}"', from=1-1, to=2-1]
	\arrow["{\eta_{H^{i+1}(A)}}", from=1-2, to=2-2]
	\arrow["{G\delta}", from=2-1, to=2-2]
\end{tikzcd}\]
where the top and bottom maps are exactly the $\delta$-homomorphisms coming from the $\delta$-functors $F\circ H$ and $G\circ H$.  This proves the claim.
\end{proof}

\section{Compatibility results for the double coset formula}
The aim of this section is to prove the following lemma required in the proof of Proposition \ref{prop-explicit-formula}.
\begin{lem}
\label{lem-kappa&satake}
Let $\Phi\in \mathcal{H}_G(V) = \End_G(\ind^G_K V)$ and $\Phi_1: \ind^P_{K_P}\Res^K_{K_P}V\rightarrow \ind^P_{K_P}\Res^K_{K_P} V$ be the induced morphism from the Iwasawa decomposition $G=PK$. Then the diagram below commutes:
\[
\begin{tikzcd}
{\ind^M_{M^+}H^i(K_U,\ind^P_{K_P}V)} \arrow[dd, "\ind^M_{M^+}H^i(\Phi_1)"'] \arrow[rr,swap]{ur}{\cong}[swap]{\kappa} &  & {\ind^M_{K_M}H^i(K_U,V)} \arrow[dd, "\mathcal{S}^i_V(\Phi)"] \\
\\
{\ind^M_{M^+}H^i(K_U,\ind^P_{K_P}V)} \arrow[rr,swap]{ur}{\cong}[swap]{\kappa}                                        &  & {\ind^M_{K_M}H^i(K_U,V)} 
\end{tikzcd}
\]
\end{lem}
To prove this lemma, we need some preliminary results.

\vspace{0.2in}

We first record two compatibility results for mates that will be useful later. 
\begin{lem}[Compatibility of mates with units and counits]
\label{mate_unit_compatibility}
Consider the diagram of functors 
\begin{center}
\begin{tikzcd}
	{\mathcal{C}} & \perp & {\mathcal{D}} \\
	\\
	{\mathcal{C'}} & \perp & {\mathcal{D'}} 
	\arrow["{L}", curve={height=-12pt}, from=1-1, to=1-3]
	\arrow["F"', from=1-1, to=3-1]
	\arrow["{R}", curve={height=-12pt}, from=1-3, to=1-1]
	\arrow["G", from=1-3, to=3-3]
	\arrow["{L'}", curve={height=-12pt}, from=3-1, to=3-3]
	\arrow["{R'}", curve={height=-12pt}, from=3-3, to=3-1]
\end{tikzcd}
\end{center}
with the natural bijection of mates
\begin{equation*}
\mathrm{Nat}(L'F,GL)\overset{\cong}{\longleftrightarrow}\mathrm{Nat}(FR,R'G)
\end{equation*}
given by
\[
\alpha\longleftrightarrow\beta.
\]
Let $\eta,\epsilon$ (respectively, $\eta',\epsilon'$) denote the unit and counit of $L\dashv R$ (respectively, $L'\dashv R'$). Then the following diagrams of natural transformations commute:
\begin{center}
\[\begin{tikzcd}
	{L'FR} && GLR &&& F && {R'L'F} \\
	\\
	{L'R'G} && G &&& FRL && {R'GL}
	\arrow["{\alpha R}",Rightarrow, from=1-1, to=1-3]
	\arrow["{L'\beta}"', Rightarrow, from=1-1, to=3-1]
	\arrow["{G\epsilon}", Rightarrow, from=1-3, to=3-3]
	\arrow["{\eta'F}", Rightarrow, from=1-6, to=1-8]
	\arrow["{F\eta}"', Rightarrow, from=1-6, to=3-6]
	\arrow["{R'\alpha}", Rightarrow, from=1-8, to=3-8]
	\arrow["{\epsilon'G}"', Rightarrow, from=3-1, to=3-3]
	\arrow["{\beta L}"', Rightarrow, from=3-6, to=3-8]
\end{tikzcd}\]
\end{center}
\end{lem}
\begin{proof}
The proof follows from the explicit formula for the bijection of mates in \cite[Proposition 2.1.9]{heyer2023left}, along with naturality of the units and counits and the triangle identities.
\end{proof}
\begin{lem}[Compatibility of mates with horizontal composition]
\label{mates}
Consider a diagram of functors
\begin{center}
\begin{tikzcd}
	{\mathcal{C}} & \perp & {\mathcal{D}} & \perp & {\mathcal{E}} \\
	\\
	{\mathcal{C'}} & \perp & {\mathcal{D'}} & \perp & {\mathcal{E'}}
	\arrow["{L_1}", curve={height=-12pt}, from=1-1, to=1-3]
	\arrow["F"', from=1-1, to=3-1]
	\arrow["{R_1}", curve={height=-12pt}, from=1-3, to=1-1]
	\arrow["{L_2}", curve={height=-12pt}, from=1-3, to=1-5]
	\arrow["G", from=1-3, to=3-3]
	\arrow["{R_2}", curve={height=-12pt}, from=1-5, to=1-3]
	\arrow["H", from=1-5, to=3-5]
	\arrow["{L_1'}", curve={height=-12pt}, from=3-1, to=3-3]
	\arrow["{R_1'}", curve={height=-12pt}, from=3-3, to=3-1]
	\arrow["{L_2'}", curve={height=-12pt}, from=3-3, to=3-5]
	\arrow["{R_2'}", curve={height=-12pt}, from=3-5, to=3-3]
\end{tikzcd}
\end{center}
Given the natural bijection from the left square
\begin{equation*}
\mathrm{Nat}(L_1'F,GL_1)\overset{\cong}{\longleftrightarrow}\mathrm{Nat}(FR_1,R_1'G)
\end{equation*}
with $$\alpha_1 \longleftrightarrow \beta_1,$$ and the natural bijection from the right square
\begin{equation*}
\mathrm{Nat}(L_2'G,HL_2)\overset{\cong}{\longleftrightarrow}\mathrm{Nat}(GR_2,R_2'H)
\end{equation*}
with $$\alpha_2\longleftrightarrow \beta_2,$$ the natural bijection from the combined square 
\begin{equation*}
\mathrm{Nat}(L_2'L_1'F,HL_2L_1)\overset{\cong}{\longleftrightarrow}\mathrm{Nat}(FR_1R_2,R_1'R_2'H)
\end{equation*}
satisfies $$\alpha_2L_1\circ L_2'\alpha_1\longleftrightarrow R_1'\beta_2\circ \beta_1R_2$$
\end{lem}
For a proof, we refer to \cite[Proposition 2.2]{2-Cat}.

\vspace{0.2in}
\textit{Notation.} Throughout the commutative diagrams below, whenever a proposition, lemma, or corollary from \cite{heyer2023left} labels an arrow, that arrow denotes the corresponding natural transformation or morphism constructed there.
\\

We now promote the isomorphism in Proposition \ref{Exp1} to the derived level:
\begin{prop}
\label{Exp2}
There is a natural isomorphism of functors $D(K_P)\rightarrow D(M)$ $$\Gamma: \ind^M_{K_M}\mathrm{R}H^0(K_U,-)\Rightarrow \ind^M_{M^+}\mathrm{R}H^0(K_U,-)\Res^P_{P^+}\ind^P_{K_P}$$ which on $i$-th cohomology is given by $\kappa^{-1}$ as above.
\end{prop}
\begin{proof}
From \cite[equation (3.4.14)]{heyer2023left}, we have a natural transformation 
$$\gamma_1: \ind^{M^+}_{K_M}\mathrm{R}H^0(K_U,-)\xRightarrow[]{\cong} \mathrm{R}H^0(K_U,-)\ind^{P^+}_{K_P}$$
We also have the natural transformation
\begin{align*}
\gamma_2: \ind^M_{M^+}\mathrm{R}H^0(K_U,-)\ind^{P^+}_{K_P}&\xRightarrow[]{\eta} \ind^M_{M^+}\mathrm{R}H^0(K_U,-)\Res^P_{P^+}\ind^P_{P^+}\ind^{P^+}_{K_P}\\
&\qquad\cong \ind^M_{M^+}\mathrm{R}H^0(K_U,-)\Res^P_{P^+}\ind^{P}_{K_P}
\end{align*}
induced by the natural transformation $\eta:\Id_{P^+}\Rightarrow \Res^P_{P^+}\ind^P_{P^+}$. Now, we let $$\Gamma:=\gamma_2\circ \ind^M_{M^+}(\gamma_1).$$ From the proof of \cite[Lemma 3.4.13]{heyer2023left}, we see that $\gamma_1$ is a natural isomorphism which on $i$-th cohomology is given by $\alpha^{-1}$ from Lemma \ref{Exp1Step1}. 
Also, we see that, by definition of $\phi$ (from the proof of Lemma \ref{Exp1Step2}), $H^i(\gamma_2)=\phi$. 
Since $\phi$ was an isomorphism ($\phi=f^{-1}$), we have that $\gamma_2$ is also an isomorphism. Now the map $\kappa=\ind^M_{M^+}(\alpha)\circ f\Rightarrow$  $\kappa^{-1}=f^{-1}\circ \ind^M_{M^+}(\alpha^{-1})=\phi\circ \ind^M_{M^+}(\alpha^{-1})$. Thus, we see that $\Gamma$ is also a natural isomorphism which on $i$-th cohomology is given by $\kappa^{-1}$.
\end{proof}
\begin{remark}
\label{rem1}
By \cite[Lemma 3.4.13]{heyer2023left}, the right adjoint of the isomorphism $\gamma_1$ is precisely the isomorphism $\Res^{P^+}_{K_P}F^{P^+}_{M^+}\xRightarrow[]{\cong}F^{K_P}_{K_M}\Res^{M^+}_{K_M}$. Thus, by the definition of $\Gamma$, taking right adjoints gives the following commutative diagram (we denote the morphism induced by $\Gamma$ on right adjoints by $R(\Gamma)$):
\begin{center}
\begin{tikzcd}
	{\Res^P_{K_P}\RInd^P_{P^+}F^{P^+}_{M^+}\Res^M_{M^+}} && {F^{K_P}_{K_M}\Res^M_{K_M}} \\
	{\Res^{P^+}_{K_P}\Res^P_{P^+}\RInd^P_{P^+}F^{P^+}_{M^+}\Res^M_{M^+}} \\
	{\Res^{P^+}_{K_P}F^{P^+}_{M^+}\Res^M_{M^+}} && {F^{K_P}_{K_M}\Res^{M^+}_{K_M}\Res^M_{M^+}}
	\arrow[swap, Rightarrow, from=1-1, to=1-3]{ur}{\cong}[swap]{R(\Gamma)}
	\arrow["\cong", Rightarrow, from=1-1, to=2-1]
	\arrow["\epsilon", Rightarrow, from=2-1, to=3-1]
	\arrow[swap, Rightarrow, from=3-1, to=3-3]{ur}{ \mathrm{Lemma}\; 3.4.13}[swap]{\cong}
	\arrow["\cong"', Rightarrow, from=3-3, to=1-3]
\end{tikzcd}
\end{center}
\end{remark}
We can now slightly modify $\Gamma$ by precomposing it with $\omega_{K_P}\otimes_{k}-:D(K_P)\rightarrow D(K_P)$ to get a natural isomorphism
$$\Gamma\circ (\omega_{K_P}\otimes_k-):\ind^M_{K_M}\mathrm{R}H^0(K_U,\omega_{K_P}\otimes_{k}-)\xRightarrow[]{\cong} \ind^M_{M^+}\mathrm{R}H^0(K_U,-)\Res^P_{P^+}\ind^P_{K_P}(\omega_{K_P}\otimes_{k}-).$$
From the fact that $\omega_{K_P}=\Res^{P^+}_{K_P}\omega_{P^+}$ (by \cite[Lemma 3.4.13]{heyer2023left}) and $\omega_{P^+}=\Res^P_{P^+}\omega_P$ (by \cite[Remark 3.4.21]{heyer2023left}), we have natural isomorphisms $\Res^{P^+}_{K_P}\circ (\omega_{P^+}^{-1}\otimes_k -)\cong (\omega_{K_P}^{-1}\otimes_k -)\circ \Res^{P^+}_{K_P}$ and $\Res^{P}_{P^+}\circ (\omega_{P}^{-1}\otimes_k -)\cong (\omega_{P^+}^{-1}\otimes_k -)\circ \Res^{P}_{P^+}$. Upon taking left adjoints, we get natural isomorphisms $(\omega_{P^+}\otimes_k-) \circ \ind^{P^+}_{K_P}\cong \ind^{P^+}_{K_P}\circ (\omega_{K_P}\otimes_k -)$ and $(\omega_{P}\otimes_k-) \circ \ind^{P}_{P^+}\cong \ind^{P}_{P^+}\circ (\omega_{P^+}\otimes_k -)$. Thus, we get an isomorphism of functors $D(K_P)\rightarrow D(M)$:
$$\tilde{\Gamma}:\ind^M_{K_M}\mathrm{R}H^0(K_U,\omega_{K_P}\otimes_{k}-)\xRightarrow[]{\cong} \ind^M_{M^+}\mathrm{R}H^0(K_U,\omega_{P^+}\otimes_{k}-)\Res^P_{P^+}\ind^P_{K_P}$$ given by (as in the proof of the above proposition) $\tilde{\Gamma}=\tilde{\gamma_2}\circ \ind^M_{M^+}(\tilde{\gamma_1})$ where 
$$\tilde{\gamma_1}: \ind^{M^+}_{K_M}\mathrm{R}H^0(K_U, \omega_{K_P}\otimes_k -)\xRightarrow[]{\cong}\mathrm{R}H^0(K_U, \omega_{P^+}\otimes_k -)\ind^{P^+}_{K_P}$$ and 
$$\tilde{\gamma_2}: \ind^M_{M^+}\mathrm{R}H^0(K_U, \omega_{P^+}\otimes_k -)\ind^{P^+}_{K_P}\xRightarrow[]{\cong}\ind^M_{M^+}\mathrm{R}H^0(K_U, \omega_{P^+}\otimes_k -)\Res^P_{P^+}\ind^P_{K_P}.$$

\begin{prop}
\label{diag}
The following diagram of functors $D(K_P)\rightarrow D(M)$ commutes up to natural isomorphism: 
\begin{center}
\begin{tikzcd}
\ind^M_{M^+}L_{K_U}\Res^P_{P^+}\ind^P_{K_P} \arrow[d,swap, Rightarrow]{ur}{\cong}[swap]{\mathrm{Proposition}\; 3.4.6} &  &  & {\ind^M_{M^+}\mathrm{R}H^0(K_U, \omega_{P^+}\otimes_k -)\Res^P_{P^+}\ind^P_{K_P}} \arrow[lll, swap, Rightarrow]{ur}{\mathrm{Corollary}\;   3.4.20}[swap]{\cong}                      \\
L_U\ind^P_{K_P} \arrow[d, swap,Rightarrow]{ur}{\cong}[swap]{ \mathrm{Proposition}\; 3.2.11}                            &  &  &                                                                                                                                                             \\
\ind^M_{K_M}L_{K_U}                                                                                   &  &  & {\ind^M_{K_M}\mathrm{R}H^0(K_U, \omega_{K_P}\otimes_k -)} \arrow[uu, swap,Rightarrow]{ur}{\tilde{\Gamma}}[swap]{\cong} \arrow[lll,swap,Rightarrow]{ur}{\mathrm{Corollary}\; 3.1.8}[swap]{\cong}
\end{tikzcd}
\end{center}
\end{prop}

To prove the above proposition, we first prove some intermediate lemmas. Consider the morphism 
$$\beta_1:(\omega_{K_P}\otimes_k-)\Res^{P^+}_{K_P}\Inf^{P^+}_{M^+}\xRightarrow[]{\cong}\Res^{P^+}_{K_P}(\omega_{P^+}\otimes_k -)\Inf^{P^+}_{M^+}\xRightarrow[]{\cong}\Res^{P^+}_{K_P}F^{P^+}_{M^+}$$ where the latter isomorphism comes from \cite[Proposition 3.4.18]{heyer2023left} and its left mate (with $F=\omega_{K_P}\otimes_k-$ and $G=\Id$) 
$$L(\beta_1):\mathrm{R}H^0(K_U,-)\ind^{P^+}_{K_P}(\omega_{K_P}\otimes_k -)\Rightarrow L_{K_U}\ind^{P^+}_{K_P}.$$
\begin{lem}
\label{diag0}
The following diagram of functors $D(K_P)\rightarrow D(M^+)$ commutes up to natural isomorphism:
\begin{center}
\begin{tikzcd}
{\mathrm{R}H^0(K_U,-)\ind^{P^+}_{K_P}(\omega_{K_P}\otimes_k -)} \arrow[rr, "\cong", Rightarrow] \arrow[rrdd, "L(\beta_1)"', Rightarrow] &  & {\mathrm{R}H^0(K_U,-)(\omega_{P^+}\otimes_k-)\ind^{P^+}_{K_P}} \arrow[dd, swap, Rightarrow]{ur}{\cong}[swap]{\mathrm{Corollary}\;   3.4.20} \\
                                                                                                                               &  &                                                                                                               \\
                                                                                                                               &  & L_{K_U}\ind^{P^+}_{K_P}                                                                                      
\end{tikzcd}     
\end{center}
\end{lem}
\begin{proof}
By \cite[Corollary 3.4.20]{heyer2023left}, the left mate of the natural isomorphism 
$$A: (\omega_{P^+}\otimes_k -) \Inf^{P^+}_{M^+}\xRightarrow[]{\cong} F^{P^+}_{M^+}$$ taking $F=\omega_{P^+}\otimes_k -, G= \Id$ is 
$$L(A):\mathrm{R}H^0(K_U,-)(\omega_{P^+}\otimes_k-)\xRightarrow[]{\cong}L_{K_U}.$$
Now, the left mate of 
$$A_1=\Res^{P^+}_{K_P}A: \Res^{P^+}_{K_P}(\omega_{P^+}\otimes_k -) \Inf^{P^+}_{M^+}\xRightarrow[]{\cong} \Res^{P^+}_{K_P}F^{P^+}_{M^+}$$ 
taking $F= \Res^{P^+}_{K_P}(\omega_{P^+}\otimes_k -), G=\Id$ is
$$L(A_1):\mathrm{R}H^0(K_U,-)\ind^{P^+}_{K_P}\Res^{P^+}_{K_P}(\omega_{P^+}\otimes_k-)\Rightarrow L_{K_U}.$$ 
The left mates $L(A)$ and $L(A_1)$ are then related by the commutativity of the following diagram:
\begin{center}
\begin{tikzcd}
{\mathrm{R}H^0(K_U,-)\ind^{P^+}_{K_P}\Res^{P^+}_{K_P}(\omega_{P^+}\otimes_k-)} \arrow[rr, "\epsilon", Rightarrow] \arrow[rrd, "L(A_1)"', Rightarrow] &  & {\mathrm{R}H^0(K_U,-)(\omega_{P^+}\otimes_k-)} \arrow[d, "L(A)", Rightarrow] \\
                                                                                                                                            &  & L_{K_U}                                                            
\end{tikzcd}
\end{center}
This follows from Lemma \ref{mates} by taking $F=\Res^{P^+}_{K_P}(\omega_{P^+}\otimes_k-)$, $G= \omega_{P^+}\otimes_k-$, $H=\Id$, $L_1=R_1=\mathrm{Id}$, $L_1'=\ind^{P^+}_{K_P}$, $R_1'=\Res^{P^+}_{K_P}$, $L_2=L_{K_U}$, $R_2=\Inf^{P^+}_{M^+}$, $L_2'=\mathrm{R}H^0(K_U,-)$, $R_2'=F^{P^+}_{M^+}$ and appropriate choices of the relevant categories.

Now, the left mate of 
$$A_2=\beta_1: (\omega_{K_P}\otimes_k-)\Res^{P^+}_{K_P}\Inf^{P^+}_{M^+}\xRightarrow[]{\cong}\Res^{P^+}_{K_P}F^{P^+}_{M^+}$$ 
taking $F=(\omega_{K_P}\otimes_k-)\Res^{P^+}_{K_P}$ and $G=\mathrm{Id}$ is 
$$L(A_2): \mathrm{R}H^0(K_U,-)\ind^{P^+}_{K_P}(\omega_{K_P}\otimes_k -)\Res^{P^+}_{K_P}\Rightarrow L_{K_U}.$$
The left mates $L(A_2)$ and $L(\beta_1)$ are related by the commutativity of the following diagram:
\begin{center}
\begin{tikzcd}
{\mathrm{R}H^0(K_U,-)\ind^{P^+}_{K_P}(\omega_{K_P}\otimes_k-)} \arrow[rr, "\eta", Rightarrow] \arrow[rrd, "L(\beta_1)"', Rightarrow] &  & {\mathrm{R}H^0(K_U,-)\ind^{P^+}_{K_P}(\omega_{K_P}\otimes_k-)\Res^{P^+}_{K_P}\ind^{P^+}_{K_P}} \arrow[d, "L(A_2)\ind^{P^+}_{K_P}", Rightarrow] \\
                                                                                                                            &  & L_{K_U}\ind^{P^+}_{K_P}                                                                                                              
\end{tikzcd}
\end{center}
This follows from Lemma \ref{mates} by taking $F=\omega_{K_P}\otimes_k-$, $G=(\omega_{K_P}\otimes_k-)\Res^{P^+}_{K_P}$, $H=\mathrm{Id}$, $L_1=\ind^{P^+}_{K_P}$, $R_1=\Res^{P^+}_{K_P}$, $L_1'= R_1'=\mathrm{Id}$, $L_2=L_{K_U}$, $R_2=\Inf^{P^+}_{M^+}$, $L_2'=\mathrm{R}H^0(K_U,-)\ind^{P^+}_{K_P}$, $R_2'=\Res^{P^+}_{M^+}F^{P^+}_{M^+}$ and appropriate choices of the relevant categories.
From the commutative diagram
\begin{center}
\begin{tikzcd}
\Res^{P^+}_{K_P}(\omega_{P^+}\otimes_k-)\Inf^{P^+}_{M^+} \arrow[rr, "A_1", Rightarrow]                                     &  & \Res^{P^+}_{K_P}F^{P^+}_{M^+} \\
                                                                                                                           &  &                               \\
(\omega_{K_P}\otimes_k -)\Res^{P^+}_{K_P}\Inf^{P^+}_{M^+} \arrow[uu, "\cong", Rightarrow] \arrow[rruu, "A_2"', Rightarrow] &  &                              
\end{tikzcd}
\end{center}
we find that the corresponding diagram of left mates commutes:
\begin{center}
\begin{tikzcd}
{\mathrm{R}H^0(K_U,-)\ind^{P^+}_{K_P}\Res^{P^+}_{K_P}(\omega_{P^+}\otimes_k-)} \arrow[rr, "L(A_1)", Rightarrow]                                    &  & L_{K_U} \\
                                                                                                                                          &  &         \\
{\mathrm{R}H^0(K_U,-)\ind^{P^+}_{K_P}(\omega_{K_P}\otimes_k-)\Res^{P^+}_{K_P}} \arrow[uu, "\cong", Rightarrow] \arrow[rruu, "L(A_2)"', Rightarrow] &  &        
\end{tikzcd}
\end{center}
Thus, we see that $L(\beta_1)$ is given by the following (outer) commutative diagram:
\begin{center}
\begin{tikzcd}
{\mathrm{R}H^0(K_U,-)\ind^{P^+}_{K_P}(\omega_{K_P}\otimes_k-)} \arrow[rr, "\eta", Rightarrow] \arrow[rrddd, "\cong", Rightarrow] \arrow[rrddddd, "L(\beta_1)"', Rightarrow] &  & {\mathrm{R}H^0(K_U,-)\ind^{P^+}_{K_P}(\omega_{K_P}\otimes_k-)\Res^{P^+}_{K_P}\ind^{P^+}_{K_P}} \arrow[d, "\cong", Rightarrow]                        \\
                                                                                                                                                                   &  & {\mathrm{R}H^0(K_U,-)\ind^{P^+}_{K_P}\Res^{P^+}_{K_P}(\omega_{P^+}\otimes_k-)\ind^{P^+}_{K_P}} \arrow[dd, "\epsilon\; \ind^{P^+}_{K_P}", Rightarrow] \\
                                                                                                                                                                   &  &                                                                                                                                             \\
                                                                                                                                                                   &  & {\mathrm{R}H^0(K_U,-)(\omega_{P^+}\otimes_k-)\ind^{P^+}_{K_P}} \arrow[dd, "L(A)\ind^{P^+}_{K_P}", Rightarrow]                                        \\
                                                                                                                                                                   &  &                                                                                                                                             \\
                                                                                                                                                                   &  & L_{K_U}\ind^{P^+}_{K_P}                                                                                                                    
\end{tikzcd}
\end{center}
Therefore, it suffices to show that the inner triangle in the above diagram commutes. But that follows by applying $\mathrm{R}H^0(K_U,-)$ to the following commutative diagram (the left square commutes by naturality of the vertical isomorphism whereas the right square commutes by using counit compatibility in Lemma \ref{mate_unit_compatibility} with $L'=L=\ind^{P^+}_{K_P}, R'=R= \Res^{P^+}_{K_P}, F= \omega_{K_P}\otimes_k -, G = \omega_{P^+}\otimes_k -$ and appropriate choices of the relevant categories):
\begin{center}
\begin{tikzcd}
(\omega_{P^+}\otimes_k-)\ind^{P^+}_{K_P} \arrow[r, "\eta", Rightarrow]                                  & (\omega_{P^+}\otimes_k-)\ind^{P^+}_{K_P}\Res^{P^+}_{K_P}\ind^{P^+}_{K_P} \arrow[r, "\epsilon", Rightarrow]                               & (\omega_{P^+}\otimes_k-)\ind^{P^+}_{K_P}                                                                    \\
\ind^{P^+}_{K_P}(\omega_{K_P}\otimes_k-) \arrow[u, "\cong", Rightarrow] \arrow[r, "\eta", Rightarrow]   & \ind^{P^+}_{K_P}(\omega_{K_P}\otimes_k-)\Res^{P^+}_{K_P}\ind^{P^+}_{K_P} \arrow[u, "\cong", Rightarrow] \arrow[r, "\cong", Rightarrow]   & \ind^{P^+}_{K_P}\Res^{P^+}_{K_P}(\omega_{P^+}\otimes_k-)\ind^{P^+}_{K_P} \arrow[u, "\epsilon"', Rightarrow]
\end{tikzcd}
\end{center}
\end{proof}

\begin{lem}
\label{diag1}
The following diagram of functors $D(K_P)\rightarrow D(M^+)$ commutes:
\begin{center}
\begin{tikzcd}
L_{K_U}\ind^{P^+}_{K_P}  &  &  & \mathrm{R}H^0(K_U, \omega_{P^+}\otimes_k -)\ind^{P^+}_{K_P} \arrow[lll, swap, Rightarrow]{ur}{\mathrm{Corollary}\;   3.4.20}[swap]{\cong}  
                \\                   
                          &  &  &                                                                                                                                                             \\
\ind^{M^+}_{K_M}L_{K_U} \arrow[uu,swap, Rightarrow]{ur}{\mathrm{Proposition}\; 3.2.11}[swap]{\cong}                                                                                 &  &  & {\ind^{M^+}_{K_M}\mathrm{R}H^0(K_U, \omega_{K_P}\otimes_k -)} \arrow[uu, swap,Rightarrow]{ur}{\tilde{\gamma_1}}[swap]{\cong} \arrow[lll,swap,Rightarrow]{ur}{\mathrm{Corollary}\; 3.1.8}[swap]{\cong}
\end{tikzcd}
\end{center}
\end{lem}
\begin{proof}
From the commutative diagram in the proof of \cite[Proposition 3.4.18]{heyer2023left}, we have the commutativity of the following (outer) diagram (we take $X=\mathbf{1}$ in Heyer's diagram):
\begin{center}
\begin{tikzcd}
\Res^{P^+}_{K_P}(\omega_{P^+}\otimes_k-)\Inf^{P^+}_{M^+} \arrow[rr, swap,Rightarrow]{ur}{\cong}[swap]{ \mathrm{Prop}\; 3.4.18}                                                     &  & \Res^{P^+}_{K_P}F^{P^+}_{M^+} \arrow[dd,swap,Rightarrow]{ur}{ \cong \tilde{\beta}}[swap]{\mathrm{Lemma}\; 3.4.13} \\
(\omega_{K_P}\otimes_k-)\Res^{P^+}_{K_P}\Inf^{P^+}_{M^+} \arrow[d,swap, Rightarrow]{ur}{\cong}[swap]{\delta_1} \arrow[u, "\cong", Rightarrow] \arrow[rru, swap,Rightarrow]{ur}{\beta_1}[swap]{\cong} \arrow[rrd, "\cong \delta", Rightarrow] &  &                                                                                                     \\
(\omega_{K_P}\otimes_k-)\Inf^{K_P}_{K_M}\Res^{M^+}_{K_M} \arrow[rr,swap,Rightarrow]{ur}{\mathrm{Corollary}\; 3.1.7} [swap]{\cong \tilde{\delta}}                                  &  & F^{K_P}_{K_M}\Res^{M^+}_{K_M}                                      \end{tikzcd}
\end{center}
Now, the left mate of the natural transformation $\tilde{\beta}:\Res^{P^+}_{K_P}F^{P^+}_{M^+}\xRightarrow[]{\cong} F^{K_P}_{K_M}\Res^{M^+}_{K_M}$ (with $F=G=\Id$) is 
$$\gamma_1: \ind^{M^+}_{K_M}\mathrm{R}H^0(K_U,-)\Rightarrow \mathrm{R}H^0(K_U,-)\ind^{P^+}_{K_P}.$$
Thus, using $\delta=\tilde{\beta}\circ \beta_1$ and \cite[Example 2.1.12(ii)]{heyer2023left}, we get that the left mates of $\delta$ and $\beta_1$ (denoted by $L(\delta)$ and $L(\beta_1)$), taken with $F=\omega_{K_P}\otimes_k-$ and $G=\Id$, are related by the following commutative diagram:
\begin{center}
\begin{tikzcd}
{\mathrm{R}H^0(K_U,-)\ind^{P^+}_{K_P}(\omega_{K_P}\otimes_k-)} \arrow[rr, "L(\beta_1)", Rightarrow]                                                              &  & L_{K_U}\ind^{P^+}_{K_P}                                          \\
{\ind^{M^+}_{K_M}\mathrm{R}H^0(K_U,-)(\omega_{K_P}\otimes_k-)} \arrow[u,swap,Rightarrow]{ur}{\cong}[swap]{ \gamma_1(\omega_{K_P}\otimes_k-)} \arrow[rr, "L(\delta)", Rightarrow] &  & L_{K_U}\ind^{P^+}_{K_P} \arrow[u,swap,Rightarrow]{ur}{\Id}[swap]{\cong}
\end{tikzcd}
\end{center}
By definition, $\tilde{\gamma_1}: \ind^{M^+}_{K_M}\mathrm{R}H^0(K_U,-)(\omega_{K_P}\otimes_k-)\Rightarrow \mathrm{R}H^0(K_U,-)(\omega_{P^+}\otimes_k-)\ind^{P^+}_{K_P}$ is the composition 
\begin{multline*}\ind^{M^+}_{K_M}\mathrm{R}H^0(K_U,-)(\omega_{K_P}\otimes_k-)\xRightarrow[]{\gamma_1(\omega_{K_P}\otimes_k-)}\mathrm{R}H^0(K_U,-)\ind^{P^+}_{K_P}(\omega_{K_P}\otimes_k-)\\
\xRightarrow[]{\cong}\mathrm{R}H^0(K_U,-)(\omega_{P^+}\otimes_k-)\ind^{P^+}_{K_P},
\end{multline*}
hence we get (using Lemma \ref{diag0}) that the following commutes:
\begin{equation}
\label{diagram1}
\begin{tikzcd}
{\ind^{M^+}_{K_M}\mathrm{R}H^0(K_U,-)(\omega_{K_P}\otimes_k-)} \arrow[rr, swap,Rightarrow]{ur}{\cong}[swap]{\tilde{\gamma_1}} \arrow[rrd,swap,Rightarrow]{ur}{L(\delta)}[swap]{\cong} &  & {\mathrm{R}H^0(K_U,-)(\omega_{P^+}\otimes_k-)\ind^{P^+}_{K_P}} \arrow[d,swap,Rightarrow]{ur}{\cong}[swap]{\mathrm{Corollary}\; 3.4.20} \\
                                                                                                                                                  &  & L_{K_U}\ind^{P^+}_{K_P}                                                                                       
\end{tikzcd}
\end{equation}
Now, the left mate of the natural transformation $\Res^{P^+}_{K_P}\Inf^{P^+}_{M^+}\xRightarrow[]{\cong}\Inf^{K_P}_{K_M}\Res^{M^+}_{K_M}$ (with $F=G=\Id$) is
$$\ind^{M^+}_{K_M}L_{K_U}\xRightarrow[]{\cong}L_{K_U}\ind^{P^+}_{K_P}$$
We note that the above is precisely the isomorphism induced by \cite[Proposition 3.2.11]{heyer2023left}.
Thus, using $\delta=\tilde{\delta}\circ \delta_1$ and \cite[Example 2.1.12(i)]{heyer2023left}, we get that the left mates of $\delta$ and $\tilde{\delta}$ (denoted by $L(\delta)$ and $L(\tilde{\delta})$), taken with $F=\omega_{K_P}\otimes_k-$ and $G=\Id$, are related by the following commutative diagram:
\begin{center}
\begin{tikzcd}
{\ind^{M^+}_{K_M}\mathrm{R}H^0(K_U,-)(\omega_{K_P}\otimes_k-)} \arrow[rr, " L(\tilde{\delta})", Rightarrow] \arrow[d,swap,Rightarrow]{ur}{\mathrm{Id}}[swap]{\cong} &  & \ind^{M^+}_{K_M}L_{K_U} \arrow[d,swap,Rightarrow]{ur}{\cong}[swap]{\mathrm{Proposition}\; 3.2.11} \\
{\ind^{M^+}_{K_M}\mathrm{R}H^0(K_U,-)(\omega_{K_P}\otimes_k-)} \arrow[rr, "L(\delta)", Rightarrow]                                                   &  & L_{K_U}\ind^{P^+}_{K_P}                                
\end{tikzcd}
\end{center}
We  now claim that $L(\tilde{\delta})$ is induced from the left mate of the isomorphism in \cite[Corollary 3.1.7]{heyer2023left}, i.e. the isomorphism in \cite[Corollary 3.1.8]{heyer2023left}. This follows from Lemma \ref{mates} by taking $F=\omega_{K_P}\otimes_k-$, $G=\mathrm{Id}$, $H=\mathrm{Id}$, $L_1=L_{K_U}$, $R_1=\Inf^{K_P}_{K_M}$, $L_1'=\mathrm{R}H^0(K_U,-)$, $R_1'=F^{K_P}_{K_M}$, $L_2=L_2'=\ind^{M^+}_{K_M}$, $R_2=R_2'=\Res^{M^+}_{K_M}$ and appropriate choices of the relevant categories. Thus, we get that $L(\delta)$ fits into the following commutative diagram:
\begin{equation}
\label{diagram2}
\begin{tikzcd}
{\ind^{M^+}_{K_M}\mathrm{R}H^0(K_U,-)(\omega_{K_P}\otimes_k-)} \arrow[rr,swap,Rightarrow]{ur}{\cong}[swap]{\mathrm{Corollary}\; 3.1.8} \arrow[rrd,swap,Rightarrow]{ur}{L(\delta)}[swap]{\cong} &  & \ind^{M^+}_{K_M}L_{K_U} \arrow[d,swap,Rightarrow]{ur}{\cong}[swap]{\mathrm{Proposition}\; 3.2.11} \\
                                                                                                                                                     &  & L_{K_U}\ind^{P^+}_{K_P}                                                           
\end{tikzcd}
\end{equation}
Combining diagrams \ref{diagram1} and \ref{diagram2}, we get the required commutative diagram. 
\end{proof}

\begin{lem}
\label{diag2}
The following diagram of functors $D(K_P)\rightarrow D(M)$ commutes (i.e., going anticlockwise around the square below yields the identity functor):
\begin{center}
\begin{tikzcd}
\ind^M_{M^+}L_{K_U}\Res^P_{P^+}\ind^P_{K_P}  \arrow[d,swap, Rightarrow]{ur}{\cong}[swap]{\mathrm{Proposition}\; 3.4.6} &  & \ind^M_{M^+}L_{K_U}\Res^P_{P^+}\ind^P_{P^+}\ind^{P^+}_{K_P} \arrow[ll, "\cong", Rightarrow] \\
L_U\ind^P_{K_P} \arrow[d, swap,Rightarrow]{ur}{\cong}[swap]{ \mathrm{Proposition}\; 3.2.11}                            &  & \ind^M_{M^+}L_{K_U}\ind^{P^+}_{K_P} \arrow[u,"\eta"', Rightarrow]                                    \\
\ind^M_{K_M}L_{K_U} \arrow[rr, "\cong"', Rightarrow]                                                  &  & \ind^M_{M^+}\ind^{M^+}_{K_M}L_{K_U} \arrow[u, swap,Rightarrow]{ur}{\mathrm{Proposition}\; 3.2.11}[swap]{\cong}    \end{tikzcd}
\end{center}
\end{lem}
\begin{proof}
By \cite[Proposition 3.4.6]{heyer2023left}, the isomorphism $\ind^{M}_{M^+}L_{K_U}\Res^P_{P^+}\ind^{P}_{K_P}\xRightarrow[]{\cong}L_{U}\ind^P_{K_P}$ is obtained by taking left adjoints of the isomorphism
$$S_\eta:\Res^P_{K_P}\Inf^P_{M}\xRightarrow[\eta]{\cong}\Res^P_{K_P}\RInd^P_{P^+}\Res^P_{P^+}\Inf^P_{M}\xRightarrow[]{\cong}\Res^P_{K_P}\RInd^P_{P^+}\Inf^{P^+}_{M^+}\Res^M_{M^+}$$
where $\eta:\Id_{D(P)}\xRightarrow[]{\cong}\RInd^P_{P^+}\Res^P_{P^+}$ in the first morphism is an isomorphism by \cite[Lemma 3.4.5 (i)]{heyer2023left}. We then have that the following diagram commutes:
$$\begin{tikzcd}[cramped, column sep=small]
	{\Res^{P^+}_{K_P}\Res^P_{P^+}\Inf^P_M} & {\Res^P_{K_P}\Inf^P_M} & {\Res^P_{K_P}\RInd^P_{P^+}\Inf^{P^+}_{M^+}\Res^M_{M^+}} \\
	& {\Res^{P^+}_{K_P}\Res^P_{P^+}\RInd^P_{P^+}\Res^P_{P^+}\Inf^P_M} & {\Res^{P^+}_{K_P}\Res^P_{P^+}\RInd^P_{P^+}\Inf^{P^+}_{M^+}\Res^M_{M^+}} \\
	{\Res^P_{K_P}\Inf^P_M} & {\Res^{P^+}_{K_P}\Res^P_{P^+}\Inf^P_M} & {\Res^{P^+}_{K_P}\Inf^{P^+}_{M^+}\Res^M_{M^+}} \\
	& {\Inf^{K_P}_{K_M}\Res^{M}_{K_M}} & {\Inf^{K_P}_{K_M}\Res^{M^+}_{K_M}\Res^M_{M^+}}
	\arrow[swap, Rightarrow, from=1-1, to=2-2]{ur}{\eta}[swap]{\cong}
	\arrow["{\mathrm{Id}}"', curve={height=12pt}, Rightarrow, from=1-1, to=3-2]
	\arrow["\cong"', Rightarrow, from=1-2, to=1-1]
	\arrow[swap, Rightarrow,  from=1-2, to=1-3]{ur}{S_\eta}[swap]{\cong}
	\arrow["\cong", Rightarrow,  from=1-3, to=2-3]
	\arrow["\cong", Rightarrow,  from=2-2, to=2-3]
	\arrow["\epsilon", Rightarrow,  from=2-2, to=3-2]
	\arrow["\epsilon", Rightarrow,  from=2-3, to=3-3]
	\arrow["\cong"', Rightarrow,  from=3-1, to=3-2]
	\arrow["\cong", Rightarrow,  from=3-2, to=3-3]
	\arrow["\cong", Rightarrow,  from=3-3, to=4-3]
	\arrow["\cong", Rightarrow,  from=4-2, to=3-1]
	\arrow["\cong"', Rightarrow,  from=4-3, to=4-2]
\end{tikzcd}$$
In the above diagram, $\epsilon:\Res^P_{P^+}\RInd^P_{P^+}\Rightarrow \mathrm{Id}$ is the counit of the adjunction. Indeed, the top pentagon, central square and bottom pentagon in the above diagram are easily seen to commute. The left vertical composition of $\eta$ and $\epsilon$ is the identity morphism by the triangle identities for the unit/counit of the adjunction $\Res^P_{P^+}\dashv \RInd^P_{P^+}$. Tracing through the above diagram, we see that going around the square below clockwise starting from $\Res^P_{K_P}\Inf^P_M$ gives the identity functor.
\begin{equation*}
\begin{tikzcd}
	{\Res^P_{K_P}\RInd^P_{P^+}\Inf^{P^+}_{M^+}\Res^M_{M^+}} && {\Res^{P^+}_{K_P}\Res^P_{P^+}\RInd^P_{P^+}\Inf^{P^+}_{M^+}\Res^M_{M^+}} \\
	{\Res^P_{K_P}\Inf^P_M} && {\Res^{P^+}_{K_P}\Inf^{P^+}_{M^+}\Res^M_{M^+}} \\
	{\Inf^{K_P}_{K_M}\Res^{M}_{K_M}} && {\Inf^{K_P}_{K_M}\Res^{M^+}_{K_M}\Res^M_{M^+}}
	\arrow["\cong", Rightarrow, from=1-1, to=1-3]
	\arrow["\epsilon", Rightarrow, from=1-3, to=2-3]
	\arrow[swap, Rightarrow, from=2-1, to=1-1]{ur}{S_\eta}[swap]{\cong}
	\arrow[swap, Rightarrow, from=2-3, to=3-3]{ur}{\cong}[swap]{\mathrm{Proposition}\; 3.2.11}
	\arrow[swap, Rightarrow, from=3-1, to=2-1]{ur}{\mathrm{Proposition}\; 3.2.11}[swap]{\cong}
	\arrow["\cong", Rightarrow, from=3-3, to=3-1]
\end{tikzcd}
\end{equation*}
Observe that the left adjoint of the counit of an adjunction is the unit of the adjunction. Taking left adjoint of each functor in the above square then gives the required commutative square. 
\end{proof}
We now combine the above lemmas to prove Proposition \ref{diag}.

\begin{proof}[Proof of Proposition \ref{diag}]
From the isomorphism in \cite[Corollary 3.4.20]{heyer2023left} and the morphism $\Id\Rightarrow \Res^P_{P^+}\ind^P_{P^+}$, we get a commutative diagram
\begin{center}
\begin{tikzcd}
\ind^M_{M^+}L_{K_U}\Res^P_{P^+}\ind^P_{P^+}\ind^{P^+}_{K_P}  &  &  & \ind^M_{M^+}\mathrm{R}H^0(K_U, \omega_{P^+}\otimes_k -)\Res^P_{P^+}\ind^P_{P^+}\ind^{P^+}_{K_P} \arrow[lll, swap, Rightarrow]{ur}{\mathrm{Corollary}\;   3.4.20}[swap]{\cong}  
                \\                   
                          &  &  &                                                                                                                                                             \\
\ind^{M}_{M^+}L_{K_U}\ind^{P^+}_{K_P} \arrow[uu,swap, Rightarrow]{ur}{}[swap]{\cong}                                                                                 &  &  & {\ind^{M}_{M^+}\mathrm{R}H^0(K_U, \omega_{P^+}\otimes_k -)\ind^{P^+}_{K_P}} \arrow[uu, swap,Rightarrow]{ur}{\tilde{\gamma_2}}[swap]{\cong} \arrow[lll,swap,Rightarrow]{ur}{\mathrm{Corollary}\; 3.4.20}[swap]{\cong}
\end{tikzcd}
\end{center}
Now, applying $\ind^M_{M^+}$ to the commutative diagram in Lemma \ref{diag1} and composing with the above commutative diagram yields the required commutative diagram (using Lemma \ref{diag2}).
\end{proof}

\begin{lem}
We have a natural isomorphism of functors $D(P^+)\rightarrow D(M^+)$
$$(\Res^{P^+}_{M^+}\omega_{P^+}\otimes_k-)\mathrm{R}H^0(K_U,-)\xRightarrow[]{\cong} \mathrm{R}H^0(K_U,-)(\omega_{P^+}\otimes_{k}-).$$ Analogously, we have a natural isomorphism of functors $D(K_P)\rightarrow D(K_M)$
$$(\Res^{K_P}_{K_M}\omega_{K_P}\otimes_k-)\mathrm{R}H^0(K_U,-)\xRightarrow[]{\cong}\mathrm{R}H^0(K_U,-)(\omega_{K_P}\otimes_k -).$$
\end{lem}
\begin{proof} 
By \cite[Lemma A.3]{hoff2025ordinary}, $\omega_{P^+}^{-1}$ exists
even though $P^+$ need not be a group. We also have $\omega_{P^+}=\Res^{P}_{P^+}\omega_P$. Moreover, as $\omega_P=\delta_P[\dim U]$, we have $\omega_{P^+}=\Inf^{P^+}_{M^+}\Res^{P^+}_{M^+}\omega_{P^+}$ (since $\delta_P$ is necessarily trivial on $U$, we know $\omega_{P^+}$ is trivial on $K_U$). The same is true for $\omega_{K_P}$. Thus, the statement for $\omega_{P^+}$ follows from the isomorphism constructed in the proof of \cite[Proposition 3.4.18]{heyer2023left}
$$\mathrm{R}H^0(K_U,-)\circ (- \otimes_k \Inf^{P^+}_{M^+}Y) \xLeftarrow[]{\cong}(-\otimes_k Y)\circ \mathrm{R}H^0(K_U, -)$$ by taking $Y=\Res^{P^+}_{M^+}\omega_{P^+}$. The statement about $\omega_{K_P}$ can be obtained by passing to the right adjoints in the following:
$$(\omega_{K_P}^{-1}\otimes_k-)\Inf^{K_P}_{K_M}\xRightarrow[]{\cong}\Inf^{K_P}_{K_M}(\Res^{K_P}_{K_M}\omega_{K_P}^{-1}\otimes_k -)\xRightarrow[]{\cong}\Inf^{K_P}_{K_M}((\Res^{K_P}_{K_M}\omega_{K_P})^{-1}\otimes_k -).$$ 
\end{proof}
\begin{remark} We note some observations:
\begin{itemize}
\item Heyer notes in \cite[Warning 3.4.10]{heyer2023left} that $\mathrm{R}H^0(K_U, -): D(P^+)\rightarrow D(M^+)$ is not the right adjoint of $\Inf^{P^+}_{M^+}$, so we cannot use the compatibility of $\omega_{P^+}\otimes_k -$ with $\Inf^{P^+}_{M^+}$ and pass to right adjoints to obtain the statement for $\mathrm{R}H^0(K_U, -)$ (unlike what we did for $\omega_{K_P}$).
\item Replacing $\omega_{P^+}$ by $\omega_{P^+}^{-1}$ in the above lemma and passing to right adjoints, we get a natural isomorphism
$$(\omega_{P^+}\otimes_k-)F^{P^+}_{M^+}\xRightarrow[]{\cong} F^{P^+}_{M^+}(\Res^{P^+}_{M^+}\omega_{P^+}\otimes_k-)$$ of functors $D(M^+)\to D(P^+)$. We also obtain an analogous result for $\omega_{K_P}$ and $F^{K_P}_{K_M}$.
\item The equality $\omega_{P^+}=\Res^P_{P^+}\omega_P$ implies an isomorphism $$\Res^P_{P^+}(\omega_{P}^{-1}\otimes_k-)\xRightarrow[]{\cong}(\omega_{P^+}^{-1}\otimes_k-)\Res^P_{P^+}$$ (we noted this earlier to get a statement about $\ind^P_{P^+}$ by passing to left adjoints). Passing to right adjoints then yields a natural isomorphism
$$\RInd^P_{P^+}(\omega_{P^+}\otimes_k-)\xRightarrow[]{\cong}(\omega_P\otimes_k -)\RInd^P_{P^+}.$$
\end{itemize}
\end{remark}

\begin{lem}
\label{epslem}
The following diagram of functors $D(P^+)\rightarrow D(P^+)$ commutes:
\begin{center}
\begin{tikzcd}
	{\Res^P_{P^+}\RInd^P_{P^+}(\omega_{P^+}\otimes_k-)} && {(\omega_{P^+}\otimes_k-)\Res^P_{P^+}\RInd^P_{P^+}} \\
	\\
	& {(\omega_{P^+}\otimes_k-)}
	\arrow["\cong", Rightarrow, from=1-1, to=1-3]
	\arrow["{\epsilon \circ \mathrm{Id}}"', Rightarrow, from=1-1, to=3-2]
	\arrow["{\mathrm{Id}\circ \epsilon}", Rightarrow, from=1-3, to=3-2]
\end{tikzcd}
\end{center}
\end{lem}
\begin{proof}
This immediately follows by using counit compatibility in Lemma \ref{mate_unit_compatibility} with $L'=L=\Res^{P}_{P^+}$, $R'=R=\RInd^{P}_{P^+}$, $F=\omega_{P}\otimes_k -$, $G=\omega_{P^+}\otimes_k -$ and appropriate choices of the relevant categories.
\end{proof}

\begin{lem}
\label{ResFlem}
The following diagram of functors $D(M^+)\rightarrow D(K_P)$ commutes:
\begin{center}
\begin{tikzcd}
	{\Res^{P^+}_{K_P}F^{P^+}_{M^+}(\Res^{P^+}_{M^+}(\omega_{P^+})\otimes_k-)} && {(\omega_{K_P}\otimes_k-)\Res^{P^+}_{K_P}F^{P^+}_{M^+}} \\
	\\
	{F^{K_P}_{K_M}\Res^{M^+}_{K_M}(\Res^{P^+}_{M^+}(\omega_{P^+})\otimes_k-)} && {(\omega_{K_P}\otimes_k-)F^{K_P}_{K_M}\Res^{M^+}_{K_M}}
	\arrow["\cong"', Rightarrow, from=1-1, to=1-3]
	\arrow[swap, Rightarrow, from=1-1, to=3-1]{ur}
    {\mathrm{Lemma}\; 3.4.13}[swap]{\cong}
	\arrow[swap, Rightarrow, from=1-3, to=3-3]{ur}{\cong}[swap]{\mathrm{Lemma}\; 3.4.13}
	\arrow["\cong", Rightarrow, from=3-1, to=3-3]
\end{tikzcd}
\end{center}
\end{lem}
\begin{proof}
Say $Y=\Res^{P^+}_{M^+}\omega_{P^+}$, and let $X\in D(M^+)$ be arbitrary. Note that 
$$\Inf^{P^+}_{M^+}(Y)=\Inf^{P^+}_{M^+}\Res^{P^+}_{M^+}(\omega_{P^+})=\omega_{P^+}$$
as $\omega_{P^+}$ is trivial on $K_U$. Also, we have 
$$\Inf^{K_P}_{K_M}\Res^{P^+}_{K_M}(\omega_{P^+})=\Inf^{K_P}_{K_M}\Res^{K_P}_{K_M}(\omega_{K_P})=\omega_{K_P}.$$
Plugging this value of $Y$ in the commutative diagram in the proof of \cite[Proposition 3.4.18]{heyer2023left} we get that the following commutes:
\begin{center}
\begin{tikzcd}[cramped, column sep=small]
	{\omega_{K_P}\otimes_k \Res^{P^+}_{K_P}F^{P^+}_{M^+}(X)\cong \Res^{P^+}_{K_P}(\omega_{P^+}\otimes_k F^{P^+}_{M^+}(X))} & {\Res^{P^+}_{K_P}(F^{P^+}_{M^+}(\Res^{P^+}_{M^+}(\omega_{P^+})\otimes_k X))} \\
	\\
	{\omega_{K_P}\otimes_k F^{K_P}_{K_M}\Res^{M^+}_{K_M}(X)} & {F^{K_P}_{K_M}(\Res^{M^+}_{K_M}(\Res^{P^+}_{M^+}(\omega_{P^+}))\otimes_k \Res^{M^+}_{K_M}(X))}
	\arrow["\cong"', from=1-1, to=1-2]
	\arrow["{\cong \mathrm{Lemma}\; 3.4.13}"', from=1-1, to=3-1]
	\arrow["{\cong \mathrm{Lemma}\; 3.4.13}", from=1-2, to=3-2]
	\arrow["\cong", from=3-1, to=3-2]
\end{tikzcd}
\end{center}
Since $X\in D(M^+)$ is arbitrary in the above diagram, we get the required commutative diagram. 
\end{proof}

\begin{cor}
\label{cor1}
The following diagram of functors $D(M)\rightarrow D(K_P)$ commutes:
\begin{center}
\begin{tikzcd}
	{\Res^P_{K_P}\RInd^P_{P^+}F^{P^+}_{M^+}\Res^M_{M^+}(\Res^P_M\omega_P^{-1}\otimes_k -)} && {(\omega_{K_P}^{-1}\otimes_k -)\Res^P_{K_P}\RInd^P_{P^+}F^{P^+}_{M^+}\Res^M_{M^+}} \\
	\\
	{F^{K_P}_{K_M}\Res^M_{K_M}(\Res^P_M\omega_P^{-1}\otimes_k-)} && {(\omega_{K_P}^{-1}\otimes_k -)F^{K_P}_{K_M}\Res^M_{K_M}}
	\arrow["{ R(\Gamma)\circ \mathrm{Id}}"', Rightarrow, from=1-1, to=3-1]
	\arrow["\cong", Rightarrow, from=1-3, to=1-1]
	\arrow["{\mathrm{Id}\circ R(\Gamma)}", Rightarrow, from=1-3, to=3-3]
	\arrow["\cong"', Rightarrow, from=3-3, to=3-1]
\end{tikzcd}
\end{center}
\end{cor}
\begin{proof}
By the description of $R(\Gamma)$ in Remark \ref{rem1}, we essentially only need to show the commutativity of $\omega\otimes_k-$ with the natural transformations 
$$\epsilon: \Res^P_{P^+}\RInd^P_{P^+}\Rightarrow \mathrm{Id} \qquad\text{and}\qquad \Res^{P^+}_{K_P}F^{P^+}_{M^+}\xRightarrow[]{\cong} F^{K_P}_{K_M}\Res^{M^+}_{K_M}.$$ 
The former is done in Lemma \ref{epslem} and the latter in Lemma \ref{ResFlem}.
\end{proof}
\begin{cor}
\label{cor2}
The following diagram of functors $D(K_P)\rightarrow D(M)$ commutes:
\begin{center}
\begin{tikzcd}
	{(\Res^P_M\omega_P\otimes_k-)\ind^M_{K_M}\mathrm{R}H^0(K_U,-)} && {(\Res^P_M\omega_P\otimes_k-)\ind^M_{M^+}\mathrm{R}H^0(K_U,-)\Res^P_{P^+}\ind^{P}_{K_P}} \\
	\\
	{\ind^M_{K_M}\mathrm{R}H^0(K_U,-)(\omega_{K_P}\otimes_k-)} && {\ind^M_{M^+}\mathrm{R}H^0(K_U,-)\Res^P_{P^+}\ind^{P}_{K_P}(\omega_{K_P}\otimes_k-)} \\
	\\
	&& {\ind^M_{M^+}\mathrm{R}H^0(K_U,\omega_{P^+}\otimes_{k}-)\Res^P_{P^+}\ind^P_{K_P}}
	\arrow["{ \mathrm{Id}\circ \Gamma}", Rightarrow, from=1-1, to=1-3]
	\arrow["\cong", Rightarrow, from=1-1, to=3-1]
	\arrow["\cong"', Rightarrow, from=1-3, to=3-3]
	\arrow["{\Gamma\circ \mathrm{Id}}"', Rightarrow, from=3-1, to=3-3]
	\arrow["{\tilde{\Gamma}}"', Rightarrow, from=3-1, to=5-3]
	\arrow["\cong"', Rightarrow, from=5-3, to=3-3]
\end{tikzcd}
\end{center}
\end{cor}
\begin{proof}
The topmost square is obtained by passing to left adjoints in Corollary \ref{cor1} and hence is commutative. The fact that the bottom triangle commutes is clear from the definition of $\tilde{\Gamma}$.
\end{proof}

We can now prove the required lemma:
\begin{proof}[Proof of Lemma \ref{lem-kappa&satake}]
We first observe that by Corollary \ref{cor2} the following diagram commutes: 
\begin{center}
\begin{tikzcd}
	{\ind^M_{M^+}\mathrm{R}H^0(K_U,\omega_{P^+}\otimes_k-)\Res^P_{P^+}\ind^P_{K_P}} & {(\Res^P_M\omega_P\otimes_k-)\ind^M_{M^+}\mathrm{R}H^0(K_U,-)\Res^P_{P^+}\ind^P_{K_P}} \\
	\\
	{\ind^M_{K_M}\mathrm{R}H^0(K_U,\omega_{K_P}\otimes_k-)} & {(\Res^P_M\omega_P\otimes_k-)\ind^M_{K_M}\mathrm{R}H^0(K_U,-)}
	\arrow["\cong", Rightarrow, from=1-1, to=1-2]
	\arrow[swap, Rightarrow, from=3-1, to=1-1]{ur}{\cong}[swap]{\tilde{\Gamma}}
	\arrow["\cong"', Rightarrow, from=3-1, to=3-2]
	\arrow[swap, Rightarrow, from=3-2, to=1-2]{ur}{\Id\circ \Gamma}[swap]{\cong}
\end{tikzcd}
\end{center}
Combining the above diagram with that of Proposition \ref{diag} and applying each functor on $V \in \Rep_k(K_P)$ (more precisely, on $\Res^K_{K_P}V$) we get
\begin{center}
\begin{tikzcd}
	{\ind^M_{M^+}L_{K_U}(\Res^P_{P^+}\ind^P_{K_P}V)} && {\Res^P_M\omega_P\otimes_k\ind^M_{M^+}\mathrm{R}H^0(K_U,\Res^P_{P^+}\ind^P_{K_P}V)} \\
	{L_U(\ind^P_{K_P}V)} \\
	{\ind^M_{K_M}L_{K_U}(V)} && {\Res^P_M\omega_P\otimes_k\ind^M_{K_M}\mathrm{R}H^0(K_U,V)}
	\arrow["\cong", from=1-1, to=1-3]
	\arrow["\cong", from=2-1, to=1-1]
	\arrow["\cong"', from=2-1, to=3-1]
	\arrow["\cong"', from=3-1, to=3-3]
	\arrow[swap, from=3-3, to=1-3]{ur}{\Id\otimes_k \Gamma(V)}[swap]{\cong}
\end{tikzcd}
\end{center}
Taking cohomology in degree $i-\dim U$ and using the fact that $\omega_P=\delta_P[\dim U]$ for some smooth character $\delta_P$ of $P$ and that $H^i(\Gamma)=\kappa^{-1}$, we get
\begin{center}
\begin{tikzcd}
	{\ind^M_{M^+}L_{K_U}^{i-\dim U}(\Res^P_{P^+}\ind^P_{K_P}V)} && {\Res^P_M\delta_P\otimes_k\ind^M_{M^+}H^i(K_U,\Res^P_{P^+}\ind^P_{K_P}V)} \\
	{L_U^{i-\dim U}(\ind^P_{K_P}V)} \\
	{\ind^M_{K_M}L_{K_U}^{i-\dim U}(V)} && {\Res^P_M\delta_P\otimes_k\ind^M_{K_M}H^i(K_U,V)}
	\arrow["\cong", from=1-1, to=1-3]
	\arrow["\cong", from=2-1, to=1-1]
	\arrow["\cong"', from=2-1, to=3-1]
	\arrow["\cong"', from=3-1, to=3-3]
	\arrow[swap, from=3-3, to=1-3]{ur}{\Id\otimes_k \kappa^{-1}}[swap]{\cong}
\end{tikzcd}
\end{center}
Since the Satake morphism $\mathcal{S}^i_V(\Phi)$ is induced from the isomorphism 
$$L_U^{i-\dim U}(\ind^P_{K_P}V)\xrightarrow[\mathrm{Proposition}\; 3.2.11]{\cong}\ind^M_{K_M}L_{K_U}^{i-\dim U}(V)$$
and tensoring with a character $\Res^P_M\delta_P$ has no effect on endomorphisms, we then see that the following diagram commutes:
\begin{center}
\begin{tikzcd}
	{\ind^M_{M^+}H^i(K_U, \Res^P_{P^+}\ind^P_{K_P}V)} && {\ind^M_{K_M}H^i(K_U,V)} \\
	\\
	{\ind^M_{M^+}H^i(K_U, \Res^P_{P^+}\ind^P_{K_P}V)} && {\ind^M_{K_M}H^i(K_U,V)}
	\arrow["{\ind^M_{M^+}H^i(K_U,\Res^P_{P^+}\Phi_1)}"', from=1-1, to=3-1]
	\arrow[swap, from=1-3, to=1-1]{ur}{\cong}[swap]{\kappa^{-1}}
	\arrow["{\mathcal{S}^i_V(\Phi)}", from=1-3, to=3-3]
	\arrow[swap, from=3-3, to=3-1]{ur}{\cong}[swap]{\kappa^{-1}}
\end{tikzcd}
\end{center}
\end{proof}

\bibliographystyle{amsalpha}
\bibliography{Satake}
\end{document}